\documentclass[10pt]{amsart}
\usepackage{mathtools}
\usepackage{kpfonts}

\newtheorem{theorem}{Theorem}[section]
\newtheorem{lemma}[theorem]{Lemma}
\newtheorem{trditev}[theorem]{Proposition}
\newtheorem{posledica}[theorem]{Corollary}

\theoremstyle{definition}

\newtheorem{example}[theorem]{Example}

\newtheorem{question}[theorem]{Question}

\theoremstyle{remark}
\newtheorem{remark}[theorem]{Remark}

\numberwithin{equation}{section}

\usepackage{amssymb}
\usepackage{mathdots}
\usepackage{pst-node}
\usepackage{enumitem}

\newcommand{\C}{\mathbb{C}}

\def\dim{\mathop{\rm dim}\nolimits}
\def\codim{\mathop{\rm codim}\nolimits}

\def\rank{\mathop{\rm rank}\nolimits}

\def\Rea{\mathop{\rm Re}\nolimits}
\def\Ima{\mathop{\rm Im}\nolimits}

\def\Orb{\mathop{\rm Orb}\nolimits}
\def\Stab{\mathop{\rm Stab}\nolimits}
\def\Ker{\mathop{\rm Ker}\nolimits}

\parskip=\smallskipamount

\makeatletter
\newcommand{\doublewidetilde}[1]{{%
  \mathpalette\double@widetilde{#1}%
}}
\newcommand{\double@widetilde}[2]{%
  \sbox\z@{$\m@th#1\widetilde{#2}$}%
  \ht\z@=.9\ht\z@
  \widetilde{\box\z@}%
}
\makeatother

\makeatletter
\renewcommand*\env@matrix[1][*\c@MaxMatrixCols c]{%
  \hskip -\arraycolsep
  \let\@ifnextchar\new@ifnextchar
  \array{#1}}
\makeatother

\begin{document}
\title[]
{
The stabilizers for the action of orthogonal similarity on symmetric matrices and orthogonal $*$-conjugacy on Hermitian matrices
%On orthogonal similarity classes of symmetric matrices and orthogonal %$*$-conjugacy classes of Hermitian matrices
}
\author{
Tadej Star\v{c}i\v{c}}
\address{Faculty of Education, University of Ljubljana, Kardeljeva Plo\v{s}\v{c}ad 16, 1000 Lju\-blja\-na, Slovenia}
\address{Institute of Mathematics, Physics and Mechanics, Jadranska
  19, 1000 Ljubljana, Slovenia}
\email{tadej.starcic@pef.uni-lj.si}
%\curraddr{}
%\email{tadej.starcic@pef.uni-lj.si}
%    General info
\date{June 26, 2020}

%\dedicatory{}

\keywords{complex points, normal forms, complex orthogonal matrices, Toeplitz matrices\\
\indent Research supported by grant P1-0291 
from ARRS, Republic of Slovenia.}

\begin{abstract} 
We describe the recursive algorithmic procedure to compute the stabilizers of the group of complex orthogonal matrices with respect to the action of similarity on the set of all symmetric matrices. Futhermore, lower bounds for dimensions of the stabilizers for the action of orthogonal $*$-conjugation on Hermitian matrices are obtained. We also prove a result that completes the classification of normal forms of Hermitian matrices under orthogonal $*$-conjugation. A key step in our proof is to solve a certain block matrix equation with Toeplitz blocks. These results are then applied to provide a theorem on normal forms of the quadratic parts of flat complex points in a real codimension $2$ submanifold in a complex manifold.
\end{abstract}

\maketitle

\section{Introduction}

First we introduce the notation and recall a few basic properties for a smooth action $\Phi\colon G\times Y\to Y$ of a Lie group $G$ ($e$ is a unit) on a smooth manifold $Y$. It must satisfy the condition:
\[
\Phi(e,y)=y, \qquad \Phi(g,\Phi(h,y))=\Phi(g h,y),\qquad g,h\in G,\quad y\in Y.
\]
The following facts are then easily deduced (check \cite[Problem 12, Theorem 1]{OV}):
\begin{itemize}  
\item For any $g\in G$ the maps $\Phi^g \colon Y\to Y$, $ y \mapsto \Phi(g,y)$ and $R_g\colon G\to G$, $h\mapsto  hg$ are diffeomorphisms.
\item For any $y\in Y$ the orbit map $\Phi_y \colon G\to Y$, $ g \mapsto \Phi(g,y)$ is smooth and equivariant (for any $g\in G$ we have $\Phi_y(R_g (h))=\Phi^g(\Phi_y(h))$, $h\in G$). Moreover, $\Phi_y$ is of constant rank ($d_{g}\Phi_y=d_{y}\Phi^g\circ d_e\Phi_y\circ (d_{e}R_g)^{-1}$).
\item An \emph{orbit} of $y$, denoted by $\Orb_{\Phi}(y)=\{\Phi(g,y)\mid g\in G\}$,
is an immersed homogeneous submanifold of dimension equal to $\rank (\Phi_y)$. 
\item The \emph{stabilizer} 
\[
\Stab_{\Phi}(y)=\{g\in G\mid \Phi(g,y)=y\}=(\phi_y)^{-1}(y)
\]
is a closed Lie subgroup in $G$ of codimension equal to $\rank (\Phi_y)$ and with the tangent space  $T_e(\Stab_{\Phi}(y))=\Ker (d_e \phi_y)$. We have 
\[
\codim (\Stab_{\Phi}(y))=\dim (\Orb_{\Phi}(y))=\rank (\Phi_y).
\]
\item $\Phi_g(y)=y'$ if and only if $g\Stab_{\Phi}(y)g^{-1}=\Stab_{\Phi}(y')$. (Stabilizers of elements in the same orbit are isomorphic.)
\end{itemize}
%Furthermore, $\Orb_{\Phi}(y)$ can be endowed with the structure of a smooth manifold but with not necessarily a subspace topology (see e.g. \cite[Theorem 9.6]{Boot}, \cite[Theorem 3.2]{Helgason}). 
%
%When no confusion is possible we denote the stabilizer and the orbit %for $y\in Y$ briefly by $\Orb(y)$ and $\Stab(y)$, respectively.

The purpose of this paper is to give a better understanding of the stratification of certain classes of complex (real) square matrices with respect to certain actions of the complex (real) orthogonal group. An important information in this direction are dimensions of orbits. If the group acting are all invertible matrices, these can be obtained directly by computing tangent spaces of orbits. In the case of the similarity action the codimension of the tangent space is simply the dimension of the vector space of solutions of certain Sylvester's equation (see e.g. \cite{Arnold}), while the case of $*$-conjugation and $T$-conjugation is somewhat more involved (see \cite{TeranDopi1}, \cite{TeranDopi2}). However, when considering actions of the orthogonal group the calculation of tangent spaces might be very intrigueing. A natural way to obtain the dimension of an orbit is then to compute the stabilizer of an action; see Sec. \ref{cereq} for this approach. 
In any case it is essential to know the representatives of orbits (normal forms), therefore we shall restrict our attention to actions on classes of matrices, such that their corresponding normal forms have already been found.  

It is well known (by Sylvester's inertia theorem) that given a real symmetric matrix $A$ there exists a real orthogonal matrix $Q$ such that $Q^{T}AQ=\Lambda$, where $\Lambda=\oplus_{j=1}^{N}\big(\oplus_{k=1}^{m} \lambda_j\big)$, and $\lambda_1,\ldots,\lambda_N$ are pairwise distinct eigenvalues of $A$. (A square matrix $Q$ is orthogonal if and only if $Q^{T}Q=I$ (or $Q^{T}=Q^{-1}$).) It is easy to see that the stabilizer of $A$ (hence $\Lambda$) with respect to the action or real orthogonal similarity is then the set of matrices of the form $Q=\oplus_{j=1}^{N} Q_j$, where $Q_j$ is $m\times m$ real orthogonal matrix for any $j\in \{1,\ldots,N\}$ (see Lemma \ref{lemas}).

Our aim is to examine the complex case. By $\mathbb{C}^{n\times n}$ we denote the group of all $n\times n$ complex matrices, and by $\mathbb{C}^{n\times n}_S$, $\mathbb{C}^{n\times n}_H$, $O_n(\mathbb{C})$, respectively, its subgroups of symmetric, Hermitian and orthogonal matrices. The action of orthogonal similarity on symmetric matrices is defined as
\begin{align}\label{aos}
\Phi\colon O_n(\mathbb{C})\times \mathbb{C}^{n\times n}_S\to \mathbb{C}^{n\times n}_S,\qquad (Q,A)\mapsto Q^{-1}AQ.
\end{align}
Matrices $A$, $B$ are thus in the same orbit for the action of orthogonal similarity (i.e. orthogonally similar) precisely when there exists an orthogonal matrix $Q$ such that
\begin{equation}\label{eqAQBs}
A=Q^{-1}BQ.
\end{equation}
The notion of orthogonal similarity coincides with the concept of orthogonal $T$-congruence; recall that $A$ and
$\widetilde{A}$ (not necessarily symmetric) are $T$-congruent if and only
if there exists a non-singular (not necessarily orthogonal) matrix $Q$
such that $\widetilde{A}=Q^{T} A Q$.
Remember also that two symmetric matrices are similar precisely when they are orthogonally similar (see e.g. \cite{Gant}).

Given two square matrices $A$, $B$ of possibly different dimensions, the equation (\ref{eqAQBs}) with $Q$ orthogonal is equivalent to the system of matrix equations
\begin{equation}\label{eqAoXXB1}
AX=XB, \qquad X^{T}X=I.
\end{equation}
Note that the first equation of (\ref{eqAoXXB1}) is the classical Sylvester's equation; its solutions are presented in the next section.
The set of solutions of (\ref{eqAoXXB1}) for $A=B$ is precisely the stabilizer of $A$ with respect to orthogonal similarity. Futhermore, it suffices to consider only the case when $A$ is a normal form. 
%We shall restrict our attention to the action on symmetric matrices.

Let us recall the symmetric canonical form (see e.g. \cite{HornJohn}).
%equation is a suitably chosen normal %form for consimilarity (see Hong and Horn %\cite[Theorem 3.1]{HongHorn}), which %enables one to reduce the equation to the %Sylvester's equation (see \cite{BHH}).
Given a symmetric matrix $A$ with its Jordan canonical form:
\begin{equation}\label{JFs}
%\hspace{-12mm}
\mathcal{J}(A)=\bigoplus_j J_{\alpha_j}(\lambda_j), \qquad \lambda_j\in \mathbb{C}, 
\end{equation}
where the elementary $m\times m$ Jordan block is denoted by
\begin{equation*}
 J_m(z)=\begin{bmatrix}
                                                      z    &  1       & \;     & 0    \\
						      \;     & z     & \ddots & \;    \\     
						      \;     & \;      & \ddots &  1     \\
                                                      0     & \;      & \;     & z   
                                   \end{bmatrix},\qquad z\in \mathbb{C},
\end{equation*}
the symmetric normal form is
\begin{equation}\label{NF1s}
\mathcal{S}(A)=\bigoplus_{j=1}^{} S_{\alpha_j}(\lambda_j),\qquad
%\end{equation}
%\begin{equation*}%\label{Sm}
S_m(z)=
\frac{1}{2}\left(
\begin{bsmallmatrix}
2z  & 1 &             & 0 \\
1 & \ddots &   \ddots        &   \\
   & \ddots  &  \ddots & 1 \\
0   &          & 1    & 2z \\
\end{bsmallmatrix}
+
i
\begin{bsmallmatrix}
0  &                &    -1  & 0\\
 &   \iddots           &   \iddots     & 1 \\
-1   &  \iddots  & \iddots &  \\
0   & 1          &     & 0 \\
\end{bsmallmatrix}
\right).
\end{equation}
Moreover,
\begin{equation}\label{NFPs}
\mathcal{S}(A)=P\mathcal{J}(A)P^{-1},
\qquad
P=\bigoplus_j P_{\alpha_j}, \quad P_{\alpha_j}=\frac{1}{\sqrt{2}}(I_{\alpha_j}+iE_{\alpha_j}).
\end{equation}
Here  $I_{\alpha}$ the $\alpha\times \alpha_j$ identity-matrix and the backward $\alpha\times \alpha$ identity-matrix (with ones on the anti-diagonal) is
$E_{\alpha}=
\begin{bsmallmatrix}
 0                 &      & 1\\
            &   \iddots     &  \\
%     1      &              &  \\
1            &           &  0\\
\end{bsmallmatrix}
$. See \cite{Djok} for a tridiagonal symmetric normal form.

Our first result is the following.

\begin{theorem}\label{stabs}
Let $\Phi$ be the action of orthogonal similarity on symmetric matrices (\ref{aos}). Suppose $A$ is a symmetric matrix and let  
%\begin{equation}\label{eqH1}
%(\mathcal{H}^1(A))\overline{X}=X(\mathcal{H}^1(A)),
%\end{equation}
%where 
$\mathcal{S}^{}(A)$ and $\mathcal{J}^{}(A)$ be of the forms (\ref{NF1s}) and (\ref{JFs}), respectively.
Let further 
$\mathcal{S}(A)=\bigoplus_{r=1}^{N}\mathcal{S}(A,\rho_r)$,
%\[
%\mathcal{H}_{}^{1}(A)=\left(\bigoplus_{j}\mathcal{H}_{}^{1}(A,\lambda_j)\right)\oplus \left(\bigoplus_k \mathcal{H}_{}^{1}(A,\mu_k)\right)\oplus %\left( \bigoplus_l \mathcal{H}_{}^{1}(A,\xi_l)\right),
%\]
where all blocks of $\mathcal{S}(A)$ corresponding to the eigenvalue $\rho_r$ with respect to $\mathcal{J}(A)$ are collected together into $\mathcal{S}^{}_{}(A,\rho_r)$. Then $\dim_{\mathbb{C}}\Stab_{\Phi} (\mathcal{S}^{}(A))=\sum_{r=1}^{N}\dim_{\mathbb{C}}\Stab_{\Phi}\mathcal{S}^{}_{}(A,\rho_r)$. 
%Then $X$ is of the form $X=\oplus_{r}X_{r}$ with $\mathcal{H}^{1}_{}(A,\rho_r)\overline{X}_{r}=X_{r}\mathcal{H}_{}^1(A,\rho_r)$. 
Moreover, if $\mathcal{S}_{}^{}(A)=\bigoplus_{r=1}^{N}\left(\bigoplus_{j=1}^{m_r} S_{\alpha_r}(\lambda)\right)$, then 
\[
\dim_{\mathbb{C}}\Stab_{\Phi} (\mathcal{S}^{}(A))=\sum_{r=1}^{N}\alpha_r m_r\big(\tfrac{1}{2}(m_r-1)+\sum_{s=1}^{r-1}m_s\big).
\]
\end{theorem}

\begin{remark}
\begin{enumerate}
\item A recursive algorithmic procedure to compute $\mathcal{S}(A)$ in Theorem \ref{stabs} can be obtained. It will be provided as part of the proof of the theorem, more precisely the proof of Lemma \ref{EqT}. The lemma is essential for the proof of Theorem \ref{stabs}. Due to technical reasons this result is stated and proved in Sec. \ref{cereq}.
\item The normal forms are known for skew-symmetric and orthogonal matrices, too. However, the corresponding matrix equations describing the stabilizers of these forms eventually lead to equations which involve an important difference in comparison to the equation that we deal in this paper (Lemma \ref{EqT}). This problem will be addressed in the subsequent paper.
\end{enumerate}
\end{remark}

We proceed with the action of orthogonal $*$-congruence on Hermitian matrices:
%The action of orthogonal similarity on symmetric matrices is defined as
%
\begin{equation}\label{azc}
\Psi\colon O_n(\mathbb{C})\times \mathbb{C}^{n\times n}_H\to \mathbb{C}^{n\times n}_H,\qquad (Q,A)\mapsto Q^{*}AQ.
\end{equation}

Hence matrices $A$ and $B$ are in the same orbit for the action of orthogonal $*$-congruence (i.e. are orthogonally $*$-congruent) if and only if there exists an orthogonal matrix $Q$ such that
\begin{equation}\label{eqAQB}
A=Q^{*}BQ.
\end{equation}
The notion of orthogonal $*$-congruence coincides with the concept of orthogonal consimilarity; recall that $A$ and
$\widetilde{A}$ (not necessarily Hermitian) are consimilar if and only
if there exists a non-singular (not necessarily orthogonal) matrix $Q$
such that $\widetilde{A}=Q^{-1}A\overline{Q}$.
Next, given two square matrices $A$ and $B$, the equation (\ref{eqAQB}) with $Q$ orthogonal is equivalent to the system of equations
\begin{equation}\label{eqAoXXB}
A\overline{X}=XB, \qquad X^{T}X=I.
\end{equation}
Solutions of (\ref{eqAoXXB}) for $A=B$ are precisely the stabilizer of $A$ with respect to the action of orthogonal $*$-conjugation. 
%Futhermore, it suffices to consider only the case when $A$ is a normal form. 
Observe that the first equation of (\ref{eqAoXXB}) is similar to Sylvester's equation. Moreover, to solve this equation it is important to choose the  appropriate normal form for consimilarity, 
%(see Hong and Horn \cite[Theorem 3.1]{HongHorn})
which then enables one to reduce the equation to Sylvester's equation.

Given a matrix $A$ let $A\overline{A}$ be similar to its Jordan canonical form (see e.g.\cite{Haantjes}):
\small
\begin{align}\label{JF}
%\qquad \qquad\qquad
&\mathcal{J}(A\overline{A})=\left(\bigoplus_j J_{\alpha_j}(\lambda_j^2)\right)\oplus \left(\bigoplus_k (J_{\beta_k}(-\mu_k^2)\oplus J_{\beta_k}(-\mu_k^2)) \right)\oplus \left(\bigoplus_l (J_{\gamma_l}(\xi_l)\oplus J_{\gamma_l}(\overline{\xi}_l)) \right),
\end{align}
\normalsize
where $\lambda_j\in \mathbb{R}_{\ge 0} $, $\mu_k\in \mathbb{R}_{>0}$, $\xi_l^{2}\in \mathbb{C}\setminus \mathbb{R}$. It yields several normal forms which are consimilar to $A$: 
\begin{itemize}
\item (Hong \cite{Hong90})
\begin{equation}\label{CF00}
%\hspace{-12mm}
\mathcal{J}_{q}(A)=\left(\bigoplus_j J_{\alpha_j}(\lambda_j)\right)\oplus\left(
\bigoplus_k 
\begin{bsmallmatrix}
0  & J_{\beta_k}(\mu_k) \\
-J_{\beta_k}(\mu_k) &  0    
\end{bsmallmatrix}\right)\oplus\left( \bigoplus_l 
\begin{bsmallmatrix}
0   & J_{\gamma_l}(\xi_l)    \\
J_{\gamma_l}(\overline{\xi}_l) &  0  
\end{bsmallmatrix}\right)
\end{equation}
%
%respectively.
%Here $\lambda_j\in \mathbb{R}_{\ge 0} $, $\mu_k\in \mathbb{R}_{>0}$, %$\xi_l^{2}\in \mathbb{C}\setminus \mathbb{R}$. 
%$\lambda_j\in \mathbb{R}_{\ge 0} $, $\mu_k\in \mathbb{R}_{>0}$, $\xi_l\in %\mathbb{C}\setminus \mathbb{R}$;
%
%\item (Bevis, Hall and Hartwig \cite{BHH}) $\mathcal{J}_{q}(A)$ is consimilar %to $(\mathcal{J}_{q}(A))^{*}$ with a symmetric matrix $S=\left(\bigoplus_j %E_{\alpha_j}\right)\oplus\left( \bigoplus_k -i(E_{\beta_k}\oplus E_{\beta_k})\right)\oplus\left( \bigoplus_l (E_{\gamma_l}\oplus E_{\gamma_l})\right)$.
%
%
\item (Hong and Horn \cite[Theorem 3.1]{HongHorn})
\begin{equation}\label{CFHH}
%\hspace{12mm}
\mathcal{J}_{q}'(A)=\left(\bigoplus_j J_{\alpha_j}(\lambda_j)\right)\oplus
\left(
\bigoplus_k 
\begin{bsmallmatrix}
0   & I_{\beta_l}    \\
J_{\beta_l}(-\mu_k^2) &  0  
\end{bsmallmatrix}
%J^q_{\beta_k}(-\mu_k^{2})
\right)
\oplus\left( \bigoplus_l 
\begin{bsmallmatrix}
0   & I_{\gamma_l}    \\
J_{\gamma_l}(\xi_l^{2}) &  0  
\end{bsmallmatrix}
%J^q_{\gamma_k}(\xi_l)
\right)
\end{equation}
%$\lambda_j\in \mathbb{R}_{\ge 0} $, $\mu_k\in \mathbb{R}_{>0}$, $\xi_l^{2}\in %\mathbb{C}\setminus \mathbb{R}$, $\Ima (\xi_l)>0$.
%i.e. there exists a non-singular matrix $T$ such that $TA\overline{T}^{-1}=
%
\item (Hong \cite[Lemma 2.1]{Hong90}) 
%
%\begin{enumerate}
%\item[{\bf FORM 1.}] $(\mathcal{H}^{\epsilon}(A),I)$,
%
\begin{equation}\label{NF1}
\mathcal{H}^{1}(A)=\left(\bigoplus_{j}^{} H_{\alpha_j}(\lambda_j)\right)\oplus\left(\bigoplus_{k}^{} K_{\beta_k}(\mu_k)\right)\oplus\left( \bigoplus_{l}^{} L_{\gamma_l}(\xi_l)\right),
\end{equation}
%
%where $\lambda_j\in \mathbb{R}_{\ge 0} $, $\mu_k\in \mathbb{R}_{>0}$, %$\xi_l\in \mathbb{C}\setminus \mathbb{R}$ are such that $\lambda_j^2$,
%$-\mu_k^2$ and $\xi_l^2$ are real non-negative, real negative and
%non-real eigenvalues of $A\overline{A}$, respectively; 
%
%
where 
\begin{align}\label{HKLmz}
&H_m(z)=
\frac{1}{2}\left(
\begin{bsmallmatrix}
0  &                &    1  & 2z\\
 &   \iddots           &   \iddots     & 1 \\
1   &  \iddots  & \iddots &  \\
2z   & 1          &     & 0 \\
\end{bsmallmatrix}
+
i
\begin{bsmallmatrix}
0  & 1 &             & 0 \\
-1 & \ddots &   \ddots        &   \\
   & \ddots  &  \ddots & 1 \\
0   &          & -1    & 0 \\
\end{bsmallmatrix}
\right),\\
&K_{m}(z)=
\begin{bsmallmatrix}
0       & -iH_m(z)    \\
iH_m(z) &  0    \\
\end{bsmallmatrix}, 
L_{m}(z)=
\begin{bsmallmatrix}
0       & H_m(z)    \\
H^{*}_m(z) &  0    \\
\end{bsmallmatrix}.\nonumber
\end{align}
%
%
%\item[{\bf FORM 2.}] $(\mathcal{J}_E^{\epsilon}(A),E(A))$,
%%
%\begin{align}\label{NF2}
%&\mathcal{J}_E^{\epsilon}(A)=\left(\bigoplus_j \epsilon_j E_{\alpha_j}J_{\alpha_j}(\lambda_j,1)\right)\oplus\left( \bigoplus_k E_{2\beta_k}J_{\beta_k}(\Lambda_k,I_2)\right),\\
%
%& E(A)=\left(\bigoplus_j E_{\alpha_j}\right)\oplus\left( \bigoplus_k E_{2\beta_k}\right),\nonumber
%\end{align}
%where $\epsilon=\{\epsilon_j\}$ with $\epsilon_j\in\{1,-1\}$, and
%$\lambda_j \in \mathbb{R}_{\ge 0}$, $\Lambda_k=\begin{bsmallmatrix}  
%	                                            a_k & -b_k\\
%	                                            b_k & a_k
%                                       \end{bsmallmatrix}$ with $a_k
%                                     \in \mathbb{R}_{\ge 0}$, $b_k\in \mathbb{R}\setminus\{0\}$ are such that $\lambda_j$, $a_k+i b_k$ are the eigenvalues of the corresponding double-sized matrix 
%$
%\begin{bsmallmatrix}
%0 & \overline{A} \\
%A & 0
%\end{bsmallmatrix}
%$;
(Note that $\mathcal{H}^{1}(A)$ is a Hermitian.)
\end{itemize}
In any case $\Ima (\xi_l)>0$ may be assumed. Note also that the blocks corresponding to the eigenvalue $0$ are uniquely determined by the so-called alternating-product rank condition \cite[Theorem 4.1]{HornJohn}.
Moreover,
\begin{equation}\label{NFP}
\mathcal{H}^{1}(A)=P^{-1}\mathcal{J}_{q}(A)\overline{P},
\end{equation}
\begin{equation}\label{P}
P=\left(\bigoplus_j P_{\alpha_j}\right)\oplus\left( \bigoplus_k e^{\frac{i\pi}{4}}(P_{\beta_k}\oplus P_{\beta_k})\right)\oplus\left(\bigoplus_l (P_{\gamma_l}\oplus P_{\gamma_l})\right), \quad P_m=\frac{e^{{\scriptscriptstyle -\frac{i\pi}{4}}}}{\sqrt{2}}(I_m+iE_m).
\end{equation}
%Here again we denoted the block backward $m$-block identity-matrix (with ones on the anti-diagonal) by
%$E_m=
%\begin{bsmallmatrix}
% 0                 &      & 1\\
%            &   \iddots     &  \\
%     1      &              &  \\
%1            &           &  0\\
%\end{bsmallmatrix}
%$.
%Note that $\mathcal{H}^{1}(A)$ is a Hermitian canonical form for $A$ and it %is unique up to a permutation of the diagonal blocks.)
Note that normal forms under cosimilarity were first developed by Haantjes  \cite{Haantjes} and Asano and Nakayama \cite{AsanoNakaya}, but normal forms given above are better suited for our application.

When $A$ is Hermitian, then by the result of Hong \cite[Theorem 2.7]{Hong89} it is consimilar with a complex orthogonal matrix to
\begin{equation}\label{NF1}
\mathcal{H}^{\epsilon}(A)=\left(\bigoplus_{j}^{}\epsilon_j H_{\alpha_j}(\lambda_j)\right)\oplus\left(\bigoplus_{k}^{} K_{\beta_k}(\mu_k)\right)\oplus\left( \bigoplus_{l}^{} L_{\gamma_l}(\xi_l)\right),
\end{equation}
%
%where $\lambda_j\in \mathbb{R}_{\ge 0} $, $\mu_k\in \mathbb{R}_{>0}$, %$\xi_l\in \mathbb{C}\setminus \mathbb{R}$ are such that $\lambda_j^2$,
%$-\mu_k^2$ and $\xi_l^2$ are real non-negative, real negative and
%non-real eigenvalues of $A\overline{A}$, respectively; 
where $\epsilon=\{\epsilon_j\in\{1,-1\}\mid \lambda_j\neq 0\,\lor\,\alpha_j \textrm{ even }\}$ witn $\epsilon_j=1$ for $\lambda_j=0$, $\alpha_j$ odd. (Trivially, $H_{2n-1}(0)$ is orthogonally congruent to $-H_{2n-1}(0)$ (see \cite[Remark 4.5]{TS}).) 
By applying this result for $A=-iB$ with a skew-Hermitian matrix $B=-B^{*}$, we immediately obtain that the skew-Hermitian canonical form for orthogonal $*$-congruence for $B$ is   
$i\mathcal{H}^{\epsilon}(B)$ (see \cite[Corollary 2.8]{Hong89}).
Note that the classification of Hermitian matrices under orthogonal similarity was first treated by Hua \cite{Hua}.

Observe that in the real case (i.e. on real symmetric matrices) the concept of orthogonal $*$-congruence coincides with orthogonal similarity (and also orthogonal consimilarity or orthogonal $T$-congruence).

The next result describes dimensions of stabilizers with respect to the action of orthogonal $*$-conjugation on Hermitian matrices. It also 
%extends \cite[Proposition 4.4]{ST}, substantialy, and completes the answer to 
%completes the classification of normal forms of Hermitian matrices under orthogonal $*$-conjugation. 
answers the question concerning uniqueness of the normal form (\ref{NF1}). 

\begin{theorem}\label{stabz}
Let $\Psi$ be the action of orthogonal $*$-conjugation on Hermitian matrices (\ref{azc}). Suppose $A$ is a square matrix and let  
%\begin{equation}\label{eqH1}
%(\mathcal{H}^1(A))\overline{X}=X(\mathcal{H}^1(A)),
%\end{equation}
%where 
$\mathcal{H}^{\epsilon}(A)$ be of the form (\ref{NF1}).
Let further 
$\mathcal{H}_{}^{1}(A)=\bigoplus_{r=1}^{N}\mathcal{H}^{1}_{}(A,\rho_r)$,
%\[
%\mathcal{H}_{}^{1}(A)=\left(\bigoplus_{j}\mathcal{H}_{}^{1}(A,\lambda_j)\right)\oplus \left(\bigoplus_k \mathcal{H}_{}^{1}(A,\mu_k)\right)\oplus %\left( \bigoplus_l \mathcal{H}_{}^{1}(A,\xi_l)\right),
%\]
where all blocks of $\mathcal{H}^{\epsilon}(A)$ corresponding to the eigenvalue $\rho_r$ with respect to $\mathcal{J}(A\overline{A})$ in (\ref{JF}) are collected together into $\mathcal{H}^{\epsilon}_{}(A,\rho_r)$. Then $\dim_{\mathbb{R}}\Stab_{\Psi} (\mathcal{H}^{\epsilon}(A))=\sum_{r=1}^{N}\dim_{\mathbb{R}}\Stab_{\Psi}\mathcal{H}^{\epsilon}_{}(A,\rho_r)$. 
%Then $X$ is of the form $X=\oplus_{r}X_{r}$ with $\mathcal{H}^{1}_{}(A,\rho_r)\overline{X}_{r}=X_{r}\mathcal{H}_{}^1(A,\rho_r)$. 
Furthermore, the following holds:
\begin{enumerate}
\item \label{stabz0} If $\mathcal{H}_{}^{\epsilon}(A)=\bigoplus_{r=1}^{N}\left(\bigoplus_{j=1}^{m_r}\epsilon_j H_{\alpha_r}(0)\right)$, then 
\[
\dim_{\mathbb{R}}\Stab_{\Psi} (\mathcal{H}^{\epsilon}(A))=\tfrac{3}{2}\sum_{\alpha_r \textrm{ even}}m_r\alpha_r-\tfrac{1}{2}\sum_{\alpha_r \textrm{ odd}}m_r(\alpha_r+1)+2\sum_{r=1}^{N}\alpha_r m_r \big(m_r+\sum_{s=1}^{r-1}2m_s\big).
\]
%$X=P^{-1}YP$, where $P=\oplus_jP_{\alpha_j}$, and $Y=[Y_{jk}]_{j,k}$ is partitioned conformally to blocks as $\mathcal{H}_{}^1(A)$ with the block $Y_{jk}$ of the %form (\ref{QT}) for $m=\alpha_j$, $n=\alpha_k$.
%We have $X=[X_{jk}]_{j,k}$ with $X_{jk}=P_{\alpha_j}^{-1}Y_{jk}P_{\alpha_k}$.
%
\item \label{stabz1} If $\mathcal{H}_{}^{\epsilon}(A)=\bigoplus_{r=1}^{N}\left(\bigoplus_{j=1}^{m_r}\epsilon_j H_{\alpha_r}(\lambda)\right)$, $\lambda > 0$, then 
\[
\dim_{\mathbb{R}}\Stab_{\Psi} (\mathcal{H}^{\epsilon}(A))=\sum_{r=1}^{N}\alpha_r m_r\big(\tfrac{1}{2}(m_r-1)+\sum_{s=1}^{r-1}m_s\big).
\]
%$X=P^{-1}YP$, where $P=\oplus_jP_{\alpha_j}$, and $Y=[Y_{jk}]_{j,k}$ is partitioned conformally to blocks as $\mathcal{H}_{}^1(A)$ with the block $Y_{jk}$ of the %form (\ref{QT}) for $m=\alpha_j$, $n=\alpha_k$.
%We have $X=[X_{jk}]_{j,k}$ with $X_{jk}=P_{\alpha_j}^{-1}Y_{jk}P_{\alpha_k}$.
%
\item \label{stabz2} If $\mathcal{H}_{}^{1}(A)=\bigoplus_{r=1}^{N}\left(\bigoplus_{k=1}^{m_r} K_{\beta_r}(\mu)\right)$, $ \mu\in \mathbb{R}_{>0}$, $n_r=2m_r$, then
\begin{equation}\label{ineq2}
\dim_{\mathbb{R}}\Stab_{\Psi} (\mathcal{H}^{\epsilon}(A))\geq \sum_{r=1}^{N}n_r\big(4\beta_r n_r-\alpha_r-2n_r+8\beta_r \sum_{s=1}^{r-1}n_s\big).
\end{equation}
\item \label{stabz3} If $\mathcal{H}_{}^{1}(A)=\bigoplus_{r=1}^{N}\left(\bigoplus_{l=1}^{m_r} L_{\gamma_r}(\xi)\right)$, $ \xi^{2}\in \mathbb{C}\setminus  \mathbb{R}$, then 
\[
\dim_{\mathbb{C}}\Stab_{\Psi} (\mathcal{H}^{\epsilon}(A))=\sum_{r=1}^{N}\gamma_r m_r\big(\tfrac{1}{2}(m_r-1)+\sum_{s=1}^{r-1}m_s\big).
\]
\end{enumerate}

Moreover, the canonical form $\mathcal{H}^{\epsilon}(A)$ with respect to the action $\Psi$  is unique up to the order of the diagonal blocks.

\end{theorem}

\begin{remark}\label{rmz}
%\hspace{-2mm}
\begin{enumerate}
\item We use the same approach (a recursive algorithm) to prove both theorems, Theorem \ref{stabs} and Theorem \ref{stabz}. However, when proving Theorem \ref{stabz} some additional intrigueing problems arise. The only technical problem we left open is to improve the inequality (\ref{ineq2}), which is acctually very close to the equality. More precisely, what remains to be found is the dimension of the set of orthogonal matrices of the form $\begin{bsmallmatrix}
A & B\\
\overline{B} & \overline{A}
\end{bsmallmatrix}$, $A,B\in \mathbb{C}^{n\times n}$.
\item An analoguous theorem holds for skew-Hermitian matrices as well. The stabilizers for $\mathcal{H}^{\epsilon}(A)$ and $i\mathcal{H}^{\epsilon}(A)$ for orthogonal $*$-conjugation on Hermitian and skew-Hermitian matrices, respectively, clearly coincide. 
\item It is known for some time that $(\mathcal{H}^{\epsilon}(A),I)$ is the generic normal form for $(*,T)$-conjugation on pairs of one Hermitian and one symmetric matrix (check \cite[Proposition 3.1]{TS} for a related action). When considering only pairs of the form $(A,I)$, Therem \ref{stabz} provides aditional information on the stratification into orbits for this action (orthogonal $(*,T)$-conjugation).
\end{enumerate}
\end{remark}

Theorem \ref{stabz} is applied to give the result on uniqeness and the dimension of the orbit of a normal form of the quadratic part of a flat complex point of a real codimension $2$ submanifold in a complex manifold.

\begin{posledica}\label{posapp} Let a real $2n$-manifold $M$ embedded $\mathcal{C}^2$-smoothly in a complex
$(n+1)$-manifold $X$ locally near a flat isolated complex point $p\in M$ be seen as a graph:
\begin{equation}\label{BasForm1}
w=\overline{z}^TAz+\Rea (z^TBz)+o(|z|^2), \quad (w(p),z(p))=(0,0),
\end{equation}
where $(z,w)=(z_1,z_2,\ldots,z_n,w)$ are suitable local coordinates on
$X$, and $A$, $B$ are Hermitian and symmetric $n\times n$ complex matrices, respectively. If in addition $B$ is nonsingular then there exists a holomorphic change of coordinates so that (\ref{BasForm1}) transforms to 
\begin{equation}
\label{BasForm2}
\widetilde{w}=\overline{\widetilde{z}}^T\mathcal{H}^{\epsilon}(A) \widetilde{z}+\Rea \left(\widetilde{z}^T \widetilde{z}\right)+o(|\widetilde{z}|^2),
\end{equation}
where $\epsilon=(\epsilon_1,\ldots,\epsilon_n)$ is unique up to the sign. Moreover, assuming $B=I$, the set of matrices $A$ such that (\ref{BasForm1}) can be transformed to (\ref{BasForm2}) is an immersed submanifold in $\mathbb{C}^{n\times n}$, and the estimate on the codimension of its the orbit coincides with the estimate made for the stabilizer of $A$ for the action $\Psi$ in Theorem \ref{stabz}. 
\end{posledica}

The organization of the rest of the paper is the following. In Sec. \ref{notation} we further introduce the notation and prepare some preliminary material. The result on solutions of a certain matrix equation (Lemma \ref{EqT}) is stated and proved in Sec. \ref{cereq}. Finally, Theorem \ref{stabs}, Theorem \ref{stabz} and Corollary \ref{posapp} are proved in the last section.

\section{Preliminaries}\label{notation}

We introduce the so-called \emph{upper-triangular Toeplitz} matrix and \emph{complex-alternating upper-triangular Toeplitz} matrix, respectively:
\small
\begin{align*}
&T(a)=
\begin{bmatrix}
  a_{0} & a_{1} & a_{2} & \ldots  & \ldots & a_{\beta-1}  \\
0 & a_0    & a_{1} & \ddots  &        & \vdots \\
 \vdots & \ddots  & a_0 & \ddots  & \ddots & \vdots \\ 
 \vdots &  & \ddots & \ddots  & a_{1} & a_{2}\\
 \vdots &        &  & \ddots  & a_{0}  & a_{1} \\
0 & \ldots & \ldots & \ldots &  0  & a_0 \\
\end{bmatrix}, \, % a=(a_0,a_1,\ldots,a_{\beta-1})\in \mathbb{C}^n
%
%T_{\pm}(a)=\begin{bmatrix}
%  a_{0} & a_{1}         & a_{2}             & \ldots   & \ldots &   a_{n-1}  %\\
%0       & -{a}_0 & -{a}_{1} & -{a}_2  &     & \vdots \\
% \vdots & \ddots            & a_0               & a_1  & \ddots &  \vdots \\ 
% \vdots &  & \ddots   & \ddots             & \ddots & \vdots\\
% \vdots &        &              & \ddots  & \ddots &(-1)^{n-1}  a_1  \\
 % \vdots &  \vdots      &  \vdots            &         &      (-1)^{n-1}  a_1        %\\
%0       & \ldots            & \ldots & \ldots &  0  &  (-1)^n a_o
%\end{bmatrix},\\
&T_c(a)=
%T(\Rea(a))+iT_{\pm}(\Ima(a))=
\begin{bmatrix}
  a_{0} & a_{1}         & a_{2}             & \ldots   & \ldots &    a_{\beta-1}  \\
0       & \overline{a}_0 & \overline{a}_{1} & \overline{a}_2  &      & \vdots \\
 \vdots & \ddots            & a_0               & a_1  & \ddots &   \vdots \\ 
 \vdots &  & \ddots   & \ddots             & \ddots &  \vdots\\
 \vdots &        &              & \ddots  & \ddots &   \vdots \\
 % \vdots &  \vdots      &  \vdots            &         &       &  a_1 &       \\
0       & \ldots            & \ldots & \ldots &  0      & \ddots
\end{bmatrix},
\end{align*}
\normalsize
with $ a=(a_0,a_1,\ldots,a_{\beta-1})\in \mathbb{C}^m$; here $T_c(a)=[t_{jk}]_{j,k=1}^{\beta}$ with $t_{jk}=\overline{t}_{(j+1)(k+1)}$ for entries on or above the diagonal. 
For example, elementary Jordan blocks are upper-diagonal Toeplitz matrices.
Sometimes it is more convenient to use block complex (complex-alternating) upper-diagonal Toeplitz matrices, respectively:
\small
\begin{align*}
&T(A)=\begin{bmatrix}
  A_{0} & A_{1}         & A_{2}             & \ldots   & \ldots &    A_{\beta-1}  \\
0       & A_0 & A_{1} & A_2  &      & \vdots \\
 \vdots & \ddots            & A_0               & A_1  & \ddots &   \vdots \\ 
 \vdots &  & \ddots   & \ddots             & \ddots &  \vdots\\
 \vdots &        &              & \ddots  & \ddots &   \vdots \\
 % \vdots &  \vdots      &  \vdots            &         &       &  A_1 &   A_{1}    \\
0       & \ldots            & \ldots & \ldots &  0      & A_0
\end{bmatrix}, 
%\quad A_0,\ldots,A_{\beta-1}\in \mathbb{C}^{m \times m},\\
&T_c(A)=\begin{bmatrix}
  A_{0} & A_{1}         & A_{2}             & \ldots   & \ldots &    A_{m-1}  \\
0       & \overline{A}_0 & \overline{A}_{1} & \overline{A}_2  &      & \vdots \\
 \vdots & \ddots            & A_0               & A_1  & \ddots &   \vdots \\ 
 \vdots &  & \ddots   & \ddots             & \ddots &  \vdots\\
 \vdots &        &              & \ddots  & \ddots &   \vdots \\
 % \vdots &  \vdots      &  \vdots            &         &       &  A_1 &       \\
0       & \ldots            & \ldots & \ldots &  0      & \ddots
\end{bmatrix},
%\quad A_0,\ldots,A_{m-1}\in \mathbb{C}^{m_r\times m_s},\\
%
%and $\widetilde{T}_c(A_0,A_1,\ldots,A_{m-1})=
%E_m T_c(A_0,A_1,\ldots,A_{m-1})E_m$, respectively.
%Note that 
%\[
%&\widetilde{T}_c(A_0,A_1,\ldots,A_{m-1})=\left\{
%\begin{array}{ll}
%T_c^{T}(\overline{A}_0^{T},A_1^{T},\overline{A}_2^{T},\ldots), & m %\textrm{  even }\\
%T_c^{T}(A_0^{T},\overline{A}_1^{T},A_2^{T},\ldots), & m \textrm{  odd %}
%\end{array}\right.,\quad m=\min\{\alpha_r,\alpha_s\},
\end{align*}
\normalsize
where $ A=(A_0,A_1,\ldots,A_{\beta-1})\in (\mathbb{C}^{m_r\times m_s})^{\beta}$; again $T_c(A)=[T_{jk}]_{j,k=1}^{\beta}$ with $T_{jk}=\overline{T}_{(j+1)(k+1)}$ on or above the diagonal.

We recall a few basic facts about upper (block) upper-diagonal (block) Toeplitz matrices, these are easy to prove, see e.g. \cite{HornJohn} or \cite{Gant}. %For other properties related to upper Toeplitz matrices see also the %exposition about Sylvester equation in Section.  

\begin{lemma}\label{lemaTop}
\begin{enumerate}
%\hspace{-2mm}
\item \hspace{-2mm} Linear combinations and products of (block) upper-triangular (block) Toeplitz matrices are (block) upper-triangular (block) Toeplitz matrices. 
\item \label{lemaTop2} 
Any two upper triangular Toeplitz matrices of the same size commute. Furthermore, a matrix $B\in \mathbb{C}^{n\times n}$ commutes with $J_n(\lambda)$, $\lambda \in \mathbb{C}$ if and only if $B$ is an
upper triangular Toeplitz matrix.
\end{enumerate}
\end{lemma}

Based on Lemma \ref{lemaTop} (\ref{lemaTop2}) we have the following classical result on solutions of the famous Sylvester's equation, see e.g \cite[Chap. VIII]{Gant}.

\begin{lemma}\label{lemas}
Let $M$, $N$ be two matrices and suppose the matrix equation $MX=XN$.
%, and let $\lambda_1,\lambda_2\in \mathbb{R}_{\geq 0}$, $\mu_1\in \mathbb{R}_{>0}$, $\eta_1,\eta_2\in \mathbb{C}\setminus \mathbb{R}_{\geq 0}$.with $\Ima (\xi_1),\Ima (\xi_2)>0$.
\begin{enumerate}
\item If $M$ and $N $ are any of the matrices of the form  
$
J_m(\lambda)$
for $m=m_1$, $\lambda=\lambda_1$ and $m=m_2$, $\lambda=\lambda_2$, respectively, with $\lambda_1\neq \lambda_2$, it then follows that $X=0$.
%Similary, the same conclusion holds if $M=J_n(\lambda_1,1)$ and $N$ is any of the matrices of the form (\ref{MN}).
%
\item If $M=J_m(\lambda)$ and $N=J_n(\lambda)$ for $\lambda\in \mathbb{C}$, then we have 
\begin{equation}\label{QT00}
X=\left\{\begin{array}{ll}
[0\quad T], & m<n \\
\begin{bmatrix}
T\\
0
\end{bmatrix}, & m>n\\
T, & n=m
\end{array}
\right.,
\end{equation}
where $T\in \mathbb{C}^{p\times p}$, $p=\min\{m,n\}$ is a complex upper-triangular Toeplitz matrix. 
\end{enumerate} 
\end{lemma}

Since up to similarity the general setting of Sylvester's equation is easily reduced to the special case of elementary Jordan blocks considered in Lemma \ref{lemas}, the following proposition is then immediate.

\begin{trditev}\label{posls}
Suppose $A$ is a square matrix and let
\begin{equation}\label{eqH1}
(\mathcal{S}(A))X=X(\mathcal{S}(A)),
\end{equation}
where $\mathcal{S}(A)$ is of the form (\ref{NF1s}).
Let further $\mathcal{S}_{}^{}(A)=\bigoplus_{r=1}^{N}\mathcal{S}_{}(A,\rho_r)$,
%\[
%\mathcal{H}_{}^{1}(A)=\left(\bigoplus_{j}\mathcal{H}_{}^{1}(A,\lambda_j)\right)\oplus \left(\bigoplus_k \mathcal{H}_{}^{1}(A,\mu_k)\right)\oplus %\left( \bigoplus_l \mathcal{H}_{}^{1}(A,\xi_l)\right),
%\]
where all blocks of $\mathcal{S}(A)$ corresponding to the eigenvalue $\rho_r$ with respect to $\mathcal{J}(A)$ are collected together into $\mathcal{S}^{}_{}(A,\rho_r)$. Then $X$ is of the form $X=\oplus_{r=1}^{N}X_{r}$ with $\mathcal{S}^{}_{}(A,\rho_r)X_{r}=X_{r}\mathcal{S}_{}(A,\rho_r)$. Moreover, if $\mathcal{S}_{}^{}(A)=\bigoplus_{j=1}^{N} S_{\alpha_j}(\lambda)$, $\lambda \in \mathbb{C}$, then $X=PYP^{-1}$, where $P=\oplus_{j=1}^{N}P_{\alpha_j}$, $P_{\alpha_j}=\frac{1}{\sqrt{2}}(I_{\alpha_j}+iE_{\alpha_j})$, and $Y=[Y_{jk}]_{j,k=1}^{N}$ is partitioned conformally to blocks as $\mathcal{H}_{}^1(A)$ with the block $Y_{jk}$ of the form (\ref{QT}) for $m=\alpha_j$, $n=\alpha_k$.
We have $X=[X_{jk}]_{j,k=1}^{N}$ with $X_{jk}=P_{\alpha_j}Y_{jk}P_{\alpha_k}^{-1}$ for $j,k\in \{1,\ldots,N\}$.
%
%Given $A=$, then the solution $X$ of the equation $AX=XA$ is of the form. 
\end{trditev}

The following fact was first observed by Bevis, Hall and Hartwig \cite{BHH}. 
%For the sake of completeness we add its short proof.

\begin{lemma}\label{lemaBHH}
Let $M$, $N$ be two matrices and suppose $M\overline{X}=XN$ is the matrix equation.
%, and let $\lambda_1,\lambda_2\in \mathbb{R}_{\geq 0}$, $\mu_1\in \mathbb{R}_{>0}$, $\eta_1,\eta_2\in \mathbb{C}\setminus \mathbb{R}_{\geq 0}$.with $\Ima (\xi_1),\Ima (\xi_2)>0$.
\begin{enumerate}
\item \label{BHH1} If $M$ and $N $ are any of the matrices of the form
\begin{equation*}%\label{MN}
J_m(\lambda,1),\quad 
\begin{bsmallmatrix}
0 & I_n \\
J_n(\eta) & 0
\end{bsmallmatrix},
%J^q_m(\eta),
%\begin{bmatrix}
%0 & I_m \\
%J_m(\xi) & 0
%\end{bmatrix},
\qquad m,n\in \mathbb{N}, \lambda\in \mathbb{R}_{\geq 0}, \eta\in \mathbb{C}\setminus \mathbb{R}_{\geq 0}
\end{equation*}
for either $m=m_1$, $\lambda=\lambda_1$ or $n=n_1$, $\eta=\eta_1$, and either $m=m_2$, $\lambda=\lambda_2$ or $n=n_2$, $\eta=\eta_2$, respectively, with $\lambda_1\neq \lambda_2$, $\eta_1\neq\eta_2$, it then follows that $X=0$.
%Similary, the same conclusion holds if $M=J_n(\lambda_1,1)$ and $N$ is any of the matrices of the form (\ref{MN}).
\item \label{BHH2} If $M=J_m(\lambda,1)$ and $N=J_n(\lambda,1)$ with $\lambda\in \mathbb{R}_{>0}$ (with $\lambda=0$), then we have 
\begin{equation}\label{QT}
X=\left\{\begin{array}{ll}
[0\quad T], & m<n \\
\begin{bmatrix}
T\\
0
\end{bmatrix}, & m>n\\
T, & n=m
\end{array}
\right.,
\end{equation}
where $T\in \mathbb{C}^{p\times p}$ for $p=\min\{m,n\}$ is a real (complex-alternating) upper-triangular Toeplitz matrix. 
%\item \label{BHH3} If $M=J_n(\lambda_1,1)$ and $N=J_m(\lambda_2,1)$ with $\lambda_1=\lambda_2=0$, then $Q$ is any complex-alternating upper-triangular Toeplitz matrix.
\item \label{BHH3} If $M=
\begin{bsmallmatrix}
0 & I_m \\
J_m(\eta) & 0
\end{bsmallmatrix}
%J^q_m(\eta)
$
and
$N=
\begin{bsmallmatrix}
0 & I_n \\
J_n(\eta) & 0
\end{bsmallmatrix}
%J^q_n(\eta)
$, $ \eta\in \mathbb{C}\setminus \mathbb{R}_{\geq 0}$, then
\begin{equation}\label{QTC}
X=\begin{bmatrix}
T_1 & T_2\\
J_m(\eta)\overline{T}_2 & \overline{T}_1
\end{bmatrix}
,
\end{equation}
where $T_1$ and $T_2$ are possibly two different matrices of the form (\ref{QT}) with $T\in \mathbb{C}^{p\times p}$ for $p=\min\{m,n\}$ being any complex upper-triangular Toeplitz matrix. Moreover, if in addition $\eta\in \mathbb{C}\setminus\mathbb{R}$, then $T_2=0$.
%, and let $\lambda_1,\lambda_2\in \mathbb{R}_{\geq 0}$, $\mu_1\in \mathbb{R}_{>0}$, $\eta_1,\eta_2\in \mathbb{C}\setminus \mathbb{R}_{\geq 0}$.with $\Ima (\xi_1),\Ima (\xi_2)>0$.
\end{enumerate} 
\end{lemma}

%\begin{proof}
%P
%\end{proof}

\begin{remark}
We use this lemma to solve the first equation of (\ref{eqAoXXB}). However, note that the solutions given by Lemma \ref{lemaBHH} (\ref{BHH3}) for $\eta\in \mathbb{C}\setminus \mathbb{R}$ are not suited well enough for our developments, hence in this case we shall use another normal form instead of $\begin{bsmallmatrix}
0 & I_m \\
J_m(\eta) & 0
\end{bsmallmatrix}$ (see the proof of Lemma \ref{posl}).
\end{remark}

\begin{lemma}\label{posl}
Suppose $A$ is a square matrix and let  
\begin{equation}\label{eqH1}
(\mathcal{H}^1(A))\overline{X}=X(\mathcal{H}^1(A)),
\end{equation}
where $\mathcal{H}^1(A)$ is of the form (\ref{NF1}).
Let further $\mathcal{H}_{}^{1}(A)=\bigoplus_{r=1}^{N}\mathcal{H}^{1}_{}(A,\rho_r)$,
%\[
%\mathcal{H}_{}^{1}(A)=\left(\bigoplus_{j}\mathcal{H}_{}^{1}(A,\lambda_j)\right)\oplus \left(\bigoplus_k \mathcal{H}_{}^{1}(A,\mu_k)\right)\oplus %\left( \bigoplus_l \mathcal{H}_{}^{1}(A,\xi_l)\right),
%\]
where all blocks of $\mathcal{H}^1(A)$ corresponding to the eigenvalue $\rho_r$ with respect to $\mathcal{J}(A\overline{A})$ in (\ref{JF}) are collected together into $\mathcal{H}^{1}_{}(A,\rho_r)$. Then $X$ is of the form $X=\oplus_{r=1}^{N} X_{r}$ with $\mathcal{H}^{1}_{}(A,\rho_r)\overline{X}_{r}=X_{r}\mathcal{H}_{}^1(A,\rho_r)$. Moreover:
\begin{enumerate}
\item \label{posl1} If $\mathcal{H}_{}^{1}(A)=\bigoplus_{j=1}^{m} H_{\alpha_j}(\lambda)$, $\lambda \geq 0$ ($\lambda=0$), then $X=P^{-1}YP$, where $P=\oplus_{j=1}^{m}P_{\alpha_j}$, and $Y=[Y_{jk}]_{j,k=1}^{m}$ is partitioned conformally to blocks as $\mathcal{H}_{}^1(A)$ with the block $Y_{jk}$ of the form (\ref{QT}) for $m=\alpha_j$, $n=\alpha_k$ and a real (complex-alternating) upper-triangular Toeplitz matrix $T$.
We have $X=[X_{jk}]_{j,k=1}^{m}$ with $X_{jk}=P_{\alpha_j}^{-1}Y_{jk}P_{\alpha_k}$.
%, $j,k\in \{1,\ldots,m\}$.
%
\item \label{posl2} If $\mathcal{H}_{}^{1}(A)=\bigoplus_{k=1}^{m} K_{\beta_k}(\mu)$, $ \mu >0$, then $X=P^{-1}V^{-1}SYS^{-1}VP$, where $P=\bigoplus_{k=1}^{m}
e^{\frac{i\pi}{4}}(P_{\beta_k}\oplus P_{\beta_k})
$,
$V=\bigoplus_{k=1}^{m}e^{i\frac{\pi}{4}}(W_{\beta_k}\oplus \overline{W_{\beta_k}})$ with $W_{\beta_k}= \oplus_{j=0}^{\beta_k-1} i^{j}$, and $S=\bigoplus_{k=1}^{m}
\begin{bsmallmatrix}
0 & U_{\beta_k}(i\mu) \\
J_{\beta_k}(-i\mu)\overline{U_{\beta_k}(i\mu) }& 0
\end{bsmallmatrix}$ with $U_{\beta_k}(i\mu)$ as a solution of the matrix equation $U_{\beta_k}(i\mu)J_{\beta_k}(-\mu^{2})=(J_{\beta_k}(i\mu))^{2}U_{\beta_k}(i\mu)$, and $=[Y_{jk}]_{j,k=1}^{m}$ is partitioned conformally to blocks as $\mathcal{H}_{}^1(A)$ and the block $Y_{jk}$ for any $j,k\in \{1,\ldots,m\}$ is of the form (\ref{QTC}) with $T_1$ and $T_2$ possibly two different matrices of the form (\ref{QT}) for a complex upper-triangular Toeplitz matrix $T$.
%We have 
%$X_{jk}=P_{\beta_j}'^{-1}V_{\beta_j}^{-1}S_{\beta_j}Y_{jk}S_{\beta_k}^{-1}V_{\beta_k}P_{\beta_k}'$ with $Y_{\beta_j\beta_k}$ of the form (\ref{QTC}).
%
%
\item \label{posl3} If $\mathcal{H}_{}^{1}(A)=\bigoplus_{l=1}^{m} L_{\gamma_l}(\xi)$, $ \xi^{2}\in \mathbb{C}\setminus  \mathbb{R}$, then $X=P^{-1}YP$, $P_{}=\oplus_{j=1}^{m}
P_{\gamma_j}\oplus P_{\gamma_j}
$, and $Y=[Y_{jk}]_{j,k=1}^{m}$ partitioned conformally to blocks as $\mathcal{H}_{}^1(A,\rho_r)$, with $Y_{jk}$ of the form (\ref{QTC}) with $T_2=0$ and $T_1$ of the form (\ref{QT}) for a complex upper-triangular Toeplitz matrix $T$. 
%(Setting 
%$U=\oplus_k U_k'=\oplus_k(U_{\gamma_k}(\xi)\oplus \overline{U_{\beta_k}(\xi)})$, 
%$Y=[Y_{jk}]_{j,k}$ we get $X=P^{-1}Y^{r}P$,
%where $U_{\beta_k}(\xi)$ is a solution of the equation $U_{\beta_k}(\xi)J_{\beta_k}(\xi^{2})=(J_{\beta_k}(\xi))^{2}U_{\beta_k}(\xi)$.) 
%We have 
%$X_{jk}=(P_{j}\oplus P_j)^{-1}Y_{jk}(P_{k}\oplus P_k)$.
\end{enumerate}
(Here we denoted $P_{\alpha}=\frac{e^{{\scriptscriptstyle -\frac{i\pi}{4}}}}{\sqrt{2}}(I_{\alpha}+iE_{\alpha})$.)
\end{lemma}

\begin{proof}
First observe how $A\overline{X}=XB$ transforms under consimilarity. For 
\[
\widetilde{A}=SA\overline{S}^{-1}, \quad\widetilde{B}=TB\overline{T}^{-1},\qquad  \det(S)\neq 0,\det(T)\neq 0
\]
we get
\begin{equation}\label{MQQNt}
\widetilde{A}\overline{Y}=Y\widetilde{B},\qquad (Y=SXT^{-1}).
\end{equation}
Therefore $X$ is a solution of $A\overline{X}=XB$ precisely when $Y=SXT^{-1}$ is a solution of $\widetilde{A}\overline{Y}=Y\widetilde{B}$.
%It is thus enough to solve the first equation of (\ref{eqAoXXB}) for the case %when $A$, $B$ are normal forms.
%

Furthermore, suppose we are given block-diagonal matrices (normal forms):
\[
A=\oplus_{\mu} A_{\mu},\quad B=\oplus_{\nu} B_{\nu}, \qquad
\widetilde{A}=\oplus_{\mu}\widetilde{A}_{\mu},\quad 
\widetilde{B}=\oplus_{\nu} \widetilde{B}_{\nu}, \qquad 
S=\oplus S_{\mu}, \quad T=\oplus T_{\nu}
\]
for some nonsingular $S_{\mu}$ and $T_{\nu}$ (hence $\widetilde{A}_{\mu}=S_{\mu}A_{\mu}\overline{S}_{\mu}^{-1}$, $\widetilde{B}_{\nu}=T_{\nu}B_{\nu}\overline{T}_{\nu}^{-1}$).
If $X=[X_{\mu \nu}]_{\mu,\nu}$ and $Y=[X_{\mu \nu}]_{\mu,\nu}$ as a block-matrices are partitioned into blocks conformally to $A$, $B$ and $\widetilde{A}$, $\widetilde{B}$, then the first equation of (\ref{eqAoXXB}) and (\ref{MQQNt}) can be seen as systems of equations, respectively:
\begin{align}\label{MQQN}
&A_{\mu}\overline{X}_{\mu \nu}=X_{\mu \nu}B_{ \nu}, \qquad \overline{A}_{\mu}\overline{Y}_{\mu \nu}=Y_{\mu \nu}\overline{B}_{ \nu},\qquad (Y_{\mu\nu}=S_{\mu}X_{\mu \nu}T_{\nu}^{-1}).
%&M_{\mu}\overline{(iQ_{\mu \nu})}=(iQ_{\mu \nu})(-N_{ \nu}),
\end{align}
%
%where $M_{\mu}$ and $N_{\nu}$ are any of the diagonal blocks of $\mathcal{H}^{\varepsilon \widehat{\epsilon}}(A)$ and $\mathcal{H}^{\widetilde{ \epsilon}}(A)$. 
%Moreover, it is sufficient to solve the equation (\ref{MQQN}) if $M_{\mu}$ and $N_{\nu}$ are the diagonal blocks of $\mathcal{H}^{1}(A)$.
%, i.e. 
%\[
%\mathcal{H}^{1}(A)Q=Q\mathcal{H}^{1}(A).
%\]
%In particular, for $\widetilde{M}=\mathcal{J}'_q(A)$, $\widetilde{N}=\mathcal{J}'_q(A)$ and applying Lemma \ref{lemaBHH} on (\ref{MQQNt1}) we obtain that $Y$ is a block-matrix, partitioned to blocks conformally to $M$ and $N$, with blocks of the form (\ref{QT}), (\ref{QTC}). In particular, if $Y_{\mu\nu}$ corresponds to blocks $M_{\mu}$ and $N_{\nu}$ associated to two different eigenvalues of $A\overline{A}$, then Lemma \ref{lemaBHH} (\ref{BHH1}) implies $Y_{\mu\nu}=0$. Furthermore, collecting all blocks corresponding to the eigenvalue $\lambda$ into $\mathcal{J}^{\lambda}_q(A)$, then yields that $Y$ is of the form $Y=\oplus_{\lambda}Y^{\lambda}$ with $\mathcal{J}^{\lambda}_q(A)\overline{Y}^{\lambda}=Y^{\lambda}\mathcal{J}^{\lambda}_q(A)$, where $Y^{\lambda}$ is partitioned conformally to blocks with $\mathcal{J}^{\lambda}_q(A)$.  
%
%Taking into account the facts presented above and using the relationship %beetwen normal forms 

In view of (\ref{MQQNt}) for $A=B=\mathcal{H}^{1}(A)$, $\widetilde{A}=\widetilde{B}=\mathcal{J}_q(A)$ and $S=T=P$ with $P$ of the form (\ref{P}) (see (\ref{NFP}), (\ref{P})), the equation (\ref{eqH1}) is equivalent to the equation
\begin{equation}\label{eqqY}
\mathcal{J}_q(A)\overline{Y}=Y\mathcal{J}_q(A), \qquad Y=PXP^{-1};
\end{equation}
here $\mathcal{J}_q(A)$ is the corresponding consimilarity normal form (\ref{CF00}).
If $Y$ is partitioned conformally to blocks as $\mathcal{H}_{}^1(A)$ and $\mathcal{J}_{}^1(A)$, then this equation splits into a system of matrix equations (see also (\ref{MQQN})).
Applying Lemma \ref{lemaBHH} (\ref{BHH1}) then immediately implies the first part of the lemma.

Next, we prove (\ref{posl1}). Assume $\mathcal{J}_{q}^{}(A)=\bigoplus_{j=1}^{m} J_{\alpha_j}(\lambda)$ and let $Y=[Y_{jk}]_{j,k=1}^{m}$ in (\ref{eqqY}) be partitioned conformally to blocks as $\mathcal{H}_{}^1(A)$ and $\mathcal{J}_{}^1(A)$ (see (\ref{MQQN})). Applying Lemma \ref{lemaBHH} (\ref{BHH2}) now implies that blocks $Y_{jk}$ are the form (\ref{QT}) and hence $X=[X_{jk}]_{j,k=1}^{m}$ with $X_{jk}=P_{\alpha_j}^{-1}Y_{jk}P_{\alpha_k}$. 
%This proves (\ref{posl1}). 

Proceed with (\ref{posl2}). 
We set $\mathcal{J}_{q}(A)=\bigoplus_{k=1}^{m} \begin{bsmallmatrix}
0 & J_{\beta_k}(\eta) \\
-J_{\beta_k}(\eta) & 0
\end{bsmallmatrix}$.
%and 
%and using (\ref{NFP}) we obtain
%$\mathcal{H}_{}^{1}(A,\rho_r)=P^{-1}\mathcal{J}_{}^{}(A,\rho_r)\overline{P}$ %and hence 
%$\mathcal{J}^{q}_{}(A)\overline{P}\overline{X}^{r}\overline{P}^{-1}=PX^{r}P^{-1}\mathcal{J}_{q}(A)$; 
%$P=\oplus_k e^{i\frac{\pi}{4}}(P_{\beta_k}\oplus P_{\beta_k})$. 
%
It is not difficult to see
\[
W_{\beta_k}'
\begin{bsmallmatrix}
0 & J_{\beta_k}(\eta) \\
-J_{\beta_k}(\eta) & 0
\end{bsmallmatrix}
=
\begin{bsmallmatrix}
0 & J_{\beta_k}(i\eta) \\
J_{\beta_k}(-i\eta) & 0
\end{bsmallmatrix}
\overline{W_{\beta_k}'}
, \quad W_{\beta_k}'=e^{i\frac{\pi}{4}}(W_{\beta_k}\oplus \overline{W}_{\beta_k}), \quad
W_{\beta_k}= \oplus_{j=0}^{\beta_k-1} i^{j}.
\]
Next, we find the antidiagonal block matrix $T_{\beta_k}(i\mu)=\begin{bsmallmatrix}
0 & U_{\beta_k}(i\mu) \\
U_{\beta_k}'(\mu) & 0
\end{bsmallmatrix}$ such that
\[
(T_{\beta_k}(i\mu))^{-1}\begin{bsmallmatrix}
0 & J_{\beta_k}(i\eta) \\
J_{\beta_k}(-i\eta) & 0
\end{bsmallmatrix}\overline{T}_{\beta_k}(i\mu)=\begin{bsmallmatrix}
0 & I_{\beta_k}\\
J_{\beta_k}(-\eta^{2}) & 0
\end{bsmallmatrix}.
\]
We have $U_{\beta_k}(i\mu)J_{\beta_k}(-\mu^{2})=J_{\beta_k}(i\mu)\overline{U'}_{\beta_k}(i\mu)$ and $U_{\beta_k}'(i\mu)=J_{\beta_k}(-i\mu)\overline{U}_{\beta_k}(i\mu)$, therefore $U_{\beta_k}(i\mu)$ is a solution of the equation $U_{\beta_k}(i\mu)J_{\beta_k}(-\mu^{2})=(J_{\beta_k}(i\mu))^{2}U_{\beta_k}(i\mu)$ and further $U_{\beta_k}'(i\mu)=J_{\beta_k}(-i\mu)\overline{U}_{\beta_k}(i\mu)$.
We set
\[
V=\bigoplus_{k=1}^{m}e^{i\frac{\pi}{4}}(W_{\beta_k}\oplus \overline{W_{\beta_k}}), \qquad 
S=\bigoplus_{k=1}^{m}
\begin{bsmallmatrix}
0 & U_{\beta_k}(i\mu) \\
J_{\beta_k}(-i\mu)\overline{U_{\beta_k}(i\mu)} & 0
\end{bsmallmatrix},
\]
thus for $\mathcal{J}_{q}'(A)=\bigoplus_{k=1}^{m}\begin{bsmallmatrix}
0 & I_{\beta_k} \\
J_{\beta_k}(-\eta^{2}) & 0
\end{bsmallmatrix}$ (in view of (\ref{CFHH})) we finally obtain
\[
\mathcal{J}_{q}'(A)\overline{Y'}=Y'\mathcal{J}_{q}'(A),\qquad Y'=S^{-1}VPXP^{-1}V^{-1}S,
\]
%\[
%\begin{bsmallmatrix}
%0 & I_{\beta_j}\\
%J_{\beta_j}(-\eta^{2}) & 0
%\end{bsmallmatrix}\overline{S}_{\beta_j}^{-1}\overline{V}_{\beta_j}\overline{P}_{\beta_j}\overline{X}_{jk}\overline{P}_{\beta_k}^{-1}\overline{V}_{\beta_k}^{-1}\overline{S}_{\beta_k}=S_{\beta_k}^{-1}V_{\beta_k}P_{\beta_k}X_{jk}P_{\beta_j}^{-1}V_{\beta_j}^{-1}S_{\beta_j}\begin{bsmallmatrix}
%0 & I_{\beta_k}\\
%J_{\beta_k}(-\eta^{2}) & 0
%\end{bsmallmatrix}.
%\]
where $Y'=[Y'_{jk}]_{j,k=1}^{m}$ is partitioned conformally to blocks as $\mathcal{H}_{}^1(A)$, and $Y_{\beta_j\beta_k}'$ of the form (\ref{QTC}). This now implies (\ref{posl2}).

It is left to show (\ref{posl3}). We set $\mathcal{J}_{q}(A)=\bigoplus_{l=1}^{m} \begin{bsmallmatrix}
0 & J_{\gamma_l}(\xi) \\
J_{\gamma_l}(\overline{\xi}) & 0
\end{bsmallmatrix}$ and find the solutions of the equation (\ref{eqqY}). 
%and using (\ref{NFP}) we obtain
%$\mathcal{H}_{}^{1}(A,\rho_r)=P^{-1}\mathcal{J}_{}^{}(A,\rho_r)\overline{P}$ %and hence 
%
%$\mathcal{J}_{q}(A,\rho_r)\overline{P}\overline{X}^{r}\overline{P}^{-1}=PX^{r}P^{-1}\mathcal{J}_{q}(A,\rho_r)$; $P=\oplus_l(P_{\gamma_l}\oplus %P_{\gamma_l})$. 
Seeing $Y=[Y_{kl}]_{k,l=1}^{m}$ conformally to blocks as $\mathcal{J}_{q}(A)$, we can write (\ref{eqqY}) blockwise; we find all block matrices $Y_{kl}(\xi)=\begin{bsmallmatrix}
R_1 & R_2 \\
R_3 & R_4
\end{bsmallmatrix}$ such that
\[
\begin{bsmallmatrix}
0 & J_{\gamma_k}(\xi) \\
J_{\gamma_k}(\overline{\xi}) & 0
\end{bsmallmatrix}\overline{R}_{kl}(\xi)
=
R_{kl}(\xi)\begin{bsmallmatrix}
0 & J_{\gamma_l}(\xi) \\
J_{\gamma_l}(\overline{\xi}) & 0
\end{bsmallmatrix}.
\]
We have 
\begin{align}\label{eqRJ}
R_{2}J_{\gamma_l}(\overline{\xi})=J_{\gamma_k}(\xi)\overline{R}_{3},\quad 
R_{3}J_{\gamma_l}(\xi)=J_{\gamma_k}(\overline{\xi})\overline{R}_{2},\\
R_{1}J_{\gamma_l}(\xi)=J_{\gamma_k}(\xi)\overline{R}_{4},\quad 
R_{4}J_{\gamma_l}(\overline{\xi})=J_{\gamma_k}(\overline{\xi})\overline{R}_{1}.\nonumber
\end{align}
Combining the first two equations we get $\overline{R}_{3}(J_{\gamma_l}(\overline{\xi}))^{2}=(J_{\gamma_k}(\xi))^{2}\overline{R}_{3}$ and $\overline{R}_{2}(J_{\gamma_l}(\xi))^{2}=(J_{\gamma_k}(\overline{\xi}))^{2}\overline{R}_{2}$, which implies $R_3=R_2=0$. Subtracting the third and the last conjugated equation of (\ref{eqRJ}) gives 
%$R_{1}(J_{\gamma_l}(\xi))^{2}=(J_{\gamma_k}(\xi))^{2}R_{1}$ and
$-(R_{1}-\overline{R}_4)J_{\gamma_l}(\xi)=J_{\gamma_k}(\xi)(R_1-\overline{R}_{4})$. We deduce that 
%$R_1$ is any solution of the equation $R_1 (J_{\gamma_l}(\xi))^{2}=(J_{\gamma_k}(\xi))^{2}R_{1}$, while the later equations yields 
$F(R_{1}-\overline{R}_4)J_{\gamma_l}(\xi)=J_{\gamma_k}(-\xi)F(R_1-\overline{R}_{4})$, $F=-1\oplus 1\oplus -1\oplus \cdots$, thus $R_4=\overline{R}_1$. The third (the fourth) equation of (\ref{eqRJ}) then yield that $R_1$
%Taking into account that $U_{\beta_k}(\xi)J_{\beta_k}(\xi^{2})=(J_{\beta_k}(\xi))^{2}U_{\beta_k}(\xi)$ we obtain $(U_{\beta_l}(\xi))^{-1}R_1 U_{\beta_k}(\xi)J_{\gamma_l}(\xi^{2})=J_{\gamma_l}(\xi^{2})(U_{\beta_l}(\xi))^{-1}R_{1}U_{\beta_k}(\xi)$, which yields that $(U_{\beta_l}(\xi))^{-1}R_{1}U_{\beta_k}(\xi)$ is of the %form (\ref{QT}) with $T$ 
is an upper triangular complex Toeplitz matrix.
%
%We set $U=\oplus_l(U_{\beta_l}(\xi)\oplus \overline{U_{\beta_l}(\xi)})$. Therefore $Y=U^{-1}ZU$, where $Z=[Z_{kl}]_{kl}$ is block upper Toeplitz with blocks of the %form (\ref{QTC}), but where $T_1$ is upper-triangular complex Toeplitz matrix of the form (\ref{QT}) and $T_2=0$. 
\end{proof}

%Observe that $Q$ is the solution of $M\overline Q=QN$ precisely when $iQ$ is the solution of $M\overline{(iQ)}=(iQ)(-N)$. 

Sometimes it is more convenient to deal with block Toeplitz matrices than with block matrices having Toeplitz blocks. This transformation of matrices can be achieved by $T$-conjugating with a suitable permutation matrix (see e.g. \cite{Lin}).

Suppose $Y=[Y_{rs}]_{r,s=1}^{N}$, where further $Y_{rs}$ is a $m_r\times m_s$ block matrix whoose blocks are of dimension $\alpha_r\times \alpha_s$ and of the form (\ref{QT})
for $m=\alpha_r$, $n=\alpha_s$, thus 
of the form 
\begin{equation}
\left\{
\begin{array}{ll}
[0\quad T], & \alpha_{r}<\alpha_{s}\\
\begin{bmatrix}
T\\
0
\end{bmatrix}, & \alpha_{r}>\alpha_{s}\\
T,& \alpha_{r}=\alpha_{s}
\end{array}\right.,
\end{equation}
and such that $T\in \mathbb{C}^{b_{rs}\times b_{rs}}$, $b_{rs}=\min\{\alpha_r, \alpha_s\}$ is any complex (complex-alternating) upper-triangular Toeplitz matrix.

Let $e_1,e_2,\ldots,e_{\alpha_r m_r}$ be the standard orthonormal basis in $\mathbb{C}^{\alpha_r m_r}$.
We set a permutation matrix formed by these vectors: 
\begin{equation}\label{perS}
\Omega_r=\left[e_1\;e_{\alpha_r+1}\;\ldots\;e_{(m_r-1)\alpha_r+1}\;e_2\;e_{\alpha_r+2}\;\ldots\;e_{(m_r-1)\alpha_r+2}\;\ldots\;e_{\alpha_r}\;e_{2\alpha_r}\;\ldots\;e_{\alpha_rm_r}\right].
\end{equation}
Observe that multiplicating with this matrix from the right puts the first, $(\alpha_r+1)$-th,\ldots,$((m_r-1)\alpha_r+1)$-th column together, and further the second, $(\alpha_r+2)$-th,\ldots,$((m_r-1)\alpha_r+2)$-th column together, and soforth. Similary, multiplicating with $\Omega_r^{T}$ from the left collects the first, $(\alpha_r+1)$-th,\ldots,$((m_r-1)\alpha_r+1)$-th row together, and further the second, $(\alpha_r+2)$-th,\ldots,$((m_r-1)\alpha_r+2)$-th row together, and soforth. 

We set $\Omega=\oplus_{r=1}^{N}\Omega_r$ and we obtain
\begin{align*}
\mathcal{Y}=\Omega^{T}Y\Omega=  &  (\oplus_{r=1}^{N}\Omega_r^{T})[Y_{rs}]_{r,s=1}^{N}(\oplus_{r=1}^{N}\Omega_r)=[\Omega_r^{T}Y_{rs}\Omega_s]_{r,s=1}^{N}.
%\Omega^{T}\widetilde{Z}\Omega=  &  (\oplus_{r=1}^{N}\Omega_r^{T})[\widetilde{Z}_{rs}]_{r,s=1}^{N}(\oplus_{r=1}^{N}\Omega_r)=[\Omega_r^{T}\widetilde{Z}_{rs}\Omega_s]_{r,s=1}^{N}.
\end{align*}
Now fix $r,s$ and let $b=\min\{\alpha_r,\alpha_s\}$. Denote 
\[
(Y_{rs})_{jk}=
\left\{
\begin{array}{ll}
[0\quad T_{jk}], & \alpha_{r}<\alpha_{s}\\
\begin{bmatrix}
T_{jk}\\
0
\end{bmatrix}, & \alpha_{r}>\alpha_{s}\\
T_{jk},& \alpha_{r}=\alpha_{s}
\end{array}\right., \qquad j\in \{1,\ldots m_r\},\quad k\in \{1,\ldots m_s\},
\]
where $T_{jk}=T(a_0^{jk},a_1^{jk},\ldots,a_{b-1}^{jk})$ (or $T_{jk}=T_c(a_0^{jk},a_1^{jk},\ldots,a_{b-1}^{jk})$).
Setting matrices $A_n=[a_n^{jk}]_{j,k=1}^{m_r,m_s}\in \mathbb{C}^{m_r\times m_r}$ for $n\in \{0,\ldots,b-1\}$ and $A=(A_0,\ldots,A_{b-1})$, we obtain that $\mathcal{Y}$ is a $N\times N$ block matrix, and its block $\mathcal{Y}_{rs}$ 
%(in the $r$-th row and $s$-th column) 
is an $\alpha_r\times \alpha_s$ block matrix of the form:
\begin{equation}\label{YTrs}
\mathcal{Y}_{rs}=\Omega_r^{T}Y_{rs}\Omega_s=
\left\{
\begin{array}{ll}
[0\quad \mathcal{T}], & \alpha_r<\alpha_s\\
\begin{bmatrix}
\mathcal{T}\\
0
\end{bmatrix}, & \alpha_r>\alpha_s\\
\mathcal{T},& \alpha_r=\alpha_s
\end{array}\right.,
\end{equation}
where 
$\mathcal{T}=T(A)$ (or $ \mathcal{T}=T_c(A)$) is a complex (complex-alternating) upper block Toeplitz matrix.
In particular, $T_c(A)=T(A)$ for real matrices $A_0,\ldots,A_{b-1}$.

\begin{example}
$N=2$, $\alpha_1=3$, $m_1=2$, $\alpha_2=2$, $m_2=3$:
\[
\Omega_3^{T}\begin{bmatrix}[cc|cc|cc]
a_1 & b_1 & a_2 & b_2 & a_3 & b_3 \\
0   & \overline{a}_1 & 0   & \overline{a}_2 & 0   & \overline{a}_3\\
0   & 0   & 0   & 0   & 0   &  0 \\
\hline
a_4 & b_4 & a_5 & b_5 & a_6 & b_6 \\
0   & \overline{a}_4 & 0   & \overline{a}_5 & 0   & \overline{a}_6\\
0   & 0   & 0   & 0   & 0   &  0 
\end{bmatrix}\Omega_2
=
\begin{bmatrix}[ccc|ccc]
a_1 & a_2 & a_3 & b_1 & b_2 & b_3 \\
a_4 & a_5 & a_6   & b_4 & b_5   & b_6\\
\hline
0   & 0   & 0   & \overline{a}_1   & \overline{a}_2   &  \overline{a}_3 \\
0 &  0 &   0 & \overline{a}_4 & \overline{a}_5 & \overline{a}_6 \\
\hline
0   & 0 & 0   & 0 & 0   & 0\\
0   & 0   & 0   & 0   & 0   &  0 
\end{bmatrix}.
\]
\end{example}

\section{Certain block matrix equation}\label{cereq}

\quad
%The equation (\ref{ortoD2}) can then be written in the following form
%coincides for $\mathcal{B}=\mathcal{B}'$ with the equation
%
In this section we consider certain block matrix equations such that blocks are upper-triangular block Toeplitz matrices. 

Let $\alpha_{1}>\alpha_{2}>\ldots >\alpha_{N}$ and suppose 
\begin{align}\label{BBF}
&\mathcal{B}=\bigoplus_{r=1}^{N}T\big(B_0^{r},B_1^{r},\ldots,B_{\alpha_r-1}^{r}\big),\, \mathcal{B}'=\bigoplus_{r=1}^{N}T\big(G_0^{r},G_1^{r},\ldots,G_{\alpha_r-1}^{r}\big),\, \mathcal{F}=\bigoplus_{r=1}^{N}E_{\alpha_r}(I_{m_r}),\\
&B_0^{r},G_0^{r},\ldots, B_{\alpha_r-1}^{r}, G_{\alpha_r-1}^{r}\in \mathbb{C}_S^{m_r\times m_r}, \quad B_0^{r},G_0^{r}\in GL_{m_r}(\mathbb{C})\nonumber
\end{align}
%$\mathcal{B}=\bigoplus_{j=1}^{N} \mathcal{B}_j$, $\mathcal{B}'=\bigoplus_{j=1}^{N} \mathcal{B}_j'$ 
are quasi-diagonal matrices, and such that diagonal blocks of $\mathcal{B}$, $\mathcal{B}'$ are block upper-triangular Toeplitz matrices of different size and with nonsingular symmetric blocks. 
We denot the backward $\alpha\times \alpha$ block $m\times m$ identity-matrix (with identity on the anti-diagonal) by
$E_{\alpha}(I_{m})=
\begin{bsmallmatrix}
 0                 &      & I_{m}\\
            &   \iddots     &  \\
%     1      &              &  \\
I_{m}            &           &  0\\
\end{bsmallmatrix}
$.
We shall solve a matrix equation
\begin{equation}\label{eqFYFIY}
\mathcal{B}'=\mathcal{F}\mathcal{Y}^{T}\mathcal{F}\mathcal{B} \mathcal{Y},
\end{equation}
where $\mathcal{Y}$ is partitioned conformaly to blocks as $\mathcal{F}$, $\mathcal{B}$, $\mathcal{B}'$, and such that 
\begin{equation}\label{0T0}
\mathcal{Y}=[\mathcal{Y}_{rs}]_{r,s=1}^{N},\quad \mathcal{Y}_{rs}=\left\{
\begin{array}{ll}
\begin{bmatrix}
0 &  \mathcal{T}_{rs}
\end{bmatrix}, & \alpha_r<\alpha_s\\
\begin{bmatrix}
\mathcal{T}_{rs}\\
0
\end{bmatrix}, & \alpha_r>\alpha_s\\
\mathcal{T}_{rs},& \alpha_r=\alpha_s
\end{array}\right., \quad r,s\in \{1,\ldots,N\},
%\quad T_{rs}=T(A_0^{rs},\ldots,A_{b_{rs}-1}^{rs}),\,a_{rs}=\min\{\alpha_r, \alpha_s\}
\end{equation}
and for all $r,s\in \{1,\ldots,N\}$ either $\mathcal{T}_{rs}$ 
%$T_{rs}=T(A_0^{rs},A_1^{rs},\ldots,A_{a_{rs}-1}^{rs})$ (or $T_{rs}=T_c(A_0^{rs},A_1^{rs},\ldots,A_{a_{rs}-1}^{rs})$ and $\mathcal{B}$, %$\mathcal{B}'$ real) with $a_{rs}=\min\{\alpha_r, \alpha_s\}$ 
is an upper-tri\-angu\-lar block Toeplitz matrix 
or $\mathcal{B}$, $\mathcal{B}'$ are real and $\mathcal{T}_{rs}$ is a  complex-alternating block Toeplitz matrix; in both cases $\mathcal{T}_{rs}$ is a square block matrix of order $b_{rs}=\min\{\alpha_r, \alpha_s\}$ with blocks of dimension $m_r\times m_s$.

The following lemma is the key ingredient in the proof of Theorem \ref{stabs} and Theorem \ref{stabz}. 

\begin{lemma}\label{EqT}
Let $\mathcal{B}$, $\mathcal{B}'$ and $\mathcal{Y}$ be of the form as in (\ref{BBF}) and (\ref{0T0}), respectively.
If matrices $\mathcal{B}$, $\mathcal{B}'$ are given, then $\mathcal{Y}$ in the equation (\ref{eqFYFIY})
%if and only 
%if $B_0^{r},G_0^{r}$ in (\ref{BBF}) have the same inertia-matrix for %all $r\in \{1,\ldots,N\}$. 
satisfies the following:
\begin{enumerate}[label={\bf (\Roman*)},ref={\Roman*},wide=0pt,itemsep=15pt]
\item \label{EqT1} If 
$\mathcal{T}_{rs}=T(A_0^{rs},A_1^{rs},\ldots,A_{b_{rs}-1}^{rs})$, 
%$a_{rs}=\min\{\alpha_r, \alpha_s\}$, 
$A_0^{rs},\ldots,A_{b_{rs}-1}^{rs}\in \mathbb{C}^{m_r\times m_s}$ for all $r,s\in \{1,\ldots,N\}$ in (\ref{0T0}), then the solution $\mathcal{Y}$ of the equation (\ref{eqFYFIY}) exists and it is 
%$\mathcal{Y}\subset \mathbb{C}^{q\times q}$ 
%($\mathcal{Y}\subset \mathbb{R}^{q\times q}$) 
an immersed complex submanifold in $\mathbb{C}^{q\times q}$ for $q=\sum_{r=1}^{N}\alpha_r m_r$ and of complex dimension
\begin{equation}\label{dimEQT1}
\sum_{r=1}^{N}\alpha_r m_r\big(\tfrac{1}{2}(m_r-1)+\sum_{s=1}^{r-1}m_s\big).
\end{equation}
Moreover, the following statements hold:
\begin{enumerate}
\item \label{EqT1a}
If $A_0^{rs},\ldots,A_{b_{rs}-1}^{rs}\in \mathbb{R}^{m_r\times m_s}$ for all $r,s\in \{1,\ldots,N\}$ and $\mathcal{B}$, $\mathcal{B}'$ are real, then
$\mathcal{Y}\cap \mathbb{R}^{q\times q}\neq \emptyset$ with $q=\sum_{r=1}^{N}\alpha_r m_r$ if and only if $B_0^{r},G_0^{r}$ in (\ref{BBF}) have the same inertia 
%(i.e. number of positive and negative eigenvalues) 
for all $r\in \{1,\ldots,N\}$. If any of the later two (hence both) conditions is fulfiled, then $\mathcal{Y}\cap \mathbb{R}^{q\times q}$ is an immersed real submanifold of real dimension (\ref{dimEQT1}).
%if $B_0^{r},G_0^{r}$ in (\ref{BBF}) have the same inertia-matrix for %all $r\in \{1,\ldots,N\}$.
%if $B_0^{r},G_0^{r}$ in (\ref{BBF}) have the same inertia-matrix for %all $r\in \{1,\ldots,N\}$.
%
\item \label{EqT1b}
When for any $r\in \{1,\ldots,N\}$, $n\in \{1,\ldots,b_{rs}-1\}$ we have $m_r=2m_r'$ and
\begin{align}\label{aass}
&B_0^{r}=u_0^{r}\big(-\mu^{2}I_{m_r}\oplus I_{m_r} \big),\quad
%&B_1^{r}=((-\mu^{2}\overline{u}_1^{r}+\overline{u}_0^{r})I_{m_r}   \oplus u_1^{r}I_{m_r}),\quad\\
%&B_2^{r}=((-\mu^{2}\overline{u}_2^{r}+\overline{u}_1^{r})I_{m_r}   \oplus u_2^{r}I_{m_r}),\\
B_n^{r}=(-\mu^{2}u_n^{r}+u_{n-1}^{r})I_{m_r}   \oplus u_2^{r}I_{m_r},\quad u_0,\ldots,u_{b_{rs}-1}\in \mathbb{R},\\
&A_0^{rs}=
\begin{bsmallmatrix}
V_0^{rs} &  W_0^{rs} \\
-\mu^{2}\overline{W}_0^{rs} &  \overline{V}_0^{rs} 
\end{bsmallmatrix}, \qquad
A_n^{rs}=
\begin{bsmallmatrix}
V_n^{rs} &  W_n^{rs} \\
-\mu^{2}\overline{V}_n^{rs}+\overline{V}_{n-1}^{rs} &  \overline{W}_k^{rs} 
\end{bsmallmatrix}, \quad
%b_{rs}-1\geq k \geq 1,
V_n^{rs}, W_n^{rs}\in \mathbb{C}^{n_r\times n_r}, \mu >0,\nonumber
\end{align}
then $\mathcal{Y}\subset \mathbb{C}^{q\times q}$ contains an immersed real submanifold of real dimension
\[
\sum_{r=1}^{N}m_r'\big(4\alpha_r m_r'-\alpha_r-2m_r'+8\alpha_r \sum_{s=1}^{r-1}m_s'\big).
\]
\end{enumerate}
\item \label{EqT2} If 
$\mathcal{T}_{rs}=T_c(A_0^{rs},A_1^{rs},\ldots,A_{b_{rs}-1}^{rs})$, 
%$a_{rs}=\min\{\alpha_r, \alpha_s\}$, 
$A_0^{rs},\ldots,A_{b_{rs}-1}^{rs}\in \mathbb{C}^{m_r\times m_s}$ for all $r,s\in \{1,\ldots,N\}$ in (\ref{0T0}), then the solution $\mathcal{Y}$ of the equation (\ref{eqFYFIY}) exists precisely when $B_0^{r},G_0^{r}$ in (\ref{BBF}) have the same inertia for all $r\in \{1,\ldots,N\}$ such that $\alpha_r$ is even. If any of the later conditions (hence both) is fulfiled, then $\mathcal{Y}\subset \mathbb{C}^{q\times q}$ with $q=\sum_{r=1}^{N}\alpha_r m_r$ is an immersed real submanifold of dimension
\begin{align*}
\tfrac{3}{2}\sum_{\alpha_r \textrm{ even}}m_r\alpha_r-\tfrac{1}{2}\sum_{\alpha_r \textrm{ odd}}m_r(\alpha_r+1)+2\sum_{r=1}^{N}\alpha_r m_r \big(m_r+\sum_{s=1}^{r-1}2m_s\big)
\end{align*}
%
%\item If 
%$T_{rs}=T(A_0^{rs},A_1^{rs},\ldots,A_{a_{rs}-1}^{rs})$, 
%$A_0^{rs},\ldots,A_{a_{rs}-1}^{rs}\in \mathbb{R}^{m_r\times m_s}$ for all $r,s\in \{1,\ldots,N\}$ in (\ref{bT3}) $\mathcal{B}$, $\mathcal{B}'$ are real, then %$\mathcal{Y}\subset \mathbb{C}^{n\times n}$ with $n=\sum_{r=1}^{N}\alpha_r m_r$ is an immersed submanifold of dimension
%\[
%\sum_{r=1}^{N}\big(\tfrac{m_r(m_r-1)}{2}+(\alpha_r-1)\tfrac{m_r(m_r-1)}{2}+\alpha_r \cdot m_r\sum_{s=1}^{r-1}m_s\big)
%\]
\end{enumerate}

\end{lemma}

\begin{proof}
We first observe a few simple facts. Since 
\[
(\mathcal{F}Y^{T}\mathcal{F}\mathcal{B} Y)^{T}=Y^{T}\mathcal{B}^{T}\mathcal{F}Y\mathcal{F} =\mathcal{F}\mathcal{F}Y^{T}\mathcal{F}(\mathcal{F}\mathcal{B}^{T}\mathcal{F})Y\mathcal{F}=\mathcal{F}(\mathcal{F}Y^{T}\mathcal{F}\mathcal{B}Y)\mathcal{F},
\]
it follows that for off-diagonal blocks ($r\neq s$) we have $(\mathcal{F}Y^{T}\mathcal{F}\mathcal{B}Y)_{rs}=0$ if and only if $(\mathcal{F}Y^{T}\mathcal{F}\mathcal{B}Y)_{sr}=0$. When comparing the left-hand side and the right-hand side of (\ref{eqFYFIY}) blockwise, it therefore suffices to observe only the blocks in the upper-triangular parts of $\mathcal{F}Y^{T}\mathcal{F}\mathcal{B} Y$ and $\mathcal{B}'$.  
Further, using Lemma \ref{lemaTop} (\ref{lemaTop2}) we deduce that  blocks of the right-hand side of the equation (\ref{eqFYFIY}) are again a block upper block triangular Toeplitz matrices. Hence it is sufficient to compare only the first rows of blocks in (\ref{eqFYFIY}).

For the sake of clarity we briefly sketch the algorithmic procedure (in several steps) how to solve the matrix equation (\ref{eqFYFIY}); details will be provided later on. Let us write its solution $\mathcal{Y}$ as a $N\times N$ block matrix whoose blocks are of the form (\ref{0T0}), 
where $\mathcal{T}_{rs}=T(A_0^{rs},A_1^{rs},\ldots,A_{b_{rs}-1}^{rs})$ (or $\mathcal{T}_{rs}=T_c(A_0^{rs},A_1^{rs},\ldots,A_{b_{rs}-1}^{rs})$) with $A_0^{rs},\ldots,A_{b_{rs}-1}^{rs}\in \mathbb{C}^{m_r\times m_s}$.
We outline the inductive 
procedure to obtain the matrix $\mathcal{Y}$.
The order of calculating the entries of $\mathcal{Y}$ is the following:  \\

\noindent
{\bf STEP a.a} (if $N\geq 2$) $A_j^{rs}$, $r>s$, $r,s\in \{1,\ldots,N\}$, $j\in \{0,\ldots,b_{rs}\}$.\\
{\bf STEP 0.0.} $A_0^{rr}$, $r\in \{1,\ldots,N\}$.\\
{\bf STEP 0.1.} $A_0^{r(r+1)}$, $r\in \{1,\ldots,N-1\}$.\\
{\bf STEP 0.2.} $A_0^{r(r+2)}$, $r\in \{1,\ldots,N-2\}$,\\
\ldots\\
{\bf STEP 0.p.} $A_0^{r(r+p)}$, $r\in \{1,\ldots,N-p\}$,\\
\ldots\\
{\bf STEP 0.N.} $A_0^{1,N}$,\\
{\bf STEP 1.0.} $A_1^{rr}$ for all $r\in \{1,\ldots,N\}$ such that $\alpha_r\geq 2$,\\
{\bf STEP 1.1.} $A_1^{r(r+1)}$ for all $r\in \{1,\ldots,N-1\}$ such that $\alpha_r\geq 2$,\\
{\bf STEP 1.2.} $A_1^{r(r+2)}$ for all $r\in \{1,\ldots,N-2\}$ such that $\alpha_r\geq 2$,\\
\ldots\\
{\bf STEP 1.p.} $A_1^{r(r+p)}$ for all $r\in \{1,\ldots,N-p\}$ such that $\alpha_r\geq 2$,\\
\ldots\\
{\bf STEP 1.N-1.} $A_1^{1N}$ if $\alpha_r\geq 2$,\\
\ldots\\
\ldots\\
{\bf STEP n.0.} $A_n^{rr}$ for all $r\in \{1,\ldots,N\}$ such that $\alpha_r\geq n+1$,\\
{\bf STEP n.1.} $A_n^{r(r+1)}$ for all $r\in \{1,\ldots,N-1\}$ such that $\alpha_r\geq n+1$,\\
{\bf STEP n.2.} $A_n^{r(r+2)}$ for all $r\in \{1,\ldots,N-2\}$ such that $\alpha_r\geq n+1$,\\
\ldots\\
{\bf STEP n.p.} $A_n^{r(r+p)}$ for all $r\in \{1,\ldots,N-p\}$ such that $\alpha_r\geq n+1$,\\
\ldots\\
{\bf STEP n.N-1.} $A_n^{1N}$ if $\alpha_r\geq n+1$,\\
\ldots\\
\ldots\\
{\bf STEP $\alpha_1-1$.0.} $A_{\alpha_1}^{11}$,

\quad
Observe that STEP a.a provides all blocks below the main diagonal of the block matrix $\mathcal{Y}=[\mathcal{Y}]_{r,s=1}^{N}$.
Next, we compute the diagonal entries of the main diagonal blocks of $\mathcal{Y}$ (see STEP 0.0). STEP 0.1 yields the diagonal entries of the first upper off-diagonal blocks of $\mathcal{Y}$. Further, STEP 0.2 gives the diagonal entries of the second upper off-diagonal blocks of $\mathcal{Y}$, and soforth. Alltogether, STEPS 0.0-0.N-1 provide the diagonal entries of the blocks in the upper triangular part of $\mathcal{Y}$. In the same fashion STEPS 1.0,1.1,1.2,\ldots,1.N-1 give the entries on the first upper off-diagonal of the blocks on the main diagonal, on the first upper off-diagonal, on the second upper off-diagonal, \ldots, on the (N-1)-th (the last) upper off-diagonal of $\mathcal{Y}$, respectively. Likewise, STEPS n.0-n.N-1 yield the entries on the n-th upper off-diagonal of each block of $[Y]_{rs}$, $s\geq r$. Finally, we  compute $A_{\alpha_1-1}^{11}$.
%Moreover, STEPS 0.0-n.p give all up to possibly n-th column of blocks %$[Y]_{rs}$ for $s=r,\ldots,r+p$.
We shall show that the equation (\ref{eqFYFIY}) is solvable if and only if $A_0^{rr}$, $r\in \{1,\ldots,N\}$ in STEP 0.0. can be computed.

In the continuation we explain this process in detail.
We set $\mathcal{S}=\mathcal{B}\mathcal{Y}$ and $\widetilde{\mathcal{Y}}=\mathcal{F}Y^{T}\mathcal{F}$. 
The entries in the $j$-th column (and in the first row) of the block $(\widetilde{\mathcal{Y}}\mathcal{S})_{rs}$ 
%in the on the main diagonal of $\widetilde{Y}\mathcal{I} Y$; these 
are obtained by multiplying the first rows of the blocks $\widetilde{Y}_{r1},\ldots, \widetilde{Y}_{rN}$ with the $j$-th columns of the blocks $(\mathcal{S})_{1s},\ldots, (\mathcal{S})_{Ns}$, respectively, and then adding them:
\begin{equation}\label{YSrs1j}
((\widetilde{\mathcal{Y}}\mathcal{S})_{rs})_{1j}=\sum_{k=1}^{N}(\widetilde{Y}_{rk})_{(1)}(\mathcal{S}_{ks})^{(j)}, \qquad r,s\in \{1,\ldots,N\}, \quad j\in \{1,\ldots,\alpha_s\}.
\end{equation}
As mentioned above it suffices to analyse the upper-triangular blocks in the matrix equality (\ref{eqFYFIY}). Therefore assume $r\leq s$ with $r=s-p$,  $p\in \{0,1,2,\ldots,s-1\}$. We get
\begin{align*}%\label{}
((\widetilde{Y}\mathcal{S})_{(s-p)s})_{1j}
=  \sum_{k=1}^{N}(\widetilde{Y}_{(s-p)k})_{(1)}(\mathcal{S}_{ks})^{(j)}, \qquad 1\leq j\leq \alpha_s,\quad 1\leq s \leq N
%=\sum_{k=1}^{s+j-1}(\widetilde{Y}_{(s-p)k})_{(1)}(\mathcal{S}_{ks})^{(j)}
%= & (\widetilde{Y}_{(s-p)s})_{(1)}(\mathcal{S}_{ss})^{(j)}+\sum_{k=s-j+1,k\neq s}^{s+j-1+p}\delta_{(s-p)s}(\widetilde{Y}_{(s-p)k})_{(1)}(\mathcal{S}_{ks})^{(j)}\nonumber
\end{align*}
When $N=1$ (hence $s=1$, $p=0$) we have
\begin{align}\label{f00}
((\widetilde{\mathcal{Y}}\mathcal{S})_{11})_{1j}
 =
% \left\{
% \begin{array}{ll}
 (\widetilde{Y}_{11})_{(1)}((\mathcal{S})_{11})^{(j)},
 %& N=1\\
 %(\widetilde{Y}_{11})_{(1)}((\mathcal{S})_{11})^{(j)}+\sum_{k=2}^{N}(\widetilde{Y}_{1k})_{(1)}((\mathcal{S})_{k1})^{(j)}, & N\geq 2
 %\end{array}
 %\right.
%=\sum_{k=s}^{s+p}(\widetilde{Y}_{(s-p)k})_{(1)}((\mathcal{S})_{ks})^{(1)}\\
%&=\sum_{k=0}^{p}(A_0^{(s-k)(s-p)})^{T}B_0^{s-k}A_0^{(s-k)s}
\end{align}
while for $N\geq 2$ (and $0\leq p\leq s-1$) we obtain 
\begin{align}\label{YSp1}
((\widetilde{\mathcal{Y}}\mathcal{S})_{1s})_{1j}= & (\widetilde{Y}_{11})_{(1)}(\mathcal{S}_{1s})^{(j)}+\sum_{k=2}^{N}(\widetilde{Y}_{1k})_{(1)}(\mathcal{S}_{ks})^{(j)}, \qquad (p=s-1)
\end{align}
\begin{align}\label{f1}
((\widetilde{\mathcal{Y}}\mathcal{S})_{(s-p)s})_{1j}= & (\widetilde{Y}_{(s-p)(s-p)})_{(1)}(\mathcal{S}_{(s-p)s})^{(j)}+\sum_{k=s-p+1}^{N}(\widetilde{Y}_{(s-p)k})_{(1)}(\mathcal{S}_{ks})^{(j)}\\
  &+\sum_{k=1}^{s-p-1}(\widetilde{Y}_{(s-p)k})_{(1)}(\mathcal{S}_{ks})^{(j)}, \qquad s\geq 2, s-N+1\leq p\leq s-2,  \nonumber\\
  %+\sum_{k=s-j+1}^{s-p}\delta_{(s-p)s}(\widetilde{Y}_{(s-p)k})_{(1)}(\mathcal{S}_{ks})^{(j)}, \qquad (s\geq 2, \quad 0\leq p\leq s-2)
%&=\sum_{k=0}^{p}(A_0^{(s-k)(s-p)})^{T}B_0^{s-k}A_0^{(s-k)s}.
%\end{align}
%\[
\label{fN}
((\widetilde{Y}\mathcal{S})_{NN})_{1j}
= & (\widetilde{Y}_{NN})_{(1)}(\mathcal{S}_{NN})^{(j)}+ \sum_{k=1}^{N-1}(\widetilde{Y}_{Nk})_{(1)}(\mathcal{S}_{kN})^{(j)}, \qquad (p=0,s=N).
%\]
\end{align}
We now split our consideration into two cases
\begin{enumerate}[wide=0pt,itemsep=15pt]
\item[{\bf Case \ref{EqT1}.}] Let $\mathcal{Y}$ be a block matrix and its blocks are complex upper-triangular Toeplitz matrices. 
Note that when $\mathcal{B},\mathcal{B}'$ are real we can use the same procedure to solve (\ref{eqFYFIY}) for $\mathcal{Y}$ real.

\quad
Since
\begin{align*}
&E_{a_{rs}}(I_{m_s})\big(T(A_0,A_1,\ldots,A_{b_{rs}-1})\big)^{T}E_{a_{rs}}(I_{m_r})=T(A_0^{T},A_1^{T},\ldots,A^{T}_{b_{rs}-1}),
%&E_{a_{rs}}(I_{m_s}) \big(T_A(A_0,A_1,\ldots,A_{b_{rs}-1})\big)^{T}E_{a_{rs}}(I_{m_r})=\left\{
%\begin{array}{ll}
%T_A(\overline{A}_0^{T},A_1^{T},\overline{A}^{T}_2,\ldots), & a_{rs} \textrm{  even }\\
%T_A(A_0^{T},\overline{A}^{T}_1,A_2^{T},\ldots), & a_{rs} \textrm{  odd }
%\end{array}\right.,
\end{align*}
it follows that blocks of $[\widetilde{\mathcal{Y}}_{rs}]_{rs}=\widetilde{\mathcal{Y}}=\mathcal{F}\mathcal{Y}^{T}\mathcal{F}$ are of the form  
\[
\widetilde{\mathcal{Y}}_{rs}=
E_{\alpha_r}(I_{m_r})\mathcal{Y}_{rs}^{T} E_{\alpha_s}(I_{m_s})=
%(E_{m_r}(I_{\alpha_r})Y_{rs} E_{m_s}(I_{\alpha_s}))^{T}=
\left\{
\begin{array}{cc}
\begin{bsmallmatrix}
\widetilde{\mathcal{T}}_{rs}\\
0
\end{bsmallmatrix}, 
& \alpha_r >\alpha_s \\
\begin{bsmallmatrix}
0 & \widetilde{\mathcal{T}}_{rs}
\end{bsmallmatrix}, & \alpha_r<\alpha_s\\
\widetilde{\mathcal{T}}_{rs}, & \alpha_r=\alpha_s
\end{array}
\right., \quad
\widetilde{T}_{rs}=
T\big((A_0^{sr})^{T},(A_1^{sr})^{T},\ldots,(A_{b_{rs}}^{sr})^{T}\big)
\]
We then have 
\[
(\widetilde{Y}_{rk})_{(1)}=\left\{
\begin{array}{ll}
\begin{bmatrix}
(A_0^{kr})^{T} & (A_1^{kr})^{T} & \ldots & (A_{a_{kr}-1}^{kr})^{T} 
\end{bmatrix}, & \alpha_k\leq \alpha_r\\
\begin{bmatrix}
0_{m_r\times m_k(\alpha_k-\alpha_r)} & (A_0^{kr})^{T} & \ldots & (A_{\alpha_r-1}^{kr})^{T}
\end{bmatrix}, & \alpha_k > \alpha_r
\end{array}
\right..
\]

Denote
\begin{align*}
&\mathcal{S}_{rs}=
\left\{
\begin{array}{cc}
\begin{bmatrix}
S_{rs}\\
0
\end{bmatrix}, 
& \alpha_r >\alpha_s \\
\begin{bmatrix}
0 & S_{rs}
\end{bmatrix}, & \alpha_r<\alpha_s\\
S_{rs}, & \alpha_r=\alpha_s
\end{array}
\right., \quad
S_{rs}=T\big(C_0^{rs},C_1^{rs},\ldots, C_{b_{rs}-1}^{rs}\big),\quad C_0^{r},\ldots, C_{b_{rs}-1}^{r}\in \mathbb{C}^{m_r\times m_r},
%C_{k}^{rs}=\sum_{j=0}^{k} B_{k-j}^{r}A_j^{rs}.%\\
\end{align*}
and observe for $1\leq j\leq \alpha_s$ that:
\begin{equation}\label{Sks}
(\mathcal{S}_{ks})^{(j)}=
\left\{
\begin{array}{ll}
\begin{bsmallmatrix}
C_{\alpha_k-1}^{ks} \\
\ldots \\
C_0^{ks} 
\end{bsmallmatrix}, & j=\alpha_s \geq \alpha_k\\
\begin{bsmallmatrix}
C_{j-1}^{ks} \\
\ldots \\
C_0^{ks}\\
0\\
\ldots\\
0
\end{bsmallmatrix}, &  \alpha_s< \alpha_k  \textrm{ or } j<\alpha_s= \alpha_k \\
\begin{bsmallmatrix}
C_{j-\alpha_s+\alpha_k-1}^{ks} \\
\ldots \\
C_{0}^{ks} \\
0\\
\ldots\\
0
\end{bsmallmatrix}, & \alpha_s>j>\alpha_s-\alpha_k\geq 0 \\
0, & \alpha_s- \alpha_k\geq j  
\end{array}
\right..
\end{equation}

\quad
We set
\begin{equation}\label{Prsk}
P^{rsk}_{n}=
\begin{bmatrix}
(A_0^{kr})^{T} & (A_1^{kr})^{T} & \ldots & (A_{n}^{rr})^{T} 
\end{bmatrix}
\begin{bmatrix}
C_{n}^{ks} \\
\ldots \\
C_{0}^{ks} \\
\end{bmatrix}=
\sum_{j=0}^{n}(A_j^{kr})^{T}C_{n-j}^{ks} .
\end{equation}
Since $C_{k}^{rs}=\sum_{j=0}^{k} B_{k-j}^{r}A_j^{rs}$ we compute
\begin{align*}
P^{rsk}_{n}=\sum_{j=0}^{n}(C_j^{kr})^{T}A_{n-j}^{ks} 
&=\sum_{j=0}^{n}\sum_{l=0}^{j}(A_l^{kr})^{T}( B_{j-l}^{r})^{T}A_{n-j}^{ks}
=\sum_{l=0}^{n}\sum_{j=l}^{n}(A_l^{kr})^{T}( B_{j-l}^{r})^{T}A_{n-j}^{ks}\\
&=\sum_{l=0}^{n}\sum_{j'=0}^{n-l}(A_{l}^{kr})^{T}( B_{j'}^{r})^{T}A_{n-l-j'}^{ks}
=\sum_{l=0}^{n}(A_l^{kr})^{T}C_{n-l}^{ks}=(P^{rsk}_{n})^{T} ,
\end{align*}
thus is follows that $P^{rsk}_{n}$ is symmetric.
Using (\ref{Sks}) we have
\begin{equation}\label{YSP}
(\widetilde{Y}_{rk})_{(1)}(\mathcal{S}_{ks})^{(j)}=
\left\{
\begin{array}{ll}
P^{rsk}_{j-\alpha_s+\alpha_r-1}, & \alpha_s \geq \alpha_k\geq \alpha_r,j>\alpha_s-\alpha_k\\
P^{rsk}_{j-1}, &  \alpha_r\geq  \alpha_k \geq \alpha_s\\
P^{rsk}_{j-\alpha_k+\alpha_r-1}, & \alpha_k\geq \alpha_s,\alpha_r \\
P^{rsk}_{j-\alpha_s+\alpha_k-1}, & \alpha_s,\alpha_r \geq \alpha_k, j>\alpha_s-\alpha_k \\
0, & \textrm{otherwise}
\end{array}
\right.,
\end{equation}
which is a symmetric matrix.

\quad
In STEP a.a we fix arbitrarily the blocks below the main diagonal of the block matrix 
$[\mathcal{Y}]_{r,s=1}^{N}$ (hence the blocks above the principal 
diagonal of the block matrix $[\mathcal{F}\mathcal{Y}^{T}\mathcal{F}]_{r,s=1}^{N}$). This adds $\sum_{r=1}^{N}\alpha_r \cdot m_r\sum_{s=1}^{r-1}m_s$ to complex (real) dimension of the solution; each block $[\mathcal{Y}]_{r,s}$, $r>s$ gives $(\alpha_s-1)m_r m_s$. \\

\noindent

\quad
Secondly, we compute the matrices in STEP 0.0. Since
\[
(\widetilde{Y}_{rk})_{(1)}=\left\{
\begin{array}{ll}
\begin{bmatrix}
(A_0^{rr})^{T} & *
\end{bmatrix}, & k=r\\
\begin{bmatrix}
0 & *
\end{bmatrix}, & k<r
\end{array}
\right.,\qquad
((\mathcal{S})_{kr})^{(1)}=
\left\{
\begin{array}{ll}
\begin{bmatrix}
B_0^{k}A_0^{kr} \\
0
\end{bmatrix}, & k\leq r\\
0, & k>r
\end{array}
\right..
\]
we then get from (\ref{YSrs1j}) for $r=s$, $j=1$ that
\begin{align*}
((\widetilde{Y}\mathcal{S})_{rr})_{11}=\sum_{k=1}^{N}(\widetilde{Y}_{rk})_{(1)}((\mathcal{S})_{kr})^{(1)}=(\widetilde{Y}_{rr})_{(1)}((\mathcal{S})_{rr})^{(1)}=(A_0^{rr})^{T}B_0^{r}A_0^{rr}, r\in \{1,\ldots,N\}.
\end{align*}
Therefore (\ref{eqFYFIY}) yields
\begin{equation}\label{GABA}
G_0^{r}=(A_0^{rr})^{T}B_0^{r}A_0^{rr}, \qquad r\in \{1,\ldots,N\}.
\end{equation}
By Autonne-Takagi factorization (see e.g. \cite[Corolarry
4.4.4]{HornJohn}), any nonsingular complex symmetric matrix is
%unitary $T$-congruent to a diagonal matrix with non-negative diagonal entries, hence 
$T$-congruent to the identity-matrix. Since $B_0^{r},G_0^{r}$ are symmetric, we have $B_0^{r}=H_{r}^{T}H_{r}$, $G_0^{r}=(H_{r}')^{T}H_{r}'$, hence (\ref{GABA}) yields 
\[
I=\big(H_{r}^{T}A_0^{rr}(H_{r}')^{-1}\big)^{T}\big(H_{r}^{T}A_0^{rr}(H_{r}')^{-1}\big).
\]
If $B_0^{r},G_0^{r}$ are real, 
then by Sylvester's theorem the equation (\ref{GABA}) has a real solution for $A_0^{rr}$ precisely when $B_0^{r},G_0^{r}$ are of the same inertia. In the complex case (real case) the complex (real) dimension of the (possible) solution is $\frac{m_r(m_r-1)}{2}$.

\quad
Next, let us compute $A_j^{r(r+p)}$ for all $N-1\geq p\geq 0$, $r\in \{1,\ldots,N-p\}$ such that $\alpha_r\geq j+1$ (STEP j.p.), while assuming that all matrices from the previous steps are already determined ($A_{j}^{rs}$ for $r>s$ and any $j$, $A_{j'}^{r(r+p')}$ for either $j'<j$ or possibly $j'=j$, $p'<p$ (if $p\geq 1$)). 
We use formulas (\ref{f00}), (\ref{YSp1}), (\ref{f1}), (\ref{fN}).

\quad
Observe that the second term in (\ref{YSp1}) and (\ref{f1}) for $j$ replaced with $j+1$ consist of summands $(\widetilde{Y}_{(s-p)k})_{(1)}(\mathcal{S}_{ks})^{(j+1)}$ with $k\geq s-p+1$. Here, $(\widetilde{Y}_{(s-p)k})_{(1)}$ for $k\geq s-p+1$ is the first row of a matrix above the main diagonal of $\widetilde{Y}$ (it was arranged in STEP a.a). Next, $(\mathcal{S}_{ks})^{(j+1)}$ for $k\geq s-p+1$ is the $j$-th column of a block in the $s$-th column of $\mathcal{S}$ (under $\mathcal{S}_{(s-p)s}$) and it (possibly) depends only on $A_{j}^{r(r+p')}$ for $0\leq p'<p$ (and $p\geq 1$) and $A_{j'}^{rs}$ for $r>s$, which have already been determined before STEP j.p; recall that $C_{k}^{rs}=\sum_{l=0}^{k} B_{k-l}^{r}A_l^{rs}$.

\quad
Further, the third term in (\ref{f1}) and the second term in (\ref{fN}) contain of summands which are products of matrices 
%$(\widetilde{Y}_{(s-p)k})_{(1)}$ and $(\mathcal{S}_{ks})^{(j)}$ for %$1\leq k\leq s-p-1$. Now, these are 
%of the form:
\[
(\widetilde{Y}_{(s-p)k})_{(1)}=
\begin{bsmallmatrix}
0& \ldots & 0 &(A_0^{k(s-p)})^{T}  & \ldots &(A_{b_{(s-p)}}^{k(s-p)})^{T} 
\end{bsmallmatrix},\quad
((\mathcal{S})_{ks})^{(j+1)}=
\begin{bsmallmatrix}
C_j^{ks} \\
\ldots\\
C_0^{ks}\\
0\\
\ldots\\
0
\end{bsmallmatrix}, \quad 1\leq k\leq s-p-1,
\]
so they (possibly) depend on $A_{j'}^{ks}$ for $j'<j$, $1\leq k\leq s-p-1$ (note $C_{j'}^{s}=\sum_{l=0}^{j'} B_{j'-l}^{r}A_l^{ks}$) and $A_{j'}^{k(s-p)}$ for $j'<j$, $1\leq k\leq s-p-1$. Again, these matrices were already determined in the previous steps.

\quad
Finally the first term of (\ref{f00}), (\ref{YSp1}), (\ref{f1}), (\ref{fN}) is a matrix product of ($r=s-p$):
\begin{align*}
&(\widetilde{Y}_{rr})_{(1)}=
\begin{bsmallmatrix}
(A_0^{rr})^{T} &  \ldots & (A_{\alpha_{r}-1}^{rr})^{T}
\end{bsmallmatrix},\\ 
&(\mathcal{S}_{rr})^{(\alpha_r-1)}=
\begin{bsmallmatrix}
C_{\alpha_r-1}^{rr} \\
\ldots\\
C_0^{rr}
\end{bsmallmatrix},\quad
(\mathcal{S}_{r(r+p)})^{(j+1)}=
\begin{bsmallmatrix}
C_j^{r(r+p)} \\
\ldots\\
C_0^{r(r+p)}\\
0\\
\ldots\\
0
\end{bsmallmatrix},\quad  j< \alpha_s-1 \textrm{ or } p\geq 1.
\end{align*}
Therefore
\begin{align*}
&(\widetilde{Y}_{rr})_{(1)}(\mathcal{S})_{rr}^{(j+1)}=(A_0^{rr})^{T}B_0^{r}A_j^{rr}+(A_j^{rr})^{T}B_0^{r}A_0^{rr}+\Xi(j,r,0),\\
&(\widetilde{Y}_{rr})_{(1)}(\mathcal{S})_{r(r+p)}^{(j+1)}=(A_0^{rr})^{T}B_0^{r}A_j^{r(r+p)}+\Xi(j,r,p),\qquad p\geq 1,
\end{align*}
where $\Xi(j,r,p)$ is symmetric (since $P_j^{rks}$ in (\ref{YSP}) is symmetric) and for $p\geq 1$ and $p=0$ depends only on $A_{j'}^{r(r+p)}$ with $j'<j$ and $A_{j'}^{rr}$ with $j'< j$, respectively. Remember that $C_j^{rs}=\sum_{l=0}B_{j-l}^{r}A_l^{rs}$.

\quad
By comparing the entries in the first row and $(j+1)$-th column of the blocks in the $s-p$-th row and $s$-th column of the matrices in left-hand side and right-hand side of the equation (\ref{eqFYFIY}) we obtain (recall also (\ref{f00}), (\ref{YSp1}), (\ref{f1}), (\ref{fN})):
\[
(A_0^{rr})^{T}A_j^{r(s+p)}=\Xi '(j,r,p), \qquad p\geq 1
\]
\[
(A_0^{rr})^{T}A_j^{rr}+(A_j^{rr})^{T}A_0^{rr}=\Xi '(j,r,0), \qquad p=0\quad  (s=r),
\]
where $\Xi'(j,s,p)$ is symmetric (see (\ref{YSP})) and depends only on $A_{j'}^{rs}$ for either $j'<j$ or $r>s$ or (possibly for $p\geq 1$) $j'=j$, $s=r+p'$ with $p'<p$. (Note that $A_0^{rr}$, $A_0^{r(r+p)}$ are invertible.)

\quad
To get $A_j^{r(r+p)}$ for $p\geq 1$ one needs to solve a simple matrix equation of the form $A^{T}X=B$ on $X$ with given $A\in GL_n(\mathbb{C})$ and $B\in \mathbb{C}^{n\times n}$ ($A\in GL_n(\mathbb{R})$, $B\in \mathbb{R}^{n\times n}$), while to get $A_j^{rr}$ we need to solve the equation of the form $A^{T}X+X^{T}A=B$ on $X$ with given $A\in GL_n(\mathbb{C})$, $B\in \mathbb{C}^{n\times n}_S$ ($A\in GL_n(\mathbb{R})$, $B\in \mathbb{R}^{n\times n}_S$); the solution is $X=\frac{1}{2}(A^{T})^{-1}B+(A^{T})^{-1}C$, where $C$ is any matrix with $C^{T}=-C$. Since the map $C\mapsto (A^{T})^{-1}C$ is biholomorphic (diffeomorphic), the complex (real) dimension of the solution is $\frac{n(n-1)}{2}$.

\quad
To prove the first statement of (\ref{EqT1}) and (\ref{EqT1a}) it remains to sum up dimensions:
\[
\sum_{r=1}^{N}\big(\alpha_r \cdot m_r\big(\sum_{s=1}^{r-1}m_s\big)+\tfrac{m_r(m_r-1)}{2}+(\alpha_r-1)\tfrac{m_r(m_r-1)}{2}\big)=
\sum_{r=1}^{N}\alpha_r m_r\big(\tfrac{1}{2}(m_r-1)+\sum_{s=1}^{r-1}m_s\big).
\]

\quad
Next, to prove (\ref{EqT1b}) we assume (\ref{aass}). We need to see that the solution $\mathcal{Y}$ obtained above is of the right form as announced in the statement of the lemma.
Using the notation
\[
K_{r}=
\begin{bsmallmatrix}
-\mu^{2}I_{m_r} &  0 \\
0 &  I_{m_r}
\end{bsmallmatrix}, \qquad 
L_{r}=
\begin{bsmallmatrix}
-\mu^{2}I_{m_r} &  0 \\
0 &  I_{m_r}
\end{bsmallmatrix}, 
\qquad r\in \{1,\ldots,N\}
\]
and
\[
\widetilde{A}_n^{rs}=
\begin{bmatrix}
V_n^{rs} &  W_1^{rs} \\
-\mu^{2}\overline{V}_n^{rs} &  \overline{W}_n^{rs} 
\end{bmatrix},\qquad
F_n^{rs}=
\begin{bmatrix}
0 &  0 \\
\overline{W}_n^{rs} & 0 
\end{bmatrix}, \qquad n\in \{0,\ldots,b_{rs}-1\},
\]
(recall that 
$A_n^{rs}=
\begin{bsmallmatrix}
V_n^{rs} &  W_n^{rs} \\
-\mu^{2}\overline{V}_n^{rs}+\overline{V}_{n-1}^{rs} &  \overline{W}_n^{rs} 
\end{bsmallmatrix}
$ in (\ref{aass})),
we write 
\begin{align*}
B_0^{r} & =u_0 K_r,\qquad B_n^{r}=u_n K_r+u_{n-1}, \quad n\in \{0,\ldots,b_{rs}-1\},\\
C_0^{rs} &= u_0K_rA_0^{rs},\qquad C_1^{rs}=K_r(u_0\widetilde{A}_1^{rs}+u_1 A_0^{rs})+u_0K_rF_0^{rs}+u_0L_rA_0^{rs},\\
C_n^{rs} & =\sum_{j=0}^{n} B_{n-j}^{r}A_j^{rs}=K_{r} \sum_{j=0}^{n}u_{n-j}A_j^{rs}+L_{r} \sum_{j=0}^{n-1}u_{n-1-j}A_j^{rs}, \qquad k\geq 2.\\
         & =K_{r} \sum_{j=0}^{n}u_{n-j}\widetilde{A}_j^{rs}+K_{r} \sum_{j=0}^{n-1}u_{n-1-j}F_j^{rs}+L_{r} \sum_{j=0}^{n-1}u_{n-1-j}\widetilde{A}_j^{rs}+L_{r} \sum_{j=0}^{n-2}u_{n-2-j}F_j^{rs},
\end{align*}
To simplify the calculations we further set
\begin{align*}
&D_0^{rs}=u_0 A_0^{rs}, \quad E_0^{rs}=0, \qquad 
%$U_n=u_0(I_{m_r}\oplus I_{m_r})\oplus \ldots \oplus u_n(I_{m_r}\oplus I_{m_r})$  
D_n^{rs}= \sum_{j=0}^{n}u_{n-j}\widetilde{A}_j^{rs},
%$U_n=u_0(I_{m_r}\oplus I_{m_r})\oplus \ldots \oplus u_n(I_{m_r}\oplus I_{m_r})$  
E_n^{rs}= \sum_{j=0}^{n}u_{n-j}F_j^{rs}, \quad n\in \{1,\ldots,b_{rs}\},\\
&\Delta_n^{ks}=\begin{bsmallmatrix}
D_n^{ks} \\
\ldots\\
D_0^{ks}\\
\end{bsmallmatrix},\quad
\varepsilon_n^{ks}=\begin{bsmallmatrix}
E_n^{ks} \\
\ldots\\
E_0^{ks}\\
\end{bsmallmatrix},\quad
\Phi_n^{ks}=\begin{bsmallmatrix}
F_n^{ks} \\
\ldots\\
F_0^{ks}\\
\end{bsmallmatrix},\quad
\widetilde{\Upsilon}_n^{ks}=\begin{bsmallmatrix}
\widetilde{A}_n^{ks} \\
\ldots\\
\widetilde{A}_0^{ks}\\
\end{bsmallmatrix}, \quad n\geq 0.\\
&\mathcal{K}_{r,n}=\oplus_{j=1}^{n} K_{r}, \quad \mathcal{L}_{r,n}=\oplus_{j=1}^{n} L_{r},\qquad  n\in \mathbb{N}, r\in \{1,\ldots,N\}
\end{align*}
It follows that (see (\ref{Prsk}))
\begin{align*}
P^{rsk}_0=& (A_0^{kr})^{T}C_0^{ks}= u_0(A_0^{kr})^{T} K_r A_0^{ks},\\
P^{rsk}_1=&\begin{bsmallmatrix}
(A_0^{kr})^{T} &  (A_{1}^{kr})^{T} 
\end{bsmallmatrix}
\begin{bsmallmatrix}
C_1^{ks} \\
C_0^{ks}
\end{bsmallmatrix} \\
= & (\widetilde{\Upsilon}_1^{ks})^{T}
\mathcal{K}_{r,2}
\Delta_1^{ks}
+(A_0^{kr})^{T}(u_0K_rF_0^{rs}+u_0L_rA_0^{rs})+u_0(F_0^{kr})^{T}K_r A_0^{rs}.
\end{align*}
and for $n\geq 2$:
\begin{align*}
P^{rsk}_n=& \begin{bsmallmatrix}
(A_0^{kr})^{T} & \ldots & (A_{n}^{kr})^{T} 
\end{bsmallmatrix}
\begin{bsmallmatrix}
C_n^{ks} \\
\ldots\\
C_0^{ks}\\
\end{bsmallmatrix}\\
= &
\big((\widetilde{\Upsilon}_n^{ks})^{T} +\begin{bsmallmatrix} 0\\
\Phi_{n-1}^{ks} \end{bsmallmatrix}^{T}\big)\mathcal{K}_{r,n+1}(\Delta_{n}^{ks}+\varepsilon_{n-1}^{ks}   ) 
+\big((\widetilde{\Upsilon}_{n-1}^{ks})^{T} +\begin{bsmallmatrix} 0\\
\Phi_{n-1}^{ks} \end{bsmallmatrix}^{T}\big)\mathcal{L}_{r,n}(\Delta_{n-1}^{ks}+\varepsilon_{n-2}^{ks} )\\
=&\left( (\Phi_{n-1}^{ks})^{T}\mathcal{K}_{r,n}
\Delta_{n-1}^{ks}
+\Upsilon_{n-1}^{ks}I_{n}(K_r)
\varepsilon_{n-1}^{ks}
+(\Upsilon_{n-1}^{ks})^{T}\mathcal{L}_{r,n}
\Delta_{n-1}^{ks} \right)+(\Phi_{n-1}^{ks})^{T} \mathcal{K}_{r,n+1}\varepsilon_{n}^{ks}\\
&+(\widetilde{\Upsilon}_n^{ks})^{T} \mathcal{K}_{r,n+1}\Delta_{n}^{ks}
%&\begin{bsmallmatrix}
%(\widetilde{A}_0^{kr})^{T} & (\widetilde{A}_1^{kr}+F_0^{ks})^{T}& %\ldots & (\widetilde{A}_{n}^{kr}+F_{n-1}^{ks})^{T} 
%\end{bsmallmatrix}I_{n+1}(K_r)
%\begin{bsmallmatrix}
%D_n^{ks}+E_{n-1}^{ks} \\
%\ldots\\
%D_1^{ks}+E_0^{ks}
%\ldots\\
%D_0^{ks}\\
%\end{bsmallmatrix}\\
%&+\begin{bsmallmatrix}
%(\widetilde{A}_0^{kr}+F_{0}^{kr})^{T} & (\widetilde{A}_{1}^{kr}+F_{0}^{kr})^{T}& \ldots & (\widetilde{A}_{n-1}^{kr}+F_{n-2}^{kr})^{T} 
%\end{bsmallmatrix}I_{n}(L_r)
%\begin{bsmallmatrix}
%D_{n-1}^{ks}+E_{n-2}^{ks} \\
%\ldots\\
%D_1^{ks}+E_0^{ks}
%\ldots\\
%D_0^{ks}
%\end{bsmallmatrix}.
\end{align*}
Denoting finally
\[
U_n^{rs}=\sum_{j=0}^{n}u_{n-j}V_j^{rs}, \qquad Z_n^{rs}=\sum_{j=0}^{n}u_{n-j}W_j^{rs}, \qquad (D_n^{rs}=\begin{bmatrix}U_n^{rs} & Z_n^{rs}\\
-\mu^{2}\overline{Z}_n^{rs} & \overline{U}_n^{rs}
\end{bmatrix})
\]
it is straightforward to compute
\begin{align*}
&
(\widetilde{\Upsilon}_n^{ks})^{T} \mathcal{K}_{r,n+1}\Delta_{n}^{ks}
%&\begin{bsmallmatrix}
%(\widetilde{A}_0^{kr})^{T} & (\widetilde{A}_1^{kr})^{T}& \ldots & %(\widetilde{A}_{n}^{kr})^{T} 
%\end{bsmallmatrix}I_{n+1}(K_r)
%\begin{bsmallmatrix}
%D_n^{ks} \\
%\ldots\\
%D_0^{ks}\\
%\end{bsmallmatrix}
=
\sum_{j=0}^{n}\begin{bsmallmatrix}
-\mu^{2}\big((V_{j}^{rs})^{T}\mathcal{V}_{n-j} -\mu^{2}(\overline{W}_{j}^{rs})^{T}\overline{Z}_{n-j}\big) & (V_{j}^{rs})^{T}Z_{n-j} +(\overline{W}_{j}^{rs})^{T}\overline{U}_{n-j}\\
(\overline{V}_{j}^{rs})^{T}\overline{Z}_{n-j} +(W_{j}^{rs})^{T}U_{n-j}   & -\mu^{2}\big((\overline{V}_{j}^{rs})^{T}\overline{U}_{n-j} -\mu^{2}(W_{j}^{rs})^{T}Z_{n-j}\big)
\end{bsmallmatrix},\\
&
(\Phi_{n-2}^{ks})^{T}\mathcal{K}_{r,n}\varepsilon_{n-2}^{ks}
%&\begin{bsmallmatrix}
%(F_{0}^{kr})^{T} & (F_{0}^{kr})^{T}& \ldots & (F_{n-2}^{kr})^{T} 
%\end{bsmallmatrix}I_{n}(L_r)
%\begin{bsmallmatrix}
%E_{n-2}^{ks} \\
%\ldots\\
%E_0^{ks}
%\end{bsmallmatrix}
=
\sum_{j=0}^{n-2}\begin{bsmallmatrix}
\overline{W}_{j}^{rs}\overline{Z}_{n-2-j} & 0 \\
0        & 0
\end{bsmallmatrix}
\end{align*}
and
\begin{align*}
&(\Phi_{n-1}^{ks})^{T}\mathcal{K}_{r,n}
\Delta_{n-1}^{ks}
+\Upsilon_{n-1}^{ks}\mathcal{K}_{r,n}
\varepsilon_{n-1}^{ks}
+(\Upsilon_{n-1}^{ks})^{T}\mathcal{L}_{r,n}
\Delta_{n-1}^{ks}=\\
%&\begin{bsmallmatrix}
%(F_0^{ks})^{T} & \ldots & (F_{n-1}^{ks})^{T} 
%\end{bsmallmatrix}I_{n}(K_r)
%\begin{bsmallmatrix}
%D_{n-1}^{ks} \\
%\ldots\\
%D_0^{ks}\\
%\end{bsmallmatrix}
%+\begin{bsmallmatrix}
%(\widetilde{A}_0^{ks})^{T} & \ldots & (\widetilde{A}_{n-1}^{ks})^{T} 
%\end{bsmallmatrix}I_{n}(K_r)
%\begin{bsmallmatrix}
%E_{n-1}^{ks} \\
%\ldots\\
%E_0^{ks}\\
%\end{bsmallmatrix}
%+\begin{bsmallmatrix}
%(\widetilde{A}_0^{kr})^{T} & \ldots & (\widetilde{A}_{n-1}^{kr})^{T} 
%\end{bsmallmatrix}I_{n}(L_r)
%\begin{bsmallmatrix}
%D_{n-1}^{ks} \\
%\ldots\\
%D_0^{ks}
%\end{bsmallmatrix}\\
=&
\sum_{j=0}^{n-1}\begin{bsmallmatrix}
-\mu^{2}(\overline{W}_{j}^{rs})^{T}\overline{Z}_{n-1-j} & \quad (\overline{W}_{j}^{rs})^{T}\overline{U}_{n-1-j}+(V_{j}^{rs})^{T}Z_{n-1-j} \\
(\overline{V}_{j}^{rs})^{T}\overline{Z}_{n-1-j}+(W_{j}^{rs})^{T}U_{n-1-j}         & \quad (W_{j}^{rs})^{T}Z_{n-1-j}
\end{bsmallmatrix}\\
&+
\sum_{j=0}^{n-1}\begin{bsmallmatrix}
(V_{j}^{rs})^{T}U_{n-1-j} -\mu^{2}(\overline{W}_{j}^{rs})^{T}\overline{Z}_{n-1-j} & 0 \\
0        & 0
\end{bsmallmatrix}
\end{align*}
%
%
%\begin{align*}
%\begin{bmatrix}
%(A_0^{kr})^{T} & \ldots & (A_{n}^{kr})^{T} 
%\end{bmatrix}
%\begin{bmatrix}
%C_n^{ks} \\
%\ldots\\
%C_0^{ks}\\
%0
%\end{bmatrix}
%=
%&\begin{bmatrix}
%(A_0^{kr})^{T} & \ldots & (A_{n-1}^{kr})^{T} 
%\end{bmatrix}I_{n}(K_r)
%\begin{bmatrix}
%D_{n-1}^{ks} \\
%\ldots\\
%D_0^{ks}
%\end{bmatrix}\\
%&+\delta_{n}\begin{bmatrix}
%(A_0^{kr})^{T} & \ldots & (A_{n-1}^{kr})^{T} 
%\end{bmatrix}I_{n}(L_r)
%\begin{bmatrix}
%D_{n-2}^{ks} \\
%\ldots\\
%D_0^{ks}
%\end{bmatrix}, \delta_{n}=\left\{\begin{array}{ll} 1, & n=1\\ 0, & n\geq 2  \end{array}\right..
%\end{align*}
%
We not write
\[
P_{n}^{rsk}=Q_{n}^{rsk}+R_{n}^{rsk},
\]
where $Q_{0}^{rsk}=u_0(A_0^{kr})^{T} K_r D_0^{ks}$, $R_{0}^{rsk}=0$ and
for $n\geq 1$:
\begin{align*}
&Q_{n}^{rsk}
=
(\widetilde{\Upsilon}_{n}^{ks})^{T} \mathcal{K}_{r,n}\Delta_{n}^{ks}
%\begin{bsmallmatrix}
%(\widetilde{A}_0^{kr})^{T} & \ldots & (\widetilde{A}_{n}^{kr})^{T} 
%\end{bsmallmatrix}I_{n}(K_r)
%\begin{bsmallmatrix}
%D_{n}^{ks} \\
%\ldots\\
%D_0^{ks}\\
%\end{bsmallmatrix}
+
\sum_{j=0}^{n-1}\begin{bsmallmatrix}
-\mu^{2}(\overline{W}_{j}^{rs})^{T}\overline{Z}_{n-1-j} & \quad (\overline{W}_{j}^{rs})^{T}\overline{U}_{n-1-j}+(V_{j}^{rs})^{T}Z_{n-1-j} \\
(\overline{V}_{j}^{rs})^{T}\overline{Z}_{n-1-j}+(W_{j}^{rs})^{T}U_{n-1-j}         & \quad (W_{j}^{rs})^{T}Z_{n-1-j}
\end{bsmallmatrix}, \qquad\\
&R_{n}^{rsk}=
\left\{
\begin{array}{ll}
\sum_{j=0}^{n-1}
\begin{bsmallmatrix}
(V_{j}^{rs})^{T}U_{n-1-j} -\mu^{2}(\overline{W}_{j}^{rs})^{T}\overline{Z}_{n-1-j} & 0 \\
0        & 0
\end{bsmallmatrix}+
\sum_{j=0}^{n-2}\begin{bsmallmatrix}
(\overline{W}_{j}^{rs})^{T}\overline{Z}_{n-2-j} & 0 \\
0        & 0
\end{bsmallmatrix}, & n\geq 2\\
\begin{bsmallmatrix}
(V_{0}^{rs})^{T}U_{0} -\mu^{2}(\overline{W}_{0}^{rs})^{T}\overline{Z}_{} & 0 \\
0        & 0
\end{bsmallmatrix}, & n=1
\end{array}
\right..
\end{align*}
Note that we have $-\mu^{2}[R_{n}^{rsk}]_{11}=[Q_{n-1}^{rsk}]_{11}$ for any $n\geq 2$. In view of (\ref{Prsk}) we have
\[
(\widetilde{Y}_{rk})_{(1)}((\mathcal{S})_{ks})^{(j+1)}=\mathcal{Q}_{j}^{rsk}+\mathcal{R}_{j}^{rsk},
\]
where we denoted 
\[
\mathcal{Q}_{j}^{rsk}=
\left\{
\begin{array}{ll}
Q^{rsk}_{j-\alpha_s+\alpha_r-1}, & \alpha_s \geq \alpha_k\geq \alpha_r,j>\alpha_s-\alpha_k\\
Q^{rsk}_{j-1}, &  \alpha_r\geq  \alpha_k \geq \alpha_s\\
Q^{rsk}_{j-\alpha_k+\alpha_r-1}, & \alpha_k\geq \alpha_s,\alpha_r \\
Q^{rsk}_{j-\alpha_s+\alpha_k-1}, & \alpha_s,\alpha_r \geq \alpha_k, j>\alpha_s-\alpha_k \\
0, & \textrm{otherwise}
\end{array}
\right.,
\]
\[
\mathcal{R}_{j}^{rsk}=
\left\{
\begin{array}{ll}
R^{rsk}_{j-\alpha_s+\alpha_r-1}, & \alpha_s \geq \alpha_k\geq \alpha_r,j>\alpha_s-\alpha_k\\
R^{rsk}_{j-1}, &  \alpha_r\geq  \alpha_k \geq \alpha_s\\
R^{rsk}_{j-\alpha_k+\alpha_r-1}, & \alpha_k\geq \alpha_s,\alpha_r \\
R^{rsk}_{j-\alpha_s+\alpha_k-1}, & \alpha_s,\alpha_r \geq \alpha_k, j>\alpha_s-\alpha_k \\
0, & \textrm{otherwise}
\end{array}
\right..
\]
%
%We assume further that $A_l^{rs}$ for $\alpha_s\geq 3$ are block diagonal of the form $\begin{bmatrix}
%V & 0 \\
%0 & \overline{V}
%\end{bmatrix}$ with $V\in \mathbb{C}^{m_r\times m_r}$. 
%Then it is not difficult to see that $-\mu^{2}[R_{n}^{rsk}]_{11}=[Q_{n-1}^{rsk}]_{11}$ for $n\geq 2$, $\alpha_s\geq 3$, hence 
%$-\mu^{2}(\mathcal{L}_{j}^{rsk})_{11}=(\mathcal{K}_{j-1}^{rsk})_{11}$ for $j\geq 2$ and $\alpha_s \geq 3$, $\mathcal{L}_{0}^{rsk}=0$.

\quad
We now prove by induction that 
$\sum_{k=1}^{N} \mathcal{Q}_{j}^{rsk}=\left\{\begin{array}{ll} u_{j}K_r, & r=s\\ 0, & r\neq s \end{array}\right.$ for $j\geq 1$, (hence $\sum_{k=1}^{N} \mathcal{R}_{j}^{rsk}=\left\{\begin{array}{ll} u_{j-1}L_r, & r=s\\ 0, & r\neq s \end{array}\right.$, $j\geq 1$).
It is clear for $j=1$. Suppose that it holds for some $j\geq 1$. We have $\sum_{k=1} (\mathcal{Q}_{j+1}^{rsk}+\mathcal{R}_{j+1}^{rsk})=\mathcal{B}_{j+1}^{rs}=\left\{\begin{array}{ll} u_{j+1}K_r+u_{j}L_r, & r=s\\ 0, & r\neq s \end{array}\right.$. Using $(-\mu^{2}\sum_{k=1}^{N} (\mathcal{R}_{j+1}^{rsk})_{11})_{11}=(\sum_{k=1}^{N} \mathcal{Q}_{j}^{rsk})_{11}=\left\{\begin{array}{ll} -\mu^{2}u_j, & r=s\\ 0, & r\neq s \end{array}\right.$, the claim is deduced for $j+1$.  

%\quad
%Clearly, the same holds also for $\alpha_s\leq 2$. 

\quad
Therefore the equation $\sum_{k=1}^{N}(\widetilde{\mathcal{Y}}_{rk})_{(1)}((\mathcal{B}\mathcal{Y})_{ks})^{(j+1)}=\mathcal{B}_{j}^{rs}$ 
is equivalent to 
\begin{equation}\label{EqK}
\sum_{k=1}^{N} \mathcal{Q}_{j}^{rsk}=  
\left\{\begin{array}{ll} u_{j}K_r, & r=s\\ 0, & r\neq s \end{array}\right..
\end{equation}

%\quad
%Next, we prove that the solution is of the prescribed form. More %precisely, we shall first prove that the above equation reduces to %equations of the form $A^{T}X+X^{T}A=B$ and $A^{T}Y=B$, where we are %given structured matrices $A$, $B$, and then try to find the form od %the solutions.

\quad
Observe that all summands of $Q_{j}^{rsk}$ (hence $\mathcal{Q}_{j}^{rsk}$) are of the form 
$\begin{bsmallmatrix}
-\mu^{2}V &  W\\
\overline{W} & \overline{V}
\end{bsmallmatrix}$, and sums of such terms are clearly of this form, as well.
%$\begin{bsmallmatrix}
%V_1 &  W_1\\
%-\mu^{2}\overline{W}_1 & \overline{V}_1
%\end{bsmallmatrix}K_k
%\begin{bsmallmatrix}
%V_2 &  W_2\\
%-\mu^{2}\overline{W}_2 & \overline{V}_2
%\end{bsmallmatrix}$
Next, we use the equation (\ref{EqK}) to see that $A_0^{rr}$, $\widetilde{A}_0^{r(r+p)}$ for $p\geq 1$, $\widetilde{A}_n^{rr}$, $\widetilde{A}_n^{r(r+p)}$ for $p,n\geq 1$ are of this form, too. We obtain equations of the form $A^{T}K A=K$, $A^{T}KX=B$,
$A^{T}KX+X^{T}KA=B$ and $A^{T}KX=B$, respectively, where $K=\begin{bsmallmatrix}
-\mu^{2}I &  0\\
0 & I
\end{bsmallmatrix}$ and we are given a matrix $B$ of the form 
$\begin{bsmallmatrix}
-\mu^{2}V &  W\\
\overline{W} & \overline{V}
\end{bsmallmatrix}$.
Observe that $A^{T}K A=K$ is equivalent to $(N^{-1}A^{T}N)(N A N^{-1})=I$, $N=\frac{1}{i\mu}I\oplus I$ ($K=N^{2}$) and further $Y^{T}Y=I$, $Y=N A N^{-1}$. Then $A^{T}KX=B$ yields $X=K^{-1}(A^{T})^{-1}B=K^{-1}KAK^{-1}B=AK^{-1}B$, which is of the form $\begin{bsmallmatrix}
V &  W\\
-\mu^{2}\overline{W} & \overline{V}
\end{bsmallmatrix}$, provided that $A$ is of this form.
Next, the solution of $A^{T}KX+X^{T}KA=B$ is $X=\frac{1}{2}((KA)^{T})^{-1}B+((KA)^{T})^{-1}C=\frac{1}{2}(AK^{-1}B+AK^{-1}C$, $C^{T}=-C$. If we want $X$ to be of the right form then $C$ must be of the same form as $B$. Thus $C$ is of the form $\begin{bsmallmatrix}
-\mu^{2}V &  W\\
\overline{W} & \overline{V}
\end{bsmallmatrix}$ with $V_1^{T}=-V_1$, $W^{*}=-W$. Here $C$ adds $(2m_s'(m_2'-1)+m_s')$ to real dimension.

\quad
It is left to sum up dimensions:
%
%\small
\begin{align*}
&\sum_{r=1}^{N}\big(\alpha_r \cdot 4m_r'\sum_{s=1}^{r-1}2m_s'+\tfrac{2m_r'(2m_r'-1)}{2}+(\alpha_r-1)(2m_r'(2m_r'-1)+m'_r)\big)=\\
=&\sum_{r=1}^{N}m_r'\big(4\alpha_r m_r'-\alpha_r-2m_r'+8\alpha_r \sum_{s=1}^{r-1}m_s'\big).
\end{align*}
%
%\normalsize
This concludes the proof of (\ref{EqT1}).

\item[{\bf Case \ref{EqT2}.}] Suppose now that $\mathcal{Y}$ is a block matrix and its blocks are upper-triangular complex-alternating Toeplitz matrices. 

\quad
We have $\mathcal{Y}$ in (\ref{0T0}) for $\mathcal{T}_{rs}=T_c(A_0^{rs},A_1^{rs},\ldots,A^{rs}_{b_{rs}-1})$ with $A_0^{rs},\ldots,A_{b_{rs}-1}^{rs}\in \mathbb{C}^{m_r\times m_s}$,  $b_{rs}=\min\{\alpha_r, \alpha_s\}$.
%
%and $\widetilde{T}_c(A_0,A_1,\ldots,A_{m-1})=
%E_m T_c(A_0,A_1,\ldots,A_{m-1})E_m$, respectively.
Note that $I_{m_s}A_0^{T}I_{m_r}=A_0^{T}$ and for $b_{rs}\geq 2$ we get
\[
E_{b_{rs}}(I_{m_s})\big(T_c(A_0,A_1,\ldots,A_{b_{rs}-1})\big)^{T}E_{b_{rs}}(I_{m_r})=\left\{
\begin{array}{ll}
T_c(\overline{A}_0^{T},A_1^{T},\ldots,
\overline{A}_{b_{rs}-2}^T,A_{b_{rs}-1}^T), & a_{rs} \textrm{ even}\\
T_c(A_0^{T},\overline{A}_1^{T},\ldots,
\overline{A}_{b_{rs}-2}^{T},A_{b_{rs}-1}^T), & a_{rs} \textrm{ odd}
\end{array}\right.
%\quad m=\min\{\alpha_r,\alpha_s\},
\]
Here the entry in the first row and the $j$-th column of $T_c(\overline{A}_0^{T},A_1^{T},\ldots,
\overline{A}_{b_{rs}-2}^T,A_{b_{rs}-1}^T)$ (or $T_c(A_0^{T},\overline{A}_1^{T},\ldots,
\overline{A}_{b_{rs}-2}^{T},A_{b_{rs}-1}^T)$) is $A_{j-1}$ for $j$ odd (even) and $\overline{A}_{j-1}$ for $j$ even (odd).
Thus blocks of $[\widetilde{\mathcal{Y}}_{rs}]_{rs}=\widetilde{\mathcal{Y}}=\mathcal{F}\mathcal{Y}^{T}\mathcal{F}$ are of the form  
\begin{align*}
&\widetilde{\mathcal{Y}}_{rs}=
E_{\alpha_r}(I_{m_r})\mathcal{Y}_{rs}^{T} E_{\alpha_s}(I_{m_s})=
%(E_{m_r}(I_{\alpha_r})Y_{rs} E_{m_s}(I_{\alpha_s}))^{T}=
\left\{
\begin{array}{cc}
\begin{bmatrix}
\widetilde{\mathcal{T}}_{rs}\\
0
\end{bmatrix}, 
& \alpha_r >\alpha_s \\
\begin{bmatrix}
0 & \widetilde{\mathcal{T}}_{rs}
\end{bmatrix}, & \alpha_r<\alpha_s\\
\widetilde{\mathcal{T}}_{rs}, & \alpha_r=\alpha_s
\end{array}
\right., \quad\\
&\widetilde{\mathcal{T}}_{rs}=
\left\{
\begin{array}{ll}
T_c\big((\overline{A}_0^{sr})^{T},(A_1^{sr})^{T},\ldots,(\overline{A}_{b_{rs}-2}^{sr})^{T},(A_{b_{rs}-1}^{sr})^{T}\big), & a_{rs} \textrm{ even}\\
T_c\big((A_0^{sr})^{T},(\overline{A}_1^{sr})^{T},\ldots,(\overline{A}_{b_{rs}-2}^{sr})^{T},(A_{b_{rs}-1}^{sr})^{T}\big), & a_{rs} \textrm{ odd}
\end{array}\right..
\end{align*}

Now we have
\begin{align*}
&\mathcal{S}_{rs}=
\left\{
\begin{array}{cc}
\begin{bmatrix}
S_{rs}\\
0
\end{bmatrix}, 
& m_r >m_s \\
\begin{bmatrix}
0 & S_{rs}
\end{bmatrix}, & m_r<m_s\\
S_{rs}, & m_r=m_s
\end{array}
\right., \quad
S_{rs}=T_c\big(C_0^{rs},C_1^{rs},\ldots, C_{n_{rs}-1}^{rs}\big), \, C_0^{r},\ldots, C_{n_{rs}-1}^{r}\in \mathbb{C}^{m_r\times m_r},
%C_{k}^{rs}=\sum_{j=0}^{k} B_{k-j}^{r}A_j^{rs}.%\\
\end{align*}
where $C_0^{rs}=B_0^{rs}$ and for $n\geq 1$:
\begin{align*}
&C_{2n}^{rs}
=B_0^rA_{2n}^{rs}+B_1^r\overline{A}_{2n-1}^{rs}+\ldots+B_{2n-1}^r\overline{A}_{1}^{rs}+B_{2n}^rA_{0}^{rs}
=\sum_{j=0}^{n}B_{2n-2j}^{r}A_{2j}^{ks}+\sum_{j=0}^{n-1}B_{2n-2j-1}^{r}\overline{A}_{2j+1}^{ks}\\
&C_{2n+1}^{rs}
=B_0^rA_{2n+1}^{rs}+B_1^r\overline{A}_{2n}^{rs}+\ldots+B_{2n}^rA_{1}^{rs}+B_{2n}^r \overline{A}_{0}^{rs}
=\sum_{j=0}^{n}B_{2n-2j}^{r}A_{2j+1}^{ks}+\sum_{j=0}^{n}B_{2n-2j+1}^{r}\overline{A}_{2j}^{ks}.
\end{align*}

\quad
To simplify the computation we introduce the following notation. Given $A=\begin{bsmallmatrix} A_0 \\
A_1 \\ \ldots \\A_{\beta-1}\end{bsmallmatrix}$, we set $A_{ac}=\begin{bsmallmatrix} B_0 \\
B_1 \\ \ldots \\B_{\beta-1}\end{bsmallmatrix}$ with $B_{2n}=A_{2n}$, $B_{2n-1}=\overline{A}_{2n-1}$ for $0\leq 2n,2n-1\leq \beta-1$. We denote $\Upsilon^{kr}_0=(\Upsilon^{kr}_0)_{ac}=A_0^{kr}$ and for $n\geq 1$
\begin{align}\label{ac}
\Upsilon^{kr}_n=\begin{bsmallmatrix}
A_0^{kr} \\ A_1^{kr} \\ \ldots  \\ A_{n}^{kr} 
\end{bsmallmatrix}, \quad \Gamma^{kr}_n=\begin{bsmallmatrix}
C_n^{kr} \\ C_{n-1}^{kr} \\ \ldots  \\ C_{0}^{kr} 
\end{bsmallmatrix},\qquad 
(\Upsilon^{kr}_n)_{ac}=\begin{bsmallmatrix}
A_0^{kr} \\ \overline{A}_1^{kr} \\ \ldots 
\end{bsmallmatrix}, \quad (\Gamma^{kr}_n)_{ac}=\begin{bsmallmatrix}
C_n^{kr} \\ \overline{C}_{n-1}^{kr} \\ \ldots  
\end{bsmallmatrix}
\end{align}
and further set $M^{rsk}_{0}=(A_0^{kr})^{T}C_0^{kr}=(A_0^{kr})^{T}B_0^{r}A_0^{kr}$, $N^{rsk}_{0}=(\overline{A}_0^{kr})^{T}C_0^{kr}=(\overline{A}_0^{kr})^{T}B_0^{r}A_0^{kr}$,
\[
M^{rsk}_{n}=(\Upsilon^{kr}_n)^{T}_{ac}(\Gamma^{kr}_n)_{ac},\qquad
N^{rsk}_{n}=(\overline{\Upsilon}^{kr}_n)^{T}_{ac}(\Gamma^{kr}_n)_{ac},\qquad n\geq 1.
\]
%
%\begin{align*}
%&M^{rsk}_{2n}=
%\begin{bsmallmatrix}
%(A_0^{kr})^{T} & (\overline{A}_1^{kr})^{T} & \ldots  & (A_{2n}^{kr})^{T} 
%\end{bsmallmatrix}
%\begin{bsmallmatrix}
%C_{2n}^{ks} \\
%\overline{C}_{2n-1}^{ks} \\
%\ldots \\
%C_{0}^{ks}
%\end{bsmallmatrix}, \quad
%M^{rsk}_{2n-1}=
%\begin{bsmallmatrix}
%(A_0^{kr})^{T} & (\overline{A}_1^{kr})^{T} & \ldots & (\overline{A}_{2n-1}^{kr})^{T} 
%\end{bsmallmatrix}
%\begin{bsmallmatrix}
%C_{2n-1}^{ks} \\
%\overline{C}_{2n-2}^{ks} \\
%\ldots \\
%\overline{C}_{0}^{ks}
%\end{bsmallmatrix}
%\\
%&N^{rsk}_{2n}=
%\begin{bsmallmatrix}
%(\overline{A}_0^{kr})^{T} & (A_1^{kr})^{T} & \ldots & (\overline{A}_{n}^{kr})^{T} 
%\end{bsmallmatrix}
%\begin{bsmallmatrix}
%C_{2n}^{ks} \\
%\overline{C}_{2n-1}^{ks} \\
%\ldots \\
%\widetilde{C}_{0}^{ks}
%\end{bsmallmatrix}, \quad
%N^{rsk}_{2n-1}=
%\begin{bsmallmatrix}
%(\overline{A}_0^{kr})^{T} & (A_1^{kr})^{T} & \ldots & (A_{n}^{kr})^{T} 
%\end{bsmallmatrix}
%\begin{bsmallmatrix}
%C_{2n}^{ks} \\
%\overline{C}_{2n-1}^{ks} \\
%\ldots \\
%\overline{C}_{0}^{ks}
%\end{bsmallmatrix}\\
%\end{align*}
%
%where $\widetilde{A}_{n}^{kr}=\left\{\begin{array}{ll}
%\overline{A}_{n}^{kr}, & n \textrm{ odd}\\
%A_{n}^{kr}, & n \textrm{ even}
%\end{array}
%\right.$ and $\widetilde{C}_{n}^{kr}=\left\{\begin{array}{ll}
%\overline{C}_{n}^{kr}, & n \textrm{ odd}\\
%C_{n}^{kr}, & n \textrm{ even}
%\end{array}
%\right.$.
We compute ($n\geq 0$):
\begin{align*}
(\overline{N}^{rsk}_{2n+1})^{T}
=&\sum_{j=0}^{n}(\overline{C}_{2j+1}^{kr})^{T}A_{2n-2j}^{ks}+\sum_{j=0}^{n}(C_{2j}^{kr})^{T}\overline{A}_{2n+1-2j}^{ks}\\
=&\sum_{j=0}^{n}\sum_{l=0}^{j}\big((\overline{A}_{2l+1}^{kr})^{T}B_{2j-2l}+(A_{2l}^{kr})^{T}B_{2j+1-2l}\big)A_{2n-2j}^{ks}\\
&+\big(\sum_{j=0}^{n}\sum_{l=0}^{j}(A_{2l}^{kr})^{T}B_{2j-2l}+\sum_{j=1}^{n}\sum_{l=0}^{j-1}(\overline{A}_{2l+1}^{kr})^{T}B_{2j-1-2l}\big)\overline{A}_{2n+1-2j}^{ks}\\
=& \sum_{l=0}^{n}\sum_{j=l}^{n}(\overline{A}_{2l+1}^{kr})^{T}B_{2j-2l}A_{2n-2j}^{ks}+\sum_{l=0}^{n}\sum_{j=l}^{n}(A_{2l}^{kr})^{T}B_{2j+1-2l}A_{2n-2j}^{ks}\\
&+\sum_{l=0}^{n}\sum_{j=l}^{n}(A_{2l}^{kr})^{T}B_{2j-2l}\overline{A}_{2n+1-2j}^{ks}+\big(\sum_{l=0}^{n-1}\sum_{j=l+1}^{n}(\overline{A}_{2l+1}^{kr})^{T}B_{2j-1-2l}\big)\overline{A}_{2n+1-2j}^{ks}\\
=&(\overline{A}_{2n+1}^{kr})^{T}B_0A_0
+\sum_{l=0}^{n}(A_{2l}^{kr})^{T}\sum_{j'=0}^{n-l}\big(B_{2j'+1}A_{2n-2l-2j'}^{ks}+B_{2j'}\overline{A}_{2n+1-2l-2j'}^{ks}\big)\\
&+\sum_{l=0}^{n-1}(\overline{A}_{2l+1}^{kr})^{T}\left(B_0A_{2n-2l}+\sum_{j'=0}^{n-1-l}\big(B_{2j'+2}A_{2n-2l-2j'-2}^{ks} +B_{2j'+1}\overline{A}_{2n-1-2l-2j}^{ks}\big)\right)
\\
=&\sum_{l=0}^{n}(\overline{A}_{2l+1}^{kr})^{T} C_{2n-2l}^{ks}+ \sum_{l=0}^{n}(A_{2l}^{kr})^{T} \overline{C}_{2n+1-2l}^{ks}=\overline{N}^{rsk}_{2n+1}
\end{align*}
Likewise we obtain for $n\geq 1$:
\begin{align*}
(N^{rsk}_{2n})^{T}
=&\sum_{j=1}^{n}(\overline{C}_{2j-1}^{kr})^{T}A_{2n+1-2j}^{ks}+\sum_{j=0}^{n}(C_{2j}^{kr})^{T}\overline{A}_{2n-2j}^{ks}\\
=&\sum_{j=1}^{n}\sum_{l=0}^{j-1}\big((A_{2l}^{kr})^{T}B_{2j-1-2l}+(\overline{A}_{2l+1}^{kr})^{T}B_{2j-2-2l}\big)A_{2n+1-2j}^{ks}\\
&+\big(\sum_{j=0}^{n}\sum_{l=0}^{j}(A_{2l}^{kr})^{T}B_{2j-2l}+\sum_{j=1}^{n}\sum_{l=0}^{j-1}(\overline{A}_{2l+1}^{kr})^{T}B_{2j-1-2l}\big)\overline{A}_{2n-2j}^{ks}
\end{align*}
\begin{align*}
=&\sum_{l=0}^{n-1}\sum_{j=l+1}^{n}\big((A_{2l}^{kr})^{T}B_{2j-1-2l}A_{2n+1-2j}^{ks}\big)+\sum_{l=0}^{n}\sum_{j=l}^{n}\big((A_{2l}^{kr})^{T}B_{2j-2l}\overline{A}_{2n-2j}^{ks}\big)\\
&+\sum_{l=0}^{n-1}\sum_{j=l+1}^{n}\big((\overline{A}_{2l+1}^{kr})^{T}B_{2j-2l}A_{2n+1-2j}^{ks}+(\overline{A}_{2l+1}^{kr})^{T}B_{2j-1-2l}\overline{A}_{2n-2j}^{ks}\big)\\
%\end{align*}
%\begin{align*}
=&(A_{2n}^{kr})^{T}B_0\overline{A}_0
+\sum_{l=0}^{n-1}(\overline{A}_{2l+1}^{kr})^{T}\sum_{j'=0}^{n-l-1}\big(B_{2j'-2}A_{2n-1-2j-2l}^{ks}+B_{2j'+1}\overline{A}_{2n-2j'-2l-2}^{ks}\big)\\
+&\sum_{l=0}^{n-1} (A_{2l}^{kr})^{T}\left(B_{0}\overline{A}_{2n-2l}^{ks}+\sum_{j'=0}^{n-l-1}\big(B_{2j'+1}A_{2n-1-2j'-2l}^{ks}+B_{2j'+2}\overline{A}_{2n-2j'-2l-2}^{ks}\big)\right)\\
=&\sum_{l=0}^{n}(A_{2l}^{kr})^{T} \overline{C}_{2n-2l}^{ks}+ \sum_{l=0}^{n-1}(\overline{A}_{2l+1}^{kr})^{T} C_{2n-1-2l}^{ks}=\overline{N}^{rsk}_{2n}
\end{align*}
In a similar fashion we prove
\[
(M^{rsk}_{2n+1})^{T}=\overline{M}^{rsk}_{2n+1}, \qquad (M^{rsk}_{2n})^{T}=M^{rsk}_{2n}.
\]
In view of (\ref{YSP}) for $r=s$ we deduce that
\begin{equation}\label{YSMN}
(\widetilde{Y}_{rk})_{(1)}(\mathcal{S}_{kr})^{(j)}=
\left\{
\begin{array}{ll}
%P^{rsk}_{j-\alpha_s+\alpha_r-1}, & \alpha_s \geq \alpha_k\geq \alpha_r,j>\alpha_s-\alpha_k\\
%P^{rsk}_{j-1}, &  \alpha_r\geq  \alpha_k \geq \alpha_s\\
\overline{N}^{rrk}_{j-\alpha_k+\alpha_r-1}, & \alpha_k\geq\alpha_r, \alpha_k \textrm{ even}, \alpha_r \textrm{ odd} \\
\overline{M}^{rrk}_{j-\alpha_k+\alpha_r-1}, & \alpha_k\geq\alpha_r, \alpha_k \textrm{ odd}, \alpha_r \textrm{ even} 
%\textrm{   or   }\alpha_k\leq \alpha_r, \alpha_k \textrm{ sodo }
\\
N^{rrk}_{j-\alpha_k+\alpha_r-1}, & \alpha_k\geq \alpha_r,\alpha_k\textrm{ even}, \alpha_r \textrm{ even} \\
M^{rrk}_{j-\alpha_k+\alpha_r-1}, & \alpha_k\geq \alpha_r, \alpha_k\textrm{ odd}, \alpha_r \textrm{ odd} 
%\textrm{   or   }\alpha_k\leq \alpha_r, \alpha_k \textrm{ sodo }
\\
M^{rrk}_{j-\alpha_r+\alpha_k-1}, & j>\alpha_r-\alpha_k\geq 0, \alpha_k \textrm{ odd}, \alpha_r \textrm{ even},\\
& \textrm{or}\quad \alpha_k, \alpha_r \textrm{ odd} \\
N^{rrk}_{j-\alpha_r+\alpha_k-1}, & j>\alpha_r-\alpha_k\geq 0, \alpha_k \textrm{ even}, \alpha_r \textrm{ odd},\\
& \textrm{or}\quad \alpha_k, \alpha_r \textrm{ even} \\
%\textrm{   or   }\alpha_k\leq \alpha_r, \alpha_k \textrm{ liho }
0, & \textrm{otherwise}
\end{array}
\right.
\end{equation}
and we observe that $(\widetilde{Y}_{rk})_{(1)}(\mathcal{S}_{kr})^{(j)}$ (hence $\sum_k(\widetilde{Y}_{rk})_{(1)}(\mathcal{S}_{kr})^{(j)}$) is symmetric for $j-\alpha_r$ even and Hermitian for $j-\alpha_r$ odd. 

\quad
Again, as in Case \ref{EqT1}. the step STEP a.a fixes arbitrarily the blocks below the main diagonal of the block matrix 
$[\mathcal{Y}]_{r,s=1}^{N}$, 
which adds $2\sum_{r=1}^{N}\alpha_r \cdot m_r\sum_{s=1}^{r-1}m_s$ to real dimension of the solution.

\quad
We proceed by computing matrices in STEP 0.0. Since
\[
(\widetilde{Y}_{rk})_{(1)}=\left\{
\begin{array}{ll}
\begin{bmatrix}
(A_0^{rr})^{T} & *
\end{bmatrix}, & k=r, r \textrm{ odd}\\
\begin{bmatrix}
(\overline{A}_0^{rr})^{T} & *
\end{bmatrix}, & k=r, r \textrm{ even}\\
\begin{bmatrix}
0 & *
\end{bmatrix}, & k<r
\end{array}
\right.,\qquad
((\mathcal{S})_{kr})^{(1)}=
\left\{
\begin{array}{ll}
\begin{bmatrix}
B_0^{k}A_0^{kr} \\
0
\end{bmatrix}, & k\leq r\\
0, & k>r
\end{array}
\right.,
\]
we obtain
\begin{equation*}
((\widetilde{Y}\mathcal{S})_{rr})_{11}=\sum_{k=1}^{N}(\widetilde{Y}_{rk})_{(1)}((\mathcal{S})_{kr})^{(1)}=(\widetilde{Y}_{rr})_{(1)}((\mathcal{S})_{rr})^{(1)}=
\left\{
\begin{array}{ll}
(A_0^{rr})^{T}B_0^{r}A_0^{rr}, & \alpha_r \textrm{ odd}\\
(\overline{A}_0^{rr})^{T}B_0^{r}A_0^{rr}, & \alpha_r \textrm{ even}
\end{array}
\right..
\end{equation*}
Therefore (\ref{eqFYFIY}) yields
\[
G_0^{r}=\left\{
\begin{array}{ll}
(A_0^{rr})^{T}B_0^{r}A_0^{rr}, & \alpha_r \textrm{ odd}\\
(\overline{A}_0^{rr})^{T}B_0^{r}A_0^{rr}, & \alpha_r \textrm{ even}
\end{array}
\right.,\qquad  r\in \{1,\ldots,N\}.
\]
Since $B_0^{r},G_0^{r}$ are real symmetric, then by Sylvester's inertia theorem this equation for $\alpha_r$ even has a solution for $A_0^{rr}$ precisely when $B_0^{r},G_0^{r}$ are of the same inertia. The case when $\alpha_r$ is odd is treated as in Case \ref{EqT1}. Furthermore, the complex dimension of the (possible) solution is $\frac{m_r(m_r-1)}{2}$ for $\alpha_r$ odd ($m_r^{2}$ for $\alpha_r$ even).

\quad
Again, we compute $A_j^{r(r+p)}$, $N-1\geq p\geq 0$, $r\in \{1,\ldots,N-p\}$ such that $\alpha_r\geq j+1$ (STEP j.p.), while assuming that all matrices from the previous steps are already determined ($A_{j}^{rs}$ for $r>s$ and any $j$, $A_{j'}^{r(r+p')}$ for either $j'<j$ or $j'=j$, $p'<p$). 
Using the same arguments as in Case \ref{EqT1}. we prove that
the second term in (\ref{YSp1}), the second and the third term in (\ref{f1}) and the second term in (\ref{fN}) possibly depend only on  matrices determined in the previous steps; recall that we still have that $C_{n}^{rs}=B_0^{r}A_{n}^{rs}+\theta_n^{rs}$ for $n\geq 1$ with $\theta_n^{rs}$ depending only on $A_{j'}^{rs}$ for $j'<n$.
Furthermore, the first term of (\ref{f00}), (\ref{YSp1}), (\ref{f1}), (\ref{fN}) is a matrix product of (r=s+p):
\begin{align*}
&(\widetilde{Y}_{rr})_{(1)}=\left\{
\begin{array}{ll}
(\Upsilon^{rr}_{\alpha_{r}-1})_{ac}^{T}, & \alpha_{r}\textrm{ odd}\\
(\overline{\Upsilon}^{rr}_{\alpha_{r}-1})_{ac}^{T}, & \alpha_{r}\textrm{ even}
\end{array}
\right.,\\
%\begin{bsmallmatrix}
%(A_0^{(s-p)(s-p)})^{T} &  \ldots & (A_{\alpha_{s-p}-1}^{(s-p)(s-p)})^{T}
%\end{bsmallmatrix},\\ 
&(\mathcal{S}_{rr})^{(\alpha_r)}=
%\begin{bsmallmatrix}
%C_j^{rr} \\
%\ldots\\
%C_0^{rr}
%\end{bsmallmatrix}
(\Gamma^{rr}_{\alpha_{r}-1})_{ac},\qquad
(\mathcal{S}_{r(r+p)})^{(j+1)}=
\begin{bsmallmatrix}
(\Gamma^{r(r+p)}_{j})_{ac} \\
\ldots\\
0\\
\ldots\\
0
\end{bsmallmatrix}, j< \alpha_r-1\textrm{ or } p\geq 0.
\end{align*}
We thus have
\begin{align*}
&(\widetilde{Y}_{rr})_{(1)}(\mathcal{S}_{r(r+p)})^{(j+1)}=\Xi(j,r,p)+\left\{\begin{array}{ll}
(A_0^{rr})^{T}B_0^{r}A_j^{r(r+p)}, & \alpha_{r} \textrm{ odd}\\
(\overline{A}_0^{rr})^{T}B_0^{r}A_j^{r(r+p)}, & \alpha_{r} \textrm{ even}, \qquad p\geq 1
%(\overline{A}_0^{(s-p)s})^{T}A_j^{(s-p)s}, & \alpha_{r} \textrm{ %even} \\
%(A_0^{(s-p)s})^{T}A_j^{(s-p)s}, & \alpha_{s-p} \textrm{ even}
\end{array}
\right.,\\
%\]
%while in case $p=0$ ($s=r$) we get
%\[
&(\widetilde{Y}_{rr})_{(1)}(\mathcal{S}_{rr})^{(j+1)}=\Xi(j,r,0)+
\left\{\begin{array}{ll}
(A_0^{rr})^{T}B_0^{r}A_j^{rr}+(A_j^{rr})^{T}B_0^{r}A_0^{rr}, &  \alpha_r,j \textrm{ odd}\\
(A_0^{rr})^{T}B_0^{r}A_j^{rr}+(\overline{A}_j^{rr})^{T}B_0^{r}\overline{A}_0^{rr}, &  \alpha_s\textrm{ odd},j \textrm{ even}\\
(\overline{A}_0^{rr})^{T}B_0^{r}A_j^{rr}+(A_j^{rr})^{T}B_0^{r}\overline{A}_0^{rr}, &  \alpha_r,j \textrm{ even}\\
(\overline{A}_0^{rr})^{T}B_0^{r}A_j^{rr}+(\overline{A}_j^{rr})^{T}B_0^{r}A_0^{rr}, &  \alpha_r\textrm{ even},j \textrm{ odd}\\
\end{array}
\right.,
\end{align*}
where $\Xi(j,r,p)$ depends only on $A_{j'}^{r(r+p')}$ for either $j'<j$ or (possibly for $p\geq 1$) $j'=j$, $p'<p$. In addition $\Xi(j,r,0)$ is symmetric (Hermitian) for $j-\alpha_r$ odd (even); see (\ref{YSMN}). (Remember that $C_{j'}^{rs}=B_{0}^{r}A_{j'}^{rs}+R(j',r,s)$, where $R(j',r,s)$, $j'\geq 1$ depends on $A_{j''}^{rs}$, $j''<j'$ and $R(0,r,s)=0$.) By comparing the entries in the first row and $(j+1)$-th column of the blocks in the $s-p$-th row and $s$-th column in (\ref{eqFYFIY}) we get:
\begin{align*}
&(\overline{A}_0^{rr})^{T}A_j^{r(r+p)}=\Xi'(j,r,p), \qquad
\alpha_{r+p}+j+\alpha_{r} \textrm{ odd}, \qquad p\geq 1\\
&(A_0^{rr})^{T}A_j^{r(r+p)}=\Xi'(j,r,p), \qquad
\alpha_{r+p}+j+\alpha_{s-p} \textrm{ even}, \qquad p\geq 1,\\
&(A_0^{rr})^{T}A_j^{rr}+(A_j^{rr})^{T}A_0^{rr}=\Xi '(j,r,0), \qquad \alpha_r,j \textrm{ odd}, (p=0),\\
&(A_0^{rr})^{T}A_j^{rr}+(\overline{A}_j^{rr})^{T}\overline{A}_0^{rr}=\Xi '(j,r,0), \qquad \alpha_r \textrm{ odd}, j \textrm{ even } (p=0)
\end{align*}
\begin{align*}
&(\overline{A}_0^{rr})^{T}A_j^{rr}+(A_j^{rr})^{T}\overline{A}_0^{rr}=\Xi '(j,r,0), \qquad \alpha_r,j \textrm{ even }, (p=0),\\
&(\overline{A}_0^{rr})^{T}A_j^{rr}+(\overline{A}_j^{rr})^{T}A_0^{rr}=\Xi '(j,r,0), \qquad \alpha_r \textrm{ even }, j \textrm{ odd} (p=0)
\end{align*}
where $\Xi'(j,r,p)$ for $p\geq 0$ depends only on $A_{j'}^{r(r+p')}$ for either $j'<j$ or (possibly for $p\geq 1$) $j'=j$, $p'<p$, and in addition $\Xi'(j,s,0)$ is symmetric (Hermitian) if $\alpha_r-j$ even (odd); recall (\ref{YSMN}). Note also that $A_0^{rr}$ are invertible.

\quad
To get $A_j^{r(r+p)}$, $p\geq 1$ one needs to solve a simple matrix equation of the form $A^{T}X=B$ (or $A^{*}X=B$) on $X$ with given $A\in GL_n(\mathbb{C})$ and $B\in \mathbb{C}^{n\times n}$, while to get $A_j^{rr}$ we need to solve the equation of the form $A^{T}X+X^{T}A=B$ (or $A^{*}X+X^{*}A=B$) on $X$ with given $A\in GL_n(\mathbb{C})$, $B=B^{T}$ ($B=B^{*}$); the solution is $X=\frac{1}{2}(A^{T})^{-1}B+(A^{T})^{-1}C$ ($X=\frac{1}{2}(A^{*})^{-1}B+(A^{*})^{-1}C)$), where $C$ is any matrix with $C^{T}=-C$ ($C^{*}=-C$). Again, the map $C\mapsto (A^{T})^{-1}C$ ($C\mapsto (A^{*})^{-1}C$) is biholomorphic (diffeomorphic), thus the real dimension of the solution is $\frac{2m_r(2m_r-1)}{2}$ (or $\frac{2m_r(2m_r-1)}{2}+m_r$).

\quad
It is only left to add up the dimensions:
\begin{align*}
&\sum_{\alpha_r \textrm{ even}}\left(\big(4m_r^{2}+(\frac{\alpha_r}{2}-1)(\tfrac{2m_r(2m_r-1)}{2}+m_r)+(\frac{\alpha_r}{2})(\tfrac{2m_r(2m_r-1)}{2})+\alpha_r \cdot 2m_r\sum_{s=1}^{r-1}2m_s\big)\right)+\\
&\sum_{\alpha_r \textrm{ odd}}\left(\big(\tfrac{2m_r(2m_r-1)}{2}+\tfrac{\alpha_r-1}{2}(\tfrac{2m_r(2m_r-1)}{2}+m_r)+\tfrac{\alpha_r-1}{2}\tfrac{2m_r(2m_r-1)}{2}+\alpha_r \cdot 2m_r\sum_{s=1}^{r-1}2m_s\big)\right)\\
=&\sum_{\alpha_r \textrm{ even}}m_r\left(\tfrac{3\alpha_r}{2}\right)+\sum_{\alpha_r \textrm{ odd}}m_r\left(-\tfrac{\alpha_r+1}{2}\right)+\sum_{r=1}^{N}2\alpha_r m_r \big(m_r+\sum_{s=1}^{r-1}2m_s\big)
\end{align*}
\end{enumerate}
This concludes the proof of (\ref{EqT2}).
\end{proof}

\begin{remark}
One could consider the equation (\ref{eqFYFIY}) even in the case when the diagonal blocks $\mathcal{B}$, $\mathcal{B}'$ might not be invertible. Note that in this more general setting the set of solutions of the equations $A^{T}X+X^{T}A=B$ and $A^{*}{T}X+X^{*}A=B$ is known (see \cite{Brad} and \cite{LR}).
\end{remark}

\begin{example}
We solve (\ref{eqFYFIY}) for $\mathcal{F}=E_4(I)\oplus E_2(I)\oplus I$, $\mathcal{B}=\mathcal{B}'=I_4(I)\oplus I_2(I)\oplus I$. Set 
\small
$\mathcal{Y}=\begin{bmatrix}[cccc|cc|c]
A_1 & B_1 & C_1 & D_1  &  H_1  &  I_1  &  J_1\\
0   & \overline{A}_1 & \overline{B}_1   & \overline{C}_1  &  0  &   \overline{H}_1  &  0\\
0   & 0   & A_1   & B_1    &  0 & 0 & 0 \\
0   & 0   &  0  &  \overline{A}_1   & 0 &0 & 0\\
\hline
0   & 0   & N_1 & P_1                         &  A_3  &  B_3  &  J_3\\
0   & 0   & 0   & \overline{N}_1              &  0    &   \overline{A}_3  &  0\\
\hline
0   & 0   & 0   & R_1                        &  0    &   R_3    & A_4 
\end{bmatrix}$ 
\normalsize

We compute:
\small
\setlength{\arraycolsep}{3pt}
\begin{align*}
&\widetilde{\mathcal{Y}}\mathcal{Y}=\begin{bmatrix}[cccc|cc|c]
\overline{A}_1^{T} & B_1^{T} & \overline{C}_1^{T} & D_1^{T}  &  \overline{N}_1^{T}  &  P_1^{T}  &  R_1^{T}\\
0   & A_1^{T} & \overline{B}_1^{T}   & C_1^{T}  &  0  &   N_1^{T}  &  0\\
0   & 0   & \overline{A}_1^{T}   & B_1^{T}    &  0 & 0 & 0 \\
0   & 0   &  0  &  A_1^{T}  & 0 &0 & 0\\
\hline
0   & 0   & \overline{H}_1^{T} & I_1^{T}                        &  \overline{A}_3^{T}  &  B_3^{T}  &  R_3^{T}\\
0   & 0   & 0   & H_1^{T}              &  0    &   A_3  &  0\\
\hline
0   & 0   & 0   & J_1^{T}                         &  0    &   J_3^{T}    & A_4 
\end{bmatrix}
\begin{bmatrix}[cccc|cc|c]
A_1 & B_1 & C_1 & D_1  &  H_1  &  I_1  &  J_1\\
0   & \overline{A}_1 & \overline{B}_1   & \overline{C}_1  &  0  &   \overline{H}_1  &  0\\
0   & 0   & A_1   & B_1    &  0 & 0 & 0 \\
0   & 0   &  0  &  \overline{A}_1   & 0 &0 & 0\\
\hline
0   & 0   & N_1 & P_1                         &  A_3  &  B_3  &  J_3\\
0   & 0   & 0   & \overline{N}_1              &  0    &   \overline{A}_3  &  0\\
\hline
0   & 0   & 0   & R_1                        &  0    &   R_3    & A_4 
\end{bmatrix}=
\end{align*}
\scriptsize
\setlength{\arraycolsep}{3pt}
\begin{align*}
&=
\begin{bmatrix}[cccc|cc|c]
\overline{A}_1^{T}A_1 & \overline{A}_1^{T}B_1+B_1^{T}\overline{A}_1 & \overline{A}_1^{T}C_1+\overline{C}_1^{T}A_1 & *  &  \overline{A}_1^{T}H_1+\overline{N}_1^{T}A_3  &  \overline{A}_1^{T}I_1+B_1^{T}\overline{H}_1+\overline{N}_1^{T}B_3  &  \overline{N}_1^{T}J_3+R_1^{T}A_4\\
           &                       & + B_1^{T}\overline{B}_1+\overline{N}_1^{T}N_1               &          &                         &           +P_1^{T}\overline{A}_3+R_1^{T}R_3 & +\overline{A}_1^{T}J_1 \\
           & A_1^{T}\overline{A}_1 &  A_1^{T}\overline{B}_1+\overline{B}_1^{T}A_1   & A_1^{T}\overline{C}_1+C_1^{T}\overline{A}_1  &  0  &   A_1^{T}\overline{H}_1+N_1^{T}\overline{A}_3  &  0\\
   &                       &                          &       + \overline{B}_1^{T}B_1+N_1^{T}\overline{N}_1    &     &                          &   \\
   &                       & \overline{A}_1^{T}A_1               &  \overline{A}_1^{T}B_1+\overline{B}_1^{T}A_1           &  0  & 0                      &      0 \\
   &    &    &  A_1^{T}\overline{A}_1  &  0   &0 & 0\\
\hline 
   &    & &                      &  \overline{A}_{3}^{T}A_3  &   \overline{A}_3^{T}B_3+B_3^{T}\overline{A}_3  &  \overline{A}_3^{T}J_3+R_3^{T}A_4\\
   &    & &                      &                &    + R_3^{T}R_3          &                       \\
   &    &    &              &         &   A_3^{T}\overline{A}_3  &  0\\
\hline
   &    &    &                       &      &      & A_4 ^{T}A_4
\end{bmatrix}
\end{align*}
\normalsize
The diagonal of the diagonal blocks gives that $A_1\ldots,A_4$ are any orthogonal matrices. Now we choose $N_1$, $P_1$, $R_1$, $R_3$ arbitrarily. The upper diagonal blocks yield $ \overline{A}_1^{T}H_1+\overline{N}_1^{T}A_3=0$ and $\overline{A}_3^{T}J_3+R_3^{T}A_4=0$, which further gives $H_1,J_1$. The second upper diagonal block yields $\overline{N}_1^{T}J_3+\overline{A}_1^{T}J_1+R_1^{T}A_4=0$, thus $J_1$ follows.

\quad
Next, observing the upper diagonal of the main diagonal blocks gives $\overline{A}_1^{T}B_1+B_1^{T}\overline{A}_1=0$ and $\overline{A}_3^{T}B_3+B_3^{T}\overline{A}_3+R_3^{T}R_3=0$, so we deduce $B_1,B_3$. From the upper diagonal of the upper diagonal we obtain $A_1^{T}I_1+B_1^{T}H_1+N_1^{T}B_3 +P_1^{T}A_3+R_1^{T}R_3=0$, so we deduce $I_1$.

\quad
The third and fourth upper diagonal block of the first principal diagonal block gives $\overline{A}_1^{T}C_1+\overline{C}_1^{T}A_1+B_1^{T}\overline{B}_1+\overline{N}_1^{T}N_1 =0$, $\overline{A}_1^{T}D_1+B_1^{T}\overline{C}_1+\overline{C}_1^{T}B_1+D_1^{T}\overline{A}_1+\overline{N}_1^{T}P_1+P_1^{T}\overline{N}_1+R_1^{T}R_1 =0$ (see $*$), therefore $C_1$, $D_1$, respectively.

\end{example}

\section{Proofs of Theorem \ref{stabs}, Theorem \ref{stabz} and Corollary \ref{posapp}}\label{sec2}

To prove Theorem \ref{stabs} (Theorem \ref{stabz}), we first use Lemma \ref{posls} (Lemma \ref{posl}) to solve the first equation of (\ref{eqAoXXB1}) (of (\ref{eqAoXXB})) on $Q$ for a given normal form $A=B$. Taking into account that $Q$ satisfies $Q^{T}Q=I$ (the second equation of (\ref{eqAoXXB1}), (\ref{eqAoXXB})), then  yields a certain matrix equation and further restricting the form of $Q$; at this point Lemma \ref{EqT} is applied.

\begin{proof}[Proof of Theorem \ref{stabs}]
We need to solve the equation
\begin{equation}\label{SQQS}
(\mathcal{S}(A))Q= Q(\mathcal{S}(A)),
\end{equation}
where $Q$ is a complex orthogonal matrix $Q$. 

By writing 
$\mathcal{S}(A)=\bigoplus_{r=1}^{N}\mathcal{S}(A,\rho_r)$,
where all blocks of 
%$\mathcal{J}'_q(A)$ (and 
$\mathcal{S}(A)$ corresponding to the eigenvalue $\rho_r$ of $A\overline{A}$ are collected together into 
%$\mathcal{J}^{\lambda}_q(A)$ (and 
$\mathcal{S}(A,\rho_r)$, 
it then follows from Lemma \ref{posl} that $Q$ in (\ref{SQQS}) is of the form $Q=\oplus_{r=1}^{N}Q^{r}$ (hence $Q=\oplus_{r}Q^{r}$), and 
%(\ref{MQQNe}) and 
(\ref{SQQS}) thus splits into equations 
\begin{equation*}%\label{SQQSl}
%\mathcal{J}^{\lambda}_q(A)\overline{Y}^{\lambda}=Y^{\lambda}\mathcal{J}^{\lambda}_q(A), \quad \lambda\in \Lambda \qquad 
\mathcal{S}(A,\rho_r)Q^{r}=Q^{r}\mathcal{S}(A,\rho_r),\quad  r\in\{1,\ldots,N\}.
\end{equation*}
%
%respectively. 
%
%
%Further recall that $-K_m(z)$, $-L_m(z)$ are orthogonally $*$-congruent to %$K_m(z)$, $L_m(z)$, respectively.
%Trivially, the existence of the complex orthogonal matrix $Q$ in %(\ref{SQQS}) is then equivalent to the existence of a complex %orthogonal matrix $Q^{r}$ and the corresponding  $X^{r}$ in %(\ref{HQQHl}) for all $r$ which correspond io the eigenvalue %$\rho_r\geq 0$. 

Suppose
\[
\mathcal{S}(A)=\bigoplus_{r=1}^{N} \mathcal{S}_{\alpha_r}(\lambda)=\bigoplus_{r=1}^{N}\left( \bigoplus_{j=1}^{m_r} S_{\alpha_r}(\lambda)\right),
\]
where $\mathcal{S}_{\alpha_r}(\lambda)$, $r\in \{1,\ldots,N\}$ is a direct sum of all blocks of size $\alpha_r\times \alpha_r $ corresponding to the eigenvalue $\lambda$.
By Lemma \ref{posls} the solution $Q$ of the equation $\mathcal{S}(A)\overline{Q}=Q\mathcal{S}(A)$ 
%is partitioned to blocks conformally to $\mathcal{H}^{1}(A)$, 
is of the form 
\begin{equation}\label{QPP}
Q=P^{-1}YP, \qquad\quad
P=\oplus_{r=1}^{N}P_{r}', \quad 
P_r'=\oplus_{j=1}^{m_r}P_{\alpha_r}, \quad
P_{\alpha_r}=\tfrac{1}{\sqrt{2}}(I_{\alpha_r}+iE_{\alpha_r}),
\end{equation}
where $Y=[Y_{rs}]_{r,s=1}^{N}$ with further $Y_{rs}$ is a $m_r\times m_s$ block matrix whoose blocks of dimension $\alpha_r\times \alpha_s$ are of the form 
%partitioned conformally to blocks as $\mathcal{H}_{}^1(A)$ and the block $Y_{jk}$ are of the form 
%(\ref{QTC}) with $T_1$ and $T_2$ possibly two different matrices of the form (\ref{QT}) for a complex upper-triangular Toeplitz matrix $T$.
%
%$Q=[Q_{rs}]_{r,s=1}^{N}$, and further $Q_{rs}$ is  a $m_r\times m_s$ block matrix whoose blocks of dimension $2\beta_r\times 2\beta_s$ are of the form 
%$[Z_{jk}]_{j,k=1}^{j=m_r,k=m_s}$ with
%, $\mathcal{J}_{q}(A)$, and using (\ref{MQQN}), the equation (\ref{HQQHl}) (or \ref{MQQNeH}) then splits into
%\[
%J_{\alpha_{\mu}}(\lambda,1)\overline{Y}_{\mu\nu}=Y_{\mu\nu}J_{\alpha_{\nu}}(\lambda,1), %\qquad Y_{\mu\nu}=P_{\alpha_{\mu}}s_{\widehat{\epsilon},\mu}Q_{\mu\nu} %s_{\widetilde{\epsilon},\nu}^{-1}P_{\alpha_{\nu}}^{-1}.
%\]
%
%$P_{\beta_r}'^{-1}W_{\beta_r}'^{-1}T_{\beta_r}Y_{jk}T_{\beta_s}^{-1}W_{\beta_s}'P_{\beta_s}'$ with $Y_{jk}$ 
of the form (\ref{QT00})
for $m=\alpha_r$, $n=\alpha_s$: 
\begin{equation}
\left\{
\begin{array}{ll}
[0\quad T], & \alpha_{r}<\alpha_{s}\\
\begin{bmatrix}
T\\
0
\end{bmatrix}, & \alpha_{r}>\alpha_{s}\\
T,& \alpha_r=\alpha_{s}
\end{array}\right.,
\end{equation}
where $T\in \mathbb{C}^{m\times m}$, $m=\min\{\alpha_r, \alpha_s\}$ is a complex upper-triangular Toeplitz matrix.

From (\ref{QPP}) we get $P=P^T$, $P^{-1}=\overline{P}$, $(P_{r}')^{2}=-\overline{P'}_{_r}^2=-(P_{\alpha_r}')^{-2}=i\big(\oplus_{j=1}^{m_r}(E_{\alpha_r})\big)$, hence $P^{2}=\overline{P}^2=-P^{-2}=iE$, where $E=\oplus_{r=1}^{N}\left(\oplus_{j=1}^{m_r} (E_{\alpha_r}) \right)$. Thus $I=Q^TQ$ if and only if 
\begin{align}\label{QTQS}
I=   & (P^TY^T(P^{-1})^T)(P^{-1}Y P )\nonumber\\
 I  = & P(P^TY^T(P^{-1})^T)(P^{-1}Y P )P^{-1},\nonumber\\
% I =   & P  P^TY^T P^{-1}P^{-1}Y\\
 %I = & S_{\widetilde{\epsilon}}^2\overline{P}^2Y^T(P^{-1})^2S_{\widehat{\epsilon}}^{2}Y\nonumber\\
 I = & P^2Y^TP^{-2}Y \\
  I = & iE Y^{T} (-iE)Y\nonumber\\
    I = & E Y^{T} E Y\nonumber
\end{align}

Proceed by $T$-conjugating with the permutation matrix $\Omega=\oplus_{r=1}^{N}\Omega_r $ with $\Omega_r$ as in (\ref{perS}) to get block matrices such that their blocks are block Toeplitz matrices:
\begin{align}\label{ortoC3S}
 %&I = \mathcal{I}_{\epsilon}\Omega\Omega^{T} \Omega^{T}Y^{T}\Omega\Omega^{T}S_{\epsilon}\Omega \Omega^{T}Y\Omega\\
 &I = (\Omega^{T}E\Omega)(\Omega^{T} Y^{T} \Omega)(\Omega^{T}E\Omega)(\Omega^{T}Y\Omega)\\
 &I = \mathcal{F}\mathcal{Y}^{T}\mathcal{F}\mathcal{Y}.\nonumber
 %\qquad I= \mathcal{I}_{\epsilon}Z \mathcal{I}_{\epsilon} \widetilde{Z}^{T},\nonumber
\end{align}
where we denoted $\mathcal{F}=\Omega^{T}E\Omega=\oplus_{j=1}^{N}E_{\alpha_r}(E_{m_r})$ and $\mathcal{Y}=\Omega^{T}Y\Omega$ is 
of the form (\ref{0T0}) with complex upper-triangular block Toeplitz matrices $\mathcal{T}_{rs}$.
%
%\begin{equation}\label{bT3S}
%(\mathcal{Y})_{rs}=\Omega_r^{T}Y_{rs}\Omega_s=\left\{
%\begin{array}{ll}
%[0\quad T], & \alpha_r>\alpha_s\\
%\begin{bmatrix}
%T\\
%0
%\end{bmatrix}, & \alpha_r<\alpha_s\\
%T,& \alpha_r=\alpha_s
%\end{array}\right.,
%
%\end{equation}
%$T=T(A_0^{rs},A_1^{rs},\ldots,A_{m-1}^{rs})$, $A_j^{rs}\in \mathbb{C}^{m_r\times m_r}$, $m=\min\{\alpha_r, \alpha_s\}$ (a complex upper-diagonal block Toeplitz matrix).
By applying Lemma \ref{EqT} (\ref{EqT1}) to (\ref{ortoC3S}) we conclude the proof of the theorem.
\end{proof}

\begin{proof}[Proof of Theorem \ref{stabz}]
Given $\epsilon=\{\epsilon_j\}_j$ let $\mathcal{H}^{\epsilon}(A)$ be of the form (\ref{NF1}). We have
\begin{align*}
&\mathcal{H}^{1}(A)=  T_{\epsilon}\mathcal{H}^{\epsilon}(A)\overline{T_{\epsilon}^{-1}}, 
\end{align*}
\begin{align*}
T_{\epsilon}=S_{\epsilon}\oplus \left(\bigoplus_{k}(I_{\beta_k}\oplus I_{\beta_k})\right)\oplus \left(\bigoplus_{l}(I_{\gamma_l}\oplus I_{\gamma_l})\right), \quad 
S_{\epsilon}=\bigoplus_js_{\epsilon,j}I_{\alpha_j}, \, 
s_{\epsilon,j}=\left\{
\begin{array}{ll}
1, & \epsilon_j=1\\
i, & \epsilon_j=-1
\end{array}\right.
\end{align*}

To study orthogonal $*$-equivalence of $\mathcal{H}^{\widehat{\epsilon}}(A)$ and $\mathcal{H}^{\widetilde{\epsilon}}(A)$, we shall solve
\begin{equation}\label{HQQH}
(\mathcal{H}^{\widehat{\epsilon}}(A))\overline{Q}= Q(\mathcal{H}^{\widetilde{\epsilon}}(A)),
\end{equation}
where $Q$ is an orthogonal matrix. 
By (\ref{MQQNt}) the equation (\ref{HQQH}) then transforms to
\begin{equation}\label{H1QQH1}
(\mathcal{H}^{1}(A))\overline{X}= X(\mathcal{H}^{1}(A)),\qquad X=T_{\widehat{\epsilon}}Q T_{\widetilde{\epsilon}}^{-1}.
\end{equation}
By writing 
\[
%\mathcal{J}'_q(A)=\bigoplus_{\lambda\in \Lambda}\mathcal{J}^{\lambda}_q(A), \qquad 
\mathcal{H}^{\epsilon}(A)=\bigoplus_{r=1}^{N}\mathcal{H}^{\epsilon}(A,\rho_r),
\]
where all blocks of 
%$\mathcal{J}'_q(A)$ (and 
$\mathcal{H}^{\epsilon}(A)$ corresponding to the eigenvalue $\rho_r$ of $A\overline{A}$ are collected together into 
%$\mathcal{J}^{\lambda}_q(A)$ (and 
$\mathcal{H}^{\epsilon}(A,\rho_r)$, 
it then follows from Lemma \ref{posl} that $X$ in (\ref{H1QQH1}) ($Q$ in (\ref{HQQH})) is of the form $X=\oplus_{r=1}^{N}X^{r}$ (hence $Q=\oplus_{r=1}^{N}Q^{r}$).
%(\ref{MQQNe}) and 
Thus (\ref{H1QQH1}) (and (\ref{HQQH})) splits into
\begin{equation}\label{HQQHl}
%\mathcal{J}^{\lambda}_q(A)\overline{Y}^{\lambda}=Y^{\lambda}\mathcal{J}^{\lambda}_q(A), \quad \lambda\in \Lambda \qquad 
\mathcal{H}^{1}(A,\rho_r)\overline{X}^{r}=X^{r}\mathcal{H}^{1}(A,\rho_r),\quad 
\quad 
%(\mathcal{H}^{\widehat{\epsilon}}(A,\rho_r)\overline{Q}^{r}=Q^{r}\mathcal{H}^{\widetilde{\epsilon}}(A,\rho_r), \quad 
X^{r}=\left\{
\begin{array}{ll}
S_{\widehat{\epsilon}}Q^{r} S_{\widetilde{\epsilon}}^{-1}, & \rho_r\geq 0\\
Q^{r}, & \textrm{otherwise}
\end{array}
\right..
\end{equation}
%
%respectively. 
%
%
%Further recall that $-K_m(z)$, $-L_m(z)$ are orthogonally $*$-congruent to %$K_m(z)$, $L_m(z)$, respectively.
Trivially, the existence of an orthogonal matrix $Q$ in (\ref{HQQH}) ($X$ in (\ref{H1QQH1})) is then equivalent to the existence of  orthogonal matrices $Q^{r}$ and the corresponding $X^{r}$ in (\ref{HQQHl}) (for all $r\in \{1,\ldots,N\}$) which correspond to the eigenvalue $\rho_r\geq 0$. Furthermore, by taking $\widetilde{\epsilon}=\epsilon$ the matrix $Q$ is in the stabilizer of $\mathcal{H}^{\epsilon}$ with respect to the action $\Psi$.

\begin{enumerate}[label={\bf Case \arabic*.},ref={Case \arabic*},   ,wide=0pt,itemsep=15pt]

\item \label{caseH}
Suppose
\[
\mathcal{H}^{\epsilon}(A)
%=\bigoplus_{r=1}^{N} \mathcal{H}_{\alpha_r}(\lambda)
=\bigoplus_{r=1}^{N}\left( \bigoplus_{j=1}^{m_r} \epsilon_{r,j} H_{\alpha_r}(\lambda)\right), \qquad \lambda\geq 0
\]
where $H_{\alpha_r}(\xi)$, $r\in \{1,\ldots,N\}$ is as in (\ref{HKLmz}) for $z=\lambda$, $m=\alpha_r$.
%where $\mathcal{H}_{\alpha_r}(\lambda)$, $r\in \{1,\ldots,N\}$ is a %direct sum of all blocks of size $\alpha_r\times \alpha_r $ %corresponding to the eigenvalue $\lambda\geq 0$.

\quad
By Lemma \ref{posl} the solution $Q$ of the equation $\mathcal{H}^{\epsilon}(A)\overline{Q}=Q\mathcal{H}^{\epsilon}(A)$ 
%is partitioned to blocks conformally to $\mathcal{H}^{1}(A)$, 
is of the form 
\[
Q=S_{\epsilon}^{-1}P^{-1}YPS_{\widetilde{\epsilon}}, \qquad
P=\oplus_{r=1}^{N}P_{r}', \quad 
P_r'=\oplus_{j=1}^{m_r}P_{\alpha_r}, \quad
P_{\alpha_r}=\tfrac{e^{-i\frac{\pi}{4}}}{\sqrt{2}}(I_{\alpha_r}+iE_{\alpha_r}),
\]
and where $Y=[Y_{rs}]_{r,s=1}^{N}$ is a block matrix whoose block $Y_{rs}$ is further a $m_r\times m_s$ block matrix with blocks of dimension $\alpha_r\times \alpha_s$ and of the form 
%partitioned conformally to blocks as $\mathcal{H}_{}^1(A)$ and the block $Y_{jk}$ are
%(\ref{QTC}) with $T_1$ and $T_2$ possibly two different matrices of the form (\ref{QT}) for a complex upper-triangular Toeplitz matrix $T$.
%
%$Q=[Q_{rs}]_{r,s=1}^{N}$, and further $Q_{rs}$ is  a $m_r\times m_s$ block matrix whoose blocks of dimension $2\beta_r\times 2\beta_s$ are of the form 
%$[Z_{jk}]_{j,k=1}^{j=m_r,k=m_s}$ with
%, $\mathcal{J}_{q}(A)$, and using (\ref{MQQN}), the equation (\ref{HQQHl}) (or \ref{MQQNeH}) then splits into
%\[
%J_{\alpha_{\mu}}(\lambda,1)\overline{Y}_{\mu\nu}=Y_{\mu\nu}J_{\alpha_{\nu}}(\lambda,1), %\qquad Y_{\mu\nu}=P_{\alpha_{\mu}}s_{\widehat{\epsilon},\mu}Q_{\mu\nu} %s_{\widetilde{\epsilon},\nu}^{-1}P_{\alpha_{\nu}}^{-1}.
%\]
%
%$P_{\beta_r}'^{-1}W_{\beta_r}'^{-1}T_{\beta_r}Y_{jk}T_{\beta_s}^{-1}W_{\beta_s}'P_{\beta_s}'$ with $Y_{jk}$ 
of the form (\ref{QT})
for $m=\alpha_r$, $n=\alpha_s$
%:
%
%\begin{equation}
%\left\{
%\begin{array}{ll}
%[0\quad T], & \alpha_{r}<\alpha_{s}\\
%\begin{bmatrix}
%T\\
%0
%\end{bmatrix}, &\alpha_{r}>\alpha_{s}\\
%T,& \alpha_{r}=\alpha_{s}
%\end{array}\right.,
%\end{equation}
(and hence with $T\in \mathbb{C}^{m\times m}$, $m=\min\{\alpha_r, \alpha_s\}$, a  real (complex-alternating) upper-triangular Toeplitz matrix for $\lambda>0$ ($\lambda=0$)).

\quad
Since $P=P^T$, $P^{-1}=\overline{P}$ we have $(P_{r}')^{2}=(\overline{P}_{_r}')^2=(P_{\alpha_r}')^{-2}=\big(\oplus_{j=1}^{m_r}(E_{\alpha_r})\big)$, hence $P^{2}=\overline{P}^2=P^{-2}=E$, where $E=\oplus_{r=1}^{N}\left(\oplus_{j=1}^{m_r} (E_{\alpha_r}) \right)$, and $S_{\epsilon}^T=S_{\epsilon}=\oplus_{r=1}^{N}\oplus_{j=1}^{m_r}\epsilon_{r,j} I_{\beta_j}=(S_{\epsilon}^2)^{-1}=S_{\epsilon}^2$.
It follows that $I=Q^TQ$ if and only if 
\begin{align}\label{QTQH}
I=   & (S_{\widetilde{\epsilon}}^TP^TY^T(P^{-1})^T(S_{\epsilon}^{-1})^T)(S_{\epsilon}^{-1}P^{-1}Y P S_{\widetilde{\epsilon}})\nonumber\\
 I  = & PS_{\widetilde{\epsilon}}(S_{\widetilde{\epsilon}}^T P^TY^T(P^{-1})^T(S_{\epsilon}^{-1})^T)(S_{\epsilon}^{-1}P^{-1}Y P S_{\widetilde{\epsilon}})S_{\widetilde{\epsilon}}^{-1}P^{-1},\nonumber\\
 I =   & P S_{\widetilde{\epsilon}}^2 P^TY^T P^{-1}S_{\epsilon}^{2}P^{-1}Y\\
 %I = & S_{\widetilde{\epsilon}}^2\overline{P}^2Y^T(P^{-1})^2S_{\widehat{\epsilon}}^{2}Y\nonumber\\
 I = & S_{\widetilde{\epsilon}}^2\overline{P}^2Y^TP^{-2}S_{\epsilon}^{2}Y\nonumber\\
  S_{\widetilde{\epsilon}}^2 = & \big(P^2 Y P^{2}\big)^{T}S_{\epsilon}^{2}Y\nonumber\\
    S_{\widetilde{\epsilon}} = & E Y^{T} ES_{\epsilon}Y\nonumber
\end{align}

\quad
By $T$-conjugating $Y$ with the permutation matrix $\Omega=\oplus_{r=1}^{N}\Omega_r $ with $\Omega_r$ as in (\ref{perS}) we now rewritte this equation to the equation with block matrices such that their blocks are block Toeplitz matrices:
\begin{align}\label{ortoC3}
 %&I = \mathcal{I}_{\epsilon}\Omega\Omega^{T} \Omega^{T}Y^{T}\Omega\Omega^{T}S_{\epsilon}\Omega \Omega^{T}Y\Omega\\
 \Omega^{T}S_{\widetilde{\epsilon}}\Omega = (\Omega^{T}E\Omega)(\Omega^{T} Y^{T} \Omega)(\Omega^{T}E\Omega)(\Omega^{T}S_{\epsilon}\Omega )(\Omega^{T}Y\Omega).
 %\qquad I= \mathcal{I}_{\epsilon}Z \mathcal{I}_{\epsilon} \widetilde{Z}^{T}.
\end{align}

We have that $\mathcal{Y}=\Omega^{T}Y\Omega$ is of the form (\ref{0T0})
%
%\begin{equation}\label{bT3}
%(\mathcal{Y})_{rs}=\Omega_r^{T}Y_{rs}\Omega_s=\left\{
%\begin{array}{ll}
%[0\quad T], & \alpha_r>\alpha_s\\
%\begin{bmatrix}
%T\\
%0
%\end{bmatrix}, & \alpha_r<\alpha_s\\
%T,& \alpha_r=\alpha_s
%\end{array}\right.,
%
%\end{equation}
for $\mathcal{T}_{rs}=T_c(A_0^{rs},A_1^{rs},\ldots,A_{b_{rs}-1}^{rs})$, $A_j^{rs}\in \mathbb{R}^{m_r\times m_r}$ (or $A_j^{rs}\in \mathbb{C}^{m_r\times m_r}$), $b_{rs}=\min\{\alpha_r, \alpha_s\}$, a real (complex-alternating) upper-diagonal Toeplitz matrix if $\lambda>0$ (when $\lambda=0)$.
Denoting $S_{\epsilon,r}=\oplus_{j=1}^{\alpha_r}\epsilon_{r,j}I_{m_r}$, $S_{\widetilde{\epsilon},r}=\oplus_{j=1}^{\alpha_r}\widetilde{\epsilon}_{r,j}I_{m_r}$ 
we obtain
\begin{align*}
\Omega^{T}S_{\epsilon}\Omega=& (\oplus_{r=1}^{N}\Omega_r^{T})(\oplus_{r=1}^{N}\oplus_{j=1}^{m_r}\epsilon_{r,j}I_{\alpha_r})(\oplus_{r=1}^{N}\Omega_r)\\
                   = & \oplus_{r=1}^{N}\Omega_r^{T}(\oplus_{j=1}^{m_r}\epsilon_{r,j}I_{\alpha_r})\Omega_r
                   =\oplus_{r=1}^{N} \oplus_{j=1}^{\alpha_r}\epsilon_{r,j}I_{m_r}
                   =\oplus_{r=1}^{N}S_{\epsilon,r}\\
\Omega^{T}S_{\widetilde{\epsilon}}\Omega= & \oplus_{r=1}^{N}S_{\widetilde{\epsilon},r}.
%[\Omega_r^{T}Z_{rs}\Omega_s]_{rs},\\
%\Omega^{T}Y\Omega=  &  (\oplus_{r=1}^{N}\Omega_r^{T})[Y_{rs}]_{r,s=1}^{N}(\oplus_{r=1}^{N}\Omega_r)=[\Omega_r^{T}Y_{rs}\Omega_s]_{r,s=1}^{N},
%\Omega^{T}\widetilde{Z}\Omega=  &  (\oplus_{r=1}^{N}\Omega_r^{T})[\widetilde{Z}_{rs}]_{r,s=1}^{N}(\oplus_{r=1}^{N}\Omega_r)=[\Omega_r^{T}\widetilde{Z}_{rs}\Omega_s]_{r,s=1}^{N}.
\end{align*}
Finally,
setting $\mathcal{B}=\oplus_{j=1}^{N}S_{\epsilon,r}$, $\mathcal{B}'=\oplus_{j=1}^{N}S_{\widetilde{\epsilon},r}$, $\mathcal{F}=\Omega^{T} E\Omega =\oplus_{j=1}^{N}E_{\alpha_r}(I_{m_r})$, the equation (\ref{ortoC3}) in the new notation is
%, $\mathcal{F}=\Omega^{T}E\Omega=\oplus_{j=1}^{N}E_{\alpha_r}(E_{m_r})$ we deduce from (\ref{QTQS}):
\[
\mathcal{B}' = \mathcal{F}\mathcal{Y}^{T}\mathcal{F}\mathcal{B}\mathcal{Y}.
\]
Applying Lemma \ref{EqT} (\ref{EqT1a}), (\ref{EqT2}) now proves (\ref{posl1}).

\quad
In particular, the solution exists precisely when $S_{\widetilde{\epsilon},r}$ and $S_{\epsilon,r}$ (hence $\oplus_{j=1}^{\alpha_r}\widetilde{\epsilon}_{r,j}$ and $\oplus_{j=1}^{\alpha_r}\epsilon_{r,j}$) have the same inertia.
It proves the last statement of the theorem (about uniqueness of the normal form $\mathcal{H}^{\epsilon}(A)$).

\item \label{caseK}

Let
\[
\mathcal{H}^{\epsilon}(A)
%=\bigoplus_{r=1}^{N} \mathcal{K}_{\beta_r}(\mu)
=\bigoplus_{r=1}^{N}\left( \bigoplus_{j=1}^{m_r} K_{\beta_r}(\mu)\right), \qquad \mu> 0,
\]
where $K_{\beta_r}(\mu)$, $r\in \{1,\ldots,N\}$ is as in (\ref{HKLmz}) for $z=\mu$, $m=\beta_r$.
%where $\mathcal{K}_{\beta_r}(\mu)$, $r\in \{1,\ldots,N\}$ is a direct %sum of all blocks of size $2\beta_r\times 2\beta_r $ corresponding to %the eigenvalue $-\mu^{2}$ of $\mathcal{J}(A\overline{A})$. 

\quad
By Lemma \ref{posl} the solution $Q$ of the equation $\mathcal{H}^{\epsilon}(A)\overline{Q}=Q\mathcal{H}^{\epsilon}(A)$ 
%is partitioned to blocks conformally to $\mathcal{H}^{1}(A)$, 
is of the form 
\[
Q=P^{-1}V^{-1}SYS^{-1}VP,
\]
where 
\begin{align*}
&P=\oplus_{r=1}^{N}P_{r}', \quad 
P_r'=\oplus_{j=1}^{m_r}e^{\frac{i\pi}{4}}(P_{\beta_k}\oplus P_{\beta_k}), \quad
P_{\beta_k}=\tfrac{e^{-i\frac{\pi}{4}}}{\sqrt{2}}(I_{\beta_k}+iE_{\beta_k}), \\
&V=\oplus_r^{N} V_r, \quad 
V_r=\oplus_{k=1}^{m_r}e^{i\frac{\pi}{4}}(W_{\beta_k}\oplus \overline{W_{\beta_k}}), \quad W_{\beta_k}= \oplus_{j=0}^{\beta_k-1} i^{j},\\ 
&S=\oplus_{r=1}^{N}S_r, \quad 
S_r=\oplus_{k=1}^{m_r}T_{\beta_k}, \quad
T_{\beta_k}=\begin{bsmallmatrix}
0 & U_{\beta_k}(i\mu) \\
J_{\beta_k}(-i\mu)\overline{U_{\beta_k}(i\mu) }& 0
\end{bsmallmatrix},
\end{align*}
with $U_{\beta_k}(i\mu)$ as a solution of the equation $U_{\beta_k}(i\mu)J_{\beta_k}(-\mu^{2})=(J_{\beta_k}(i\mu))^{2}U_{\beta_k}(i\mu)$, and $Y=[Y_{rs}]_{r,s=1}^{N}$ with $Y_{rs}$ a $m_r\times m_s$ block matrix whoose blocks of dimension $2\beta_r\times 2\beta_s$ are 
%partitioned conformally to blocks as $\mathcal{H}_{}^1(A)$ and the block $Y_{jk}$ are of the form 
%(\ref{QTC}) with $T_1$ and $T_2$ possibly two different matrices of the form (\ref{QT}) for a complex upper-triangular Toeplitz matrix $T$.
%
%$Q=[Q_{rs}]_{r,s=1}^{N}$, and further $Q_{rs}$ is  a $m_r\times m_s$ block matrix whoose blocks of dimension $2\beta_r\times 2\beta_s$ are of the form 
%$[Z_{jk}]_{j,k=1}^{j=m_r,k=m_s}$ with
%, $\mathcal{J}_{q}(A)$, and using (\ref{MQQN}), the equation (\ref{HQQHl}) (or \ref{MQQNeH}) then splits into
%\[
%J_{\alpha_{\mu}}(\lambda,1)\overline{Y}_{\mu\nu}=Y_{\mu\nu}J_{\alpha_{\nu}}(\lambda,1), %\qquad Y_{\mu\nu}=P_{\alpha_{\mu}}s_{\widehat{\epsilon},\mu}Q_{\mu\nu} %s_{\widetilde{\epsilon},\nu}^{-1}P_{\alpha_{\nu}}^{-1}.
%\]
%
%$P_{\beta_r}'^{-1}W_{\beta_r}'^{-1}T_{\beta_r}Y_{jk}T_{\beta_s}^{-1}W_{\beta_s}'P_{\beta_s}'$ with $Y_{jk}$ 
of the form (\ref{QTC})
for $m=\beta_r$ and 
$T_{1}$, $T_{2}$ 
of the form 
\begin{equation}
\left\{
\begin{array}{ll}
[0\quad T], & \beta_{r}<\beta_{s}\\
\begin{bmatrix}
T\\
0
\end{bmatrix}, & \beta_{r}>\beta_{s}\\
T,& \beta_{r}=\beta_{s}
\end{array}\right.,
\end{equation}
and $T\in \mathbb{C}^{m\times m}$, $m=\min\{\alpha_r, \alpha_s\}$ is any complex upper-diagonal Toeplitz matrix.

\quad
%Denote further , $S_r=\oplus_{j=1}^{m_r} T_{\beta_r}$, $Y=[Y_{jk}]_{j,k}$, thus 
%\[
%Q_{rs}=P_r^{-1}V_r^{-1}S_rY(S_s)^{-1}V_sP_s.
%\]
%where $Y=[Y_{\mu\nu}]_{\mu\nu}$ with $Y_{\mu\nu}$ is of the form (\ref{bT}), is then the candidate to solve (\ref{HlQQHl}).
Recall that $P=P^T$, $P^{-1}=\overline{P}$. 
Since $(P_{r}')^{2}=-(\overline{P}_{_r}')^2=-(P_{\beta_r}')^{-2}=i\big(\oplus_{j=1}^{m_r}(E_{\beta_r}\oplus E_{\beta_r})\big)$, we have $P^{2}=-\overline{P}^2=-P^{-2}=iE$, where $E=\oplus_{r=1}^{N}\left(\oplus_{j=1}^{m_r} (E_{\beta_r}\oplus E_{\beta_r}) \right)$. 
Thus it follows that $I=Q^TQ$ if and only if 
\begin{align}\label{eqSTIS}
I=   & (P^{T}V^{T}S^{-T}Y^{T}S^{T}V^{-T}P^{-T})(P^{-1}V^{-1}S Y S^{-1}VP)\\
 I  = & S^{-1}VPP'^{T}V^{T}S^{-T}Y^{T}S^{T}V^{-T}P^{-2}V^{-1}S Y S^{-1}VPP^{-1}V^{-1}S\nonumber\\
 I =   & S^{-1}V(iE)V^{T}S^{-T}Y^{T}S^{T}V^{-T}(-iE)V^{-1}S Y\nonumber\\
 %I = & S_{\widetilde{\epsilon}}^2\overline{P}^2Y^T(P^{-1})^2S_{\widehat{\epsilon}}^{2}Y\nonumber\\
 I = & S^{-1}VEV^{T}S^{-T}Y^{T}S^{T}V^{-T}EV^{-1}S Y\nonumber\\
 S^{T}V^{-T}EV^{-1}S = & Y^{T}S^{T}V^{-T}EV^{-1}S Y\nonumber\\
 ES^{T}V^{-T}EV^{-1}S = & (EY^{T}E)(ES^{T}V^{-T}EV^{-1}S) Y\nonumber\\
 \mathcal{I} = & (EY^{T}E)\mathcal{I} Y\nonumber,
\end{align}
where we denoted $\mathcal{I}=ES^{T}V^{-T}EV^{-1}S$.

\quad
Observe that
\[
V^{-T}EV^{-1}=\bigoplus_{k=1}^{N}\left( i^{\beta_k}\bigoplus_{j=1}^{m_k} \big( (-1)^{\beta_k}E_{\beta_k}\oplus -E_{\beta_k}\big)\right).
\]
Next, since $(J_{\beta_k}(-i\mu))^{T}E_{\beta_k}=E_{\beta_k}J_{\beta_k}(-i\mu)$, 
%the block in the first row and in the first column of the block matrix $S^{T}_k\mathcal{I}_kS_k$ is 
then using the calculation
\begin{align*}
%-i^{\beta_k}
(\overline{U_{\beta_k}(i\mu))}^{T}(J_{\beta_k}(-i\mu))^{T}E_{\beta_k}J_{\beta_k}(-i\mu)\overline{U_{\beta_k}(i\mu) }
&=
%-i^{\beta_k}
(\overline{U_{\beta_k}(i\mu))}^{T}E_{\beta_k}(J_{\beta_k}(-i\mu))^{2}\overline{U_{\beta_k}(i\mu) }=\\
&=
%-i^{\beta_k}
(\overline{U_{\beta_k}(i\mu))}^{T}E_{\beta_k}\overline{U_{\beta_k}(i\mu) }J_{\beta_k}(-\mu^{2}),
\end{align*}
we obtain that
\small
\begin{equation}\label{IU}
\mathcal{I}=\bigoplus_{k=1}^{N}\left( i^{\beta_k}\bigoplus_{j=1}^{m_k} 
\begin{bsmallmatrix}
-E_{\beta_k}(\overline{U_{\beta_k}(i\mu))}^{T}E_{\beta_k}\overline{U_{\beta_k}(i\mu) }J_{\beta_k}(-\mu^{2}) & 0 \\
0  & (-1)^{\beta_k}E_{\beta_k}(U_{\beta_k}(i\mu))^{T}E_{\beta_k}U_{\beta_k}(i\mu)
\end{bsmallmatrix}\right)
\end{equation}
\normalsize
is a quasi-diagonal matrix. Since 
\begin{align*}
 &E_{\beta_k}U_{\beta_k}(i\mu)^{T}E_{\beta_k}U_{\beta_k}(i\mu)J_{\beta_k}(-\mu^{2})=E_{\beta_k}U_{\beta_k}(i\mu)^{T}E_{\beta_k}(J_{\beta_k}(i\mu))^{2}U_{\beta_k}(i\mu)\\
=&E_{\beta_k}U_{\beta_k}(i\mu)^{T}((J_{\beta_k}(i\mu))^{2})^{T}E_{\beta_k}U_{\beta_k}(i\mu)=E_{\beta_k}(U_{\beta_k}(i\mu)J_{\beta_k}(-\mu^{2}))^{T}E_{\beta_k}U_{\beta_k}(i\mu)\\
=&J_{\beta_k}(-\mu^{2})E_{\beta_k}(U_{\beta_k}(i\mu))^{T}E_{\beta_k}U_{\beta_k}(i\mu),
\end{align*}
if follows that $E_{\beta_k}U_{\beta_k}(i\mu)^{T}E_{\beta_k}U_{\beta_k}(i\mu)$ are upper-triangular matrices. Furthermore, $U_{\beta_k}$ can be chosen so that
the odd rows have real entries and even rows have purely imaginary or zero entries, e.g. we take real eigenvector and then recursively solve equations $((J_{\beta_k}(i\mu))^{2}+\mu^{2})v_{n+1}=v_{n}$ with $v_{n}=E_{\beta_k}^{(n)}$, $n\in \{1,\ldots,\beta_k-1\}$. Then all the nonvanishing entries of $E_{\beta_k}U_{\beta_k}(i\mu)^{T}E_{\beta_k}U_{\beta_k}(i\mu)$ would be purely imaginary for $\beta_k$ even and real for $\beta_k$ odd.
%of the form 
%$\begin{smallbmatrix}
%1 & 0
%\end{smallbmatrix}$ so that 
%since $U_{\beta_k}(i\mu)$ can be chosen up to multiplication by a constant, we may assume that 
%can be multiplied by any suiappropriately (multiplied by   (if the entries on the main diagonals are $i(-1)^{\beta_k}\frac{\mu^{2}}{(2\mu)^{\beta_k-1}}$, $i(-1)^{\beta_k}\frac{1}{(2\mu)^{\beta_k-1}}$, then the other entries are purely imaginary as well). 
%Set
Thus
\[
E_{\beta_k}U_{\beta_k}(i\mu)^{T}E_{\beta_k}U_{\beta_k}(i\mu) = i^{\beta_k+1}T(u_0^{k},u_1^{k},\ldots,u_{\beta_k-1}^{k}),\qquad u_0^{k},u_1^{k},\ldots,u_{\beta_k-1}^{k}\in \mathbb{R},
\]
%and hence
%\begin{align*}
%E_{\beta_k}(\overline{U_{\beta_k}(i\mu))}^{T}E_{\beta_k}\overline{U_{\beta_k}(i\mu) }J_{\beta_k}(-\mu^{2}) & =(-i)^{\beta_k+1}T(u_0^{k},u_1^{k},\ldots,u_{\beta_k-1}^{k})J_{\beta_k}(-\mu^{2})\\
%&=(-\mu^{2}\overline{T}(u_0^{k},u_1^{k},\ldots,u_{\beta_k-1}^{k})+\overline{T}(0,u_0^{k},u_1^{k},\ldots,u_{\beta_k-2}^{k}))\\
% & =(-i)^{\beta_k+1}\overline{T}(-\mu^{2}u_0^{k},-\mu^{2}u_1^{k}+u_0^{k},\ldots,-\mu^{2}u_{\beta_k-1}^{k}+a_{\beta_k-2}^{k}).
%\end{align*}
Finally, from (\ref{IU}) we deduce
\begin{equation*}
\mathcal{I}=i\bigoplus_{k=1}^{N}\left(\bigoplus_{j=1}^{m_k} 
\begin{bsmallmatrix}
T(u_0^{k},\ldots,u_{\beta_k-1}^{k})J_{\beta_k}(-\mu^{2}) & 0 \\
0  & T(u_0^{k},\ldots,u_{\beta_k-1}^{k})
\end{bsmallmatrix}\right), \qquad  u_0^{k},\ldots,u_{\beta_k-1}^{k}\in \mathbb{R}.
\end{equation*}

\quad
Proceed by $T$-conjugating $Y$ with a suitable permutation matrix to get a block matrix such that its blocks are block Toeplitz matrices.
Let $e_1,e_2,\ldots,e_{\alpha_r m_r}$ be the standard orthonormal basis in $\mathbb{C}^{rm_r}$.
We set a matrix formed by these vectors: 
\begin{equation}\label{perK}
\Omega=\oplus_{r=1}^{N}\Omega_r,
\end{equation}
\small
\begin{align*}
\Omega_r=[
& e_1\;e_{2\beta_r+1}\;\ldots\;e_{2(m_r-1)\beta_r+1}\;
e_{\beta_r+1}\;e_{3\beta_r+1}\;\ldots\;e_{(2m_r-1)\beta_r+1}\;
e_2\;e_{2\beta_r+2}\;\ldots\;e_{2(m_r-1)\beta_r+2}\;\ldots\;\\
& \qquad e_{\beta_r+2}\;e_{3\beta_r+2}\;\ldots\;e_{(2m_r-1)\beta_r+2}\;\ldots\;
e_{\beta_r}\;e_{3\beta_r}\;\ldots\;e_{\beta_r (2m_r-1)}\;
e_{2\beta_r}\;e_{4\beta_r}\;\ldots\;e_{\beta_r (2m_r)}
]
\end{align*}
\normalsize
Observe that multiplicating with $\Omega_r$ from the right puts the first, $(2\alpha_r+1)$-th,\ldots, $((2m_r-1)\alpha_r+1)$-th column together, further 
$e_{\alpha_r+1}$-th,\ldots $e_{(2m_r-1)\alpha_r+1}$-th column together, then
the second,\ldots,$(2(m_r-1)\alpha_r+2)$-th column together, and soforth. Similarly, by multiplicating with $\Omega_r^{T}$ from the left we collect the rows together.

\quad
From (\ref{eqSTIS}) it follows that
\begin{align}\label{ortoD2}
 %\Omega^{T}S^{T}\mathcal{I}S\Omega = & \Omega^{T} Y^{T}\Omega\Omega^{T}S^{T}\mathcal{I}S \Omega\Omega^{T} Y\Omega\\
 \Omega^{T} \mathcal{I}\Omega = & (\Omega^{T}E\Omega)(\Omega^{T} Y\Omega)^{T}(\Omega^{T}E\Omega)(\Omega^{T}\mathcal{I} \Omega)(\Omega^{T} Y\Omega). 
 %\qquad I= \mathcal{I}_{\epsilon}Z \mathcal{I}_{\epsilon} \widetilde{Z}^{T}.
\end{align}
%
%
%More precisely,
%\begin{align*}
%\Omega^{T}S^{T}\mathcal{I}S\Omega=& (\oplus_{r=1}^{N}\Omega_r^{T})(\oplus_{r=1}^{N}\oplus_{j=1}^{m_r}\epsilon_{r,j}I_{\alpha_r})(\oplus_{r=1}^{N}\Omega_r)\\
%                   =&  \oplus_{r=1}^{N}\Omega_r^{T}(\oplus_{j=1}^{m_r}\epsilon_{r,j}I_{\alpha_r})\Omega_r
%                   =\oplus_{r=1}^{N} \oplus_{k=1}^{\alpha_r}\oplus_{j=1}^{m_r}\epsilon_{r,j}\\
%\Omega^{T}Y\Omega=  &  (\oplus_{r=1}^{N}\Omega_r^{T})[Y_{rs}]_{r,s=1}^{N}(\oplus_{r=1}^{N}\Omega_r)=[\Omega_r^{T}Y_{rs}\Omega_s]_{r,s=1}^{N},%\\
%\end{align*}
%
%The blocks $Z_{rs}$ are again block matrices whoose blocks are of the form (\ref{bT2}) and (\ref{bTw2}), respectively. Thus $\Omega_r^{T}Z_{rs}\Omega_s$  are %$\alpha_r\times \alpha_s$  
%upper block Toeplitz matrices of the form (\ref{bT2}) and of the form (\ref{bTw2}), but with $T_{\mu\nu}$ or $\widetilde{T}_{\mu\nu}$ replaced by an block Toeplitz %matrices of the form
%
%\small
%\begin{align*}
%&T_c(A_0,A_1,\ldots,A_{m-1})=\begin{bmatrix}
%  \mathcal{A}_{0} & A_{1}         & A_{2}             & \ldots   & \ldots &    A_{m-1}  \\
%0       & A_0 & A_{1} & A_2  &      & \vdots \\
% \vdots & \ddots            & A_0               & A_1  & \ddots &   \vdots \\ 
% \vdots &  & \ddots   & \ddots             & \ddots &  \vdots\\
% \vdots &        &              & \ddots  & \ddots &   \vdots \\
%0       & \ldots            & \ldots & \ldots &  0      & \ddots
%\end{bmatrix}, \quad A_0,\ldots,A_{m-1}\in \mathbb{C}^{m_r\times m_s},\\
%\end{align*}
%\normalsize
%
Set $\mathcal{Y}=\Omega^{T} Y\Omega$ and observe that its blocks $\mathcal{Y}_{rs}$ are of the form (\ref{0T0}) with 
%
%\begin{equation}\label{bT3}
%[\mathcal{Y}]_{rs}=[\Omega^{T}Y\Omega]_{rs}=\left\{
%\begin{array}{ll}
%[0\quad T_{rs}], & \beta_r<\beta_s\\
%\begin{bmatrix}
%T_{rs}\\
%0
%\end{bmatrix}, & \beta_r<\beta_s\\
%T_{rs},& \beta_r=\beta_s
%\end{array}\right.,
%\end{equation}
$\mathcal{T}_{rs}=T(A_0^{rs},A_1^{rs},\ldots,A_{b_{rs}-1}^{rs})$, $b_{rs}=\min\{\beta_r, \beta_s\}$, where further
\begin{align*}%\label{eqCHE}
A_0^{rs}=
\begin{bmatrix}
V_0^{rs} &  W_0^{rs} \\
-\mu^{2}\overline{W}_0^{rs} &  \overline{V}_0^{rs} 
\end{bmatrix}, \qquad
&A_n^{rs}=
\begin{bmatrix}
V_n^{rs} &  W_n^{rs} \\
-\mu^{2}\overline{V}_n^{rs}+\overline{V}_{n-1}^{rs} &  \overline{W}_n^{rs} 
\end{bmatrix}, \quad n\in \{1,\ldots,b_{rs}-1\}\nonumber
\end{align*}
with $V_0^{rs},W_0^{rs},\ldots,V_{b_{rs}-1}^{rs},W_{b_{rs}-1}^{rs}\in \mathbb{C}^{m_r\times m_r}.$ 
We denote $\mathcal{F}=\Omega^{T} E\Omega$ and 
\begin{align*}
\mathcal{B}=-i\Omega^{T}\mathcal{I}\Omega=\bigoplus_{r=1}^{N}T\big(B_0^{r},B_1^{r},\ldots,B_{\alpha_r-1}^{r}\big), 
%which is a quasi-diagonal matrix with upper-triangular block Toeplitz %diagonal blocks. We have $
%=\bigoplus_{r=1}^{N}T\big(B_0^{r},B_1^{r},\ldots,B_{\alpha_r-1}^{r}\big)$, where for every $r\in \{1,\ldots,N\}$:
%
%\begin{align*}
%&B_0^{r}=
%(\mu^{2}\overline{u}_0^{k}I_{m_k}   \oplus (-1)^{\beta_k}u_0^{k}I_{m_k})=
%\tfrac{i(-1)^{\beta_k}}{(2\mu)^{\beta_k-1}}
%u_0^{r}\big(-\mu^{2}I_{m_r}\oplus I_{m_r} \big),\\
%\qquad u_0\in \mathbb{R}\\
%&B_n^{r}=((-\mu^{2}\overline{u}_n^{r}+\overline{u}_{k-1}^{r})I_{m_r}   %\oplus u_k^{r}I_{m_r}),\quad k\in \{1,\ldots, b_{rs}-1\}, \quad %u_0^{r},\ldots,u_{b_{rs}-1}^{r} \in \mathbb{R}.
%&B_2^{r}=((-\mu^{2}\overline{u}_2^{r}+\overline{u}_1^{r})I_{m_r}   %\oplus u_2^{r}I_{m_r}),\\
\end{align*}
with $B_0^{r},\ldots,B_{b_{rs}-1}^{r}$, $r\in \{1,\ldots,N\}$ as in (\ref{aass}).
By applying Lemma \ref{EqT} (\ref{EqT1b}) to equation $\mathcal{B}=\mathcal{F}\mathcal{Y}^{T}\mathcal{F}\mathcal{B}\mathcal{Y}$ (obtained from (\ref{ortoD2})) we conclude the proof of (\ref{stabz2}).

\item \label{caseL} 

Let
\[
\mathcal{H}^{\epsilon}(A)
%\bigoplus_{r=1}^{N} \mathcal{H}_{\alpha_r}(\xi)
=\bigoplus_{r=1}^{N}\left( \bigoplus_{j=1}^{m_r}  L_{\alpha_r}(\xi)\right), \qquad \xi^{2}\in \mathbb{C}\setminus \mathbb{R},
\]
where $H_{\alpha_r}(\xi)$, $r\in \{1,\ldots,N\}$ is as in (\ref{HKLmz}) for $z=\xi$, $m=\alpha_r$.
%a direct sum of all blocks of size $2\alpha_r\times 2\alpha_r $ %corresponding to the eigenvalue $\xi$. 
%This case is treated similarly as \ref{caseH} for $\lambda>0$, but this time with complex Toeplitz matrices.

\quad
By Lemma \ref{posl} the solution $Q$ of the equation $\mathcal{H}^{\epsilon}(A)\overline{Q}=Q\mathcal{H}^{\epsilon}(A)$ 
%is partitioned to blocks conformally to $\mathcal{H}^{1}(A)$, 
is of the form 
\[
Q=P^{-1}YP,\qquad
P=\oplus_{r=1}^{N}P_{r}', \quad 
P_r'=\oplus_{j=1}^{m_r}P_{\beta_k}\oplus P_{\beta_k}, \quad
P_{\beta_k}=\tfrac{e^{-i\frac{\pi}{4}}}{\sqrt{2}}(I_{\beta_k}+iE_{\beta_k}),
\]
where $Y=[Y_{rs}]_{r,s=1}^{N}$ and further $Y_{rs}$ is a $m_r\times m_s$ block matrix whoose blocks of dimension $2\beta_r\times 2\beta_s$ are of the form 
%partitioned conformally to blocks as $\mathcal{H}_{}^1(A)$ and the block $Y_{jk}$ are of the form 
%(\ref{QTC}) with $T_1$ and $T_2$ possibly two different matrices of the form (\ref{QT}) for a complex upper-triangular Toeplitz matrix $T$.
%
%$Q=[Q_{rs}]_{r,s=1}^{N}$, and further $Q_{rs}$ is  a $m_r\times m_s$ block matrix whoose blocks of dimension $2\beta_r\times 2\beta_s$ are of the form 
%$[Z_{jk}]_{j,k=1}^{j=m_r,k=m_s}$ with
%, $\mathcal{J}_{q}(A)$, and using (\ref{MQQN}), the equation (\ref{HQQHl}) (or \ref{MQQNeH}) then splits into
%\[
%J_{\alpha_{\mu}}(\lambda,1)\overline{Y}_{\mu\nu}=Y_{\mu\nu}J_{\alpha_{\nu}}(\lambda,1), %\qquad Y_{\mu\nu}=P_{\alpha_{\mu}}s_{\widehat{\epsilon},\mu}Q_{\mu\nu} %s_{\widetilde{\epsilon},\nu}^{-1}P_{\alpha_{\nu}}^{-1}.
%\]
%
%$P_{\beta_r}'^{-1}W_{\beta_r}'^{-1}T_{\beta_r}Y_{jk}T_{\beta_s}^{-1}W_{\beta_s}'P_{\beta_s}'$ with $Y_{jk}$
% 
(\ref{QTC})
for $m=\beta_r$, $n=\beta_s$, thus 
of the form 
\begin{equation}
\left\{
\begin{array}{ll}
[0\quad T], & \beta_{r}<\beta_{s}\\
\begin{bmatrix}
T\\
0
\end{bmatrix}, & \beta_{r}>\beta_{s}\\
T,& \beta_{r}=\beta_{s}
\end{array}\right., \qquad T=T_1\oplus \overline{T}_1
\end{equation}
for a complex upper-triangular Toeplitz matrix $T_1\in \mathbb{C}^{p\times p}$, $p=\min\{\beta_r, \beta_s\}$.

\quad 
Similarly, as (\ref{QTQH}), we now deduce that $I=Q^TQ$ if and only if 
\begin{align*}
I=   & P^TY^T(P^{-1})^TP^{-1}Y P \\
%I  = & PS_{\widetilde{\epsilon}}(S_{\widetilde{\epsilon}}^T %P^TY^T(P^{-1})^T(S_{\widehat{\epsilon}}^{-1})^T)(S_{\widehat{\%epsilon}}^{-1}P^{-1}Y P %S_{\widetilde{\epsilon}})S_{\widetilde{\epsilon}}^{-1}P^{-1},\%\
% I =   & P S_{\widetilde{\epsilon}}^2 P^TY^T %P^{-1}S_{\widehat{\epsilon}}^{2}P^{-1}Y\\
 %I = & %S_{\widetilde{\epsilon}}^2\overline{P}^2Y^T(P^{-1})^2S_{\wide%hat{\epsilon}}^{2}Y\\
 %I = & %S_{\widetilde{\epsilon}}^2\overline{P}^2Y^TP^{-2}S_{\widehat{\epsilon}}^{2}Y\\
 % S_{\widetilde{\epsilon}}^2 = & \big(P^2 Y P^{2}\big)^{T}S_{\widehat{\epsilon}}^{2}Y\\
    I = & E Y^{T} EY,
\end{align*}
where $E=\oplus_{r=1}^{N}\left(\oplus_{j=1}^{m_r} (E_{\beta_r}\oplus E_{\beta_r}) \right)$. (We have $P^{2}=\overline{P}^2=P^{-2}=E$.)

\quad
Using a permutation matrix $\Omega$ from (\ref{perK}) we write
\begin{align*}
 %\Omega^{T}S^{T}\mathcal{I}S\Omega = & \Omega^{T} Y^{T}\Omega\Omega^{T}S^{T}\mathcal{I}S \Omega\Omega^{T} Y\Omega\\
 I = & (\Omega^{T}E\Omega)(\Omega^{T} Y\Omega)^{T}(\Omega^{T}E\Omega)(\Omega^{T} Y\Omega). 
 %\qquad I= \mathcal{I}_{\epsilon}Z \mathcal{I}_{\epsilon} \widetilde{Z}^{T}.
\end{align*}
We have $\mathcal{Y}=(\mathcal{Y}_{rs})_{rs}=\Omega^{T} Y\Omega$ with $\mathcal{Y}_{rs}$ of the form (\ref{0T0}) for   
%
%\begin{equation}\label{bT3}
%[\mathcal{Y}]_{rs}=[\Omega^{T}Y\Omega]_{rs}=\left\{
%\begin{array}{ll}
%[0\quad T_{rs}], & \beta_r<\beta_s\\
%\begin{bmatrix}
%T_{rs}\\
%0
%\end{bmatrix}, & \beta_r<\beta_s\\
%T_{rs},& \beta_r=\beta_s
%\end{array}\right.
%\end{equation}
%
$\mathcal{T}_{rs}=T(A_0^{rs},\ldots,A_{b_{rs}-1}^{rs})$, $b_{rs}=\min\{\beta_r, \beta_s\}$, where $A_n^{rs}=V_n^{rs}\oplus \overline{V}_n^{rs}$, $V_n^{rs}\in \mathbb{C}^{m_r\times m_r}$ for all $n\in \{0,\ldots,b_{rs}-1\}$.

\quad
Taking the submatrix $\mathcal{V}$ (or $\overline{\mathcal{V}}$) formed by entries contained in rows and columns (of $\mathcal{Y}$) which contain any $V_n^{rs}$ (or $\overline{V}_n^{rs}$) yields the equation $I=\widetilde{\mathcal{V}}\mathcal{V}$ ($I=(\overline{\widetilde{\mathcal{V}}}\overline{\mathcal{V}}$) with 
$(\mathcal{V})_{rs}=T(V_0^{rs},\ldots,V_{b_{rs}-1}^{rs})$. 
It is clear that $\widetilde{\mathcal{Y}}\mathcal{Y}=I$ holds precisely when 
$\widetilde{\mathcal{V}}\mathcal{V}=I$ holds.
%This equation is clearly equivalent to the equation obtained above, and 
We solve the later equation by applying Lemma \ref{EqT} (\ref{EqT1}).
\end{enumerate}

This concludes the proof of the theorem.
\end{proof}

\begin{proof}[Proof of Theorem \ref{posapp}]
Let $GL_n(\mathbb{C})$ denote the group of all non-singular $n\times n$ matrices. Any holomorphic change of coordinates that
preserves the general form of (\ref{BasForm1}) has the same effect on
the quadratic part of (\ref{BasForm1}) as a complex-linear change of the form
\[
\begin{bmatrix}
z \\
w
\end{bmatrix}
=
\begin{bmatrix}
P & b \\
0 & c
\end{bmatrix}
\begin{bmatrix}
\widetilde{z} \\
\widetilde{w}
\end{bmatrix}, \quad P\in GL_n(\mathbb{C}),\, b\in \C^n, \,c\in \mathbb{C}\setminus \{0\}
.\]
Using this linear changes of coordinates, the form (\ref{BasForm1})
transforms into
%
%\[
%\widetilde{w}=\overline{\widetilde{z}}^T\left(\frac{1}{c}P^{*}AP\right)\widetilde{z}+
%\frac{1}{2}{\widetilde{z}}^T\left(\frac{1}{c}P^TBP\right)\widetilde{z}+\frac{1}{2}{\overline{\widetilde{z}}}^T
%\overline{\left(\frac{1}{\overline c}P^T B P\right)}
%\overline{\widetilde{z}} +\frac{1}{c} b^T \widetilde{w}+o(|\widetilde{z}|^2)
%\]
%
%\noindent or
%
%\[
%\widetilde{w}-\frac{1}{c} b^T
%\widetilde{w}-\frac{1}{2}{\widetilde{z}}^T\left(\frac{\overline c-c}{|c|^2}P^TBP\right)\widetilde{z}=
%\overline{\widetilde{z}}^T\left(\frac{1}{c}P^{*}AP\right)\widetilde{z}+
%\Rea\left({\widetilde{z}}^T\left(\frac{1}{\overline c}P^T B P\right)
%\widetilde{z}\right) +o(|\widetilde{z}|^2).
%\]
%
%\noindent If we denote \(\widehat{w}=\widetilde{w}-\frac{1}{c} b^T
%\widetilde{w}-\frac{1}{2}{\widetilde{z}}^T\left(\frac{\overline
%    c-c}{|c|^2}P^TBP\right)\widetilde{z},\) we get the equation
%
\begin{equation}\label{BasForm2}
\widetilde{w}=\overline{\widetilde{z}}^T\left(\frac{1}{c}P^{*}AP\right) \widetilde{z}+\Rea \left(\widetilde{z}^T \left(\frac{1}{\overline{c}} P^TBP\right) \widetilde{z}\right)+o(|\widetilde{z}|^2).
\end{equation}
It is clear that by scaling $P$, we can assume
$|c|=1$ and eliminate $c$ in the second term of (\ref{BasForm2}).
Since $P^{*}AP$ is Hermitian if and only if $A$ is
Hermitian, we have $c\in\{1,-1\}$. Next, the Autonne-Takagi theorem yields that a symmetric nonsingular matrix is $T$-congruent to the identity matrix, so we may assume $B=I$. To preserve $B=I$ the matrix $P$ must be orthogonal.
%$B$ $(A,B)\sim (\widetilde{A},I)$.

Since $-K_m(z)$ and $-L_m(z)$ are orthogonally $*$-congruent to $K_m(z)$ and $L_m(z)$, respectively (see e.g. \cite[Lemma 2.6]{Hong89}), matrices $-\mathcal{H}^{\epsilon}(A)$ and $\mathcal{H}^{-\epsilon}(A)$ are orthogonally $*$-congruent, too. Further, recall that $Q^*H_{2n-1}(0)Q=-H_{2n-1}(0)$, $n\in \mathbb{N}$, where $Q=-1\oplus 1\oplus -1 \oplus \ldots  \oplus 1\oplus  -1$. To conclude the proof we now use Theorem \ref{stabz}.
(The last statement of the corollary follows trivially.)
\end{proof}

\section{Open questions}

We point out a few interesting questions whoose answers would give a further understanding of the stratfication of certan classes of matrices with respect to some important matrix group actions.
We begin with a couple of questions mentioned in the first section. We plan to answer the first one in our future researh. 

\begin{question}
What are the dimensions of orbits of the actions of orthogonal similarity on skew-symmetric and on orthogonal matrices?
\end{question}

\begin{question} 
What is the dimension of the subset of orthogonal matrices of the form $\begin{bsmallmatrix}
A & B\\
\overline{B} & \overline{A}
\end{bsmallmatrix}$, $A,B\in \mathbb{C}^{n\times n}$?
\end{question}

This is the only (technical) problem left open concerning the dimension of the stabilizer of $*$-conjugation on Hermitian matrices (see the Remark), though \ref{rmz} Theorem \ref{stabz} is a strong result already in the present form.

The next step would be to determine the inclusion relationship between the closures of the orbits. This has been so far inspected for actions of the group of all invertible matrices (not neccessarily orthogonal) on all matrices. 
The case of similarity was first studied by Arnold (see e.g. \cite{Arnold}) and then through the works of Markus and Parilis \cite{MarkusParilis} and Edelman, Elmroth and K\aa gstrom \cite{EEK}, among others.
Next, $*$-conjugation or $T$-conjugation seem to be much more involved, even in lower dimensions; see the papers  Futorny, Klimenko and Sergeichuk \cite{FKS1} and Dmytryshyn, Futorny, K\aa gstr\"{o}m, Klimenko and Sergeichuk \cite{FKS2}).

\begin{question}
What is the relationship beetwen the closures of orbits with respect to actions of orthogonal similarity (*-conjugacy) on symmetric, skew-symmetric  or orthogonal (Hermitian, skew-Hermitian) matrices.
\end{question}

The last question is naturally related to the topic.

\begin{question}
What are the dimensions of orbits and the relationship beetwen them under the action of unitary similarity ($*$-conjugacy ot $T$-conjugacy) on symmetric, skew-symmetric, Hermitian, skew-Hermitian or unitary matrices. 
%What is the relationship beetwen the closures of their orbits.
\end{question}

%
%
%
%

%
%
%
%

%\smallskip
%\textit{Acknowledgement.}

%This research was supported by Slovenian Research Agency (grant no. P1-0291).

%\bibliographystyle{amsplain}

\end{document}